\documentclass[reqno]{amsart}

\usepackage{graphicx,subcaption}
\usepackage[utf8]{inputenc}
\usepackage[colorlinks=true, linkcolor=, citecolor=blue]{hyperref}
\usepackage[backend=bibtex]{biblatex}
\usepackage{amssymb}
\usepackage{amsmath}
\usepackage{graphicx}
\usepackage{amsthm,amsmath}
\usepackage{amsfonts}
\usepackage{graphicx}
\DeclareUnicodeCharacter{2124}{ }
\usepackage{enumitem}
\usepackage{bm}
\numberwithin{equation}{section}
\newtheorem{theo}{Theorem}[section]
\newtheorem{theorem}{Theorem}[section]

\newtheorem{assumpintro}{Assumption}[section]

\newtheorem{lemma}[theo]{Lemma}

\newtheorem{prop}[theo]{Proposition}
\newtheorem{proposition}[theo]{Proposition}

\theoremstyle{definition}

\newtheorem{definition}[theo]{Definition}
\theoremstyle{remark}

\newtheorem{remark}[theo]{Remark}

\usepackage{lmodern}
\usepackage{caption}

\usepackage{color}
\usepackage{dsfont}

\newcommand\tempskipped[1]{}

\newcommand\C{\mathbb{C}}

\newcommand\R{\mathbb{R}}

\newcommand\E{\mathbb{E}}

\newcommand\cst{\operatorname{cst}}

\renewcommand\Re{\operatorname{Re}}
\renewcommand\Im{\operatorname{Im}}

\newcommand\osc{\operatorname{osc}}

\newcommand\cE{\mathcal{E}}
\newcommand\cQ{\mathcal{Q}}
\newcommand\cS{\mathcal{S}}

\newcommand\cX{\mathcal{X}}
\newcommand\cT{\mathcal{T}}
\newcommand\cY{\mathcal{Y}}
\newcommand\cM{\mathcal{M}}
\newcommand\cA{\mathcal{A}}

\newcommand\dm{\diamond}
\newcommand\Dm{\diamondsuit}

\def\LipKd{{\mbox{\textsc{Lip(}$\kappa$\textsc{,}$\delta$\textsc{)}}}}
\def\ExpFat{{{\mbox{\textsc{Exp-Fat(}$\delta$\textsc{,}$\rho$\textsc{)}}}}}
\def\ExpFatt{{{\mbox{\textsc{Exp-Fat(}$\delta$\textsc{,}$\rho(\delta)$\textsc{)}}}}}

\newcommand\Unif{{\mbox{\textsc{Unif(}$\delta$\textsc{)}}}}
\newcommand\Uniff{{\mbox{\textsc{Unif(}$\delta,r_0,\theta_0$\textsc{)}}}}
\newcommand\Unifff{{\mbox{\textsc{Unif(}$\frac{1}{n},10,\frac{\pi}{10}$\textsc{)}}}}
\newcommand\Uniffff{{\mbox{\textsc{Unif(}$\delta,2r_0,\frac{\theta_0}{2}$\textsc{)}}}}

\newcommand\vcirc[1]{v^\circ_1,\ldots,v^\circ_{#1}}
\newcommand\svcirc[1]{\sigma_{v^\circ_1}\ldots \sigma_{v^\circ_{#1}}}
\newcommand\vbullet[1]{v^\bullet_1,\ldots,v^\bullet_{#1}}
\newcommand\muvbullet[1]{\mu_{v^\bullet_1}\ldots\mu_{v^\bullet_{#1}}}

\bibliography{Completebibliography.bib}

\begin{document}

\title{The near critical random bond Ising model via embedding deformation}

\author[Rémy Mahfouf]{Rémy Mahfouf$^\mathrm{A}$}

\thanks{\textsc{${}^\mathrm{A}$ Université de Genève.}}

\thanks{\emph{E-mail:} \texttt{remy.mahfouf@unige.ch}}

\maketitle

\begin{abstract}
Using the formalism of differential equations, we introduce a new method to continuously deform the $s$-embeddings associated with a family of Ising models as their coupling constants vary. This provides a geometric interpretation of the critical scaling window $\asymp n^{-1}$ for the model on the $n \times n$ box. We then drive this deterministic deformation process by i.i.d.\ Brownian motions on each edge, centered at the critical model, thereby generating random $s$-embeddings as solutions to stochastic differential equations attached to near-critical random bond Ising models. In this setting, with high probability with respect to the random environment, the Ising model remains conformally invariant in the scaling limit, even when the standard deviation of the random variables (up to logarithmic corrections) is $n^{-\frac{1}{3}} \gg n^{-1}$, far exceeding the deterministic critical window. We also construct an Ising model with slightly correlated (in space) random coupling constants, whose critical window is $ \asymp \log(n)^{-1}$ on the $n \times n$ box. Our method, which can also be applied to the dimer context, naturally extends to a much broader class of graphs and opens a new approach to understanding the critical Ising model in random environments.
\end{abstract}

\section{Introduction}\label{sec:introduction}
\subsection{General context} The Ising model, introduced nearly a century ago by Ising and Lenz \cite{ising1925beitrag}, remains one of the most studied models in probability and statistical mechanics. This article focuses on its planar version with nearest-neighbour interactions and no external magnetic field—a case extensively explored by both physicists and mathematicians due to its integrable structure, which allows for explicit computations of local and global observables (see e.g.\ the monographs \cite{friedli-velenik-book,mccoy-wu-book,palmer2007planar} and references therein). We adopt a dual convention to the standard setup, assigning $\pm 1$ spins to the set $G^\circ$ of faces of a planar graph $G$. When $G$ is finite and connected, each edge $e \in E(G)$—which separates two faces $v^\circ_{\pm}(e) \in G^\circ$—is assigned a positive coupling constant $J_e$. This ferromagnetic model favors configurations in which neighboring spins align. For a fixed inverse temperature $\beta > 0$, one defines a probabilistic model on spin configurations $\sigma \in \{\pm 1\}^{G^\circ}$, with partition function given by\begin{equation}
\label{eq:intro-Zcirc}
\mathcal{Z}(G)\ :=\ \sum_{\sigma:G^\circ\to\{\pm 1\}}\exp\big[\,\beta\sum_{e\in E(G)}J_e\sigma_{v^\circ_-(e)}\sigma_{v^\circ_+(e)}\,\big].
\end{equation}

The above definition is purely combinatorial and was used by Chelkak in \cite{Ch-ICM18,Che20} to propose a practical construction of an embedding associated with a weighted planar Ising graph $(G,x)$, inspired by similar approaches in other statistical mechanics models (e.g.\ Tutte's barycentric embeddings for discrete harmonic functions). These so-called $s$-embeddings are well suited to generalize the study of Ising fermions following Smirnov's breakthrough results on the square lattice \cite{Smi-ICM06,Smirnov_Ising}, and to establish conformal invariance or covariance of the critical model \cite{Smirnov_Ising,CHI,ChSmi2,CheIzy13,HS-energy-Ising,Che20,Izy-phd}, as predicted by Conformal Field Theory (see e.g.\ \cite{Zam,Zam2}). This framework was later extended to near-critical models \cite{park2018massive,park-iso,CIM-universality}, confirming connections to solutions of the massive Dirac equation \cite{mccoy1977painleve,sato1979holonomic}. Remarkably, $s$-embeddings unify and go far beyond previous approaches, providing an additional link between the Ising model on highly irregular graphs to massive fermions in Minkowski space $\mathbb{R}^{(2,1)}$ and generalized solutions of conjugate Beltrami equations (see \cite{Che20,Mah23,MahPar25a,ChePar}). In particular, this framework now allows us to rigorously analyze criticality in the scaling limits of a broad class of Ising models on degenerate and irregular grids, far beyond the realm of symmetries and integrability.

The goal of this paper is to introduce a seemingly naive yet surprisingly effective approach to studying the planar Ising model. Given a weighted planar Ising graph $(G,x)$, an associated $s$-embedding $\cS$ allows to construct a graphical representation of the underlying \emph{weighted model}. This representation (not uniquely defined) is obtained by finding a vector in the kernel of a linear system determined by the Ising weights of $(G,x)$. Two challenges arise when one tries to use the associated formalism. First, solving this linear system is generally non-trivial unless the weights exhibit some integrable or symmetric structure. Second, one may wonder if there exists some continuity of the graphical representation with respect to the weights: it is tempting to think that if two sets of Ising weight $x$ and $\tilde{x}$ are close to each other (in a suitable sense), one should be able to construct corresponding $s$-embeddings which are also. However, this requires very careful treatment, as the kernel of a linear system is in general unstable under small perturbations of its coefficients. 

Fortunately, the $s$-embeddings construction is sufficiently rich to allow some explicit and controlled deformations of the graphical representation as the weights vary using fermions of the associated Ising model, as noticed for the first time in the present paper. In particular, one of the output of the idea developed here is some differential construction of $s$-embeddings when moving continuously the Ising weights, solving at least partially one of the bottlenecks of the theory which is finding some $s$-embedding attached to a given Ising model, and therefore apply all the available discrete complex analysis machinery \cite{Che20,CLR1,CLR2,Mah23,MahPar25a,ChePar,MahPHD}. We illustrate this with two applications of this embedding deformation idea. The first provides a geometric interpretation of the correlation length: in the near-critical regime, a sharp change in crossing probabilities corresponds to a change in the graphical representation. The second application enables a linearisation of the model near criticality, even when random perturbations of the weights have magnitudes highly  exceeding (by a cubic root power) the deterministic critical window. In particular, we show that a system with independently sampled near-supercritical and near-subcritical weights at each edge averages to an exactly critical model. 

In the present paper, we only apply the method in its simplest setting, starting from the critical square lattice and producing embeddings with bounded angles and edge lengths of comparable size. We hope that this work, whose philosophy also applies to the dimer model, introduces enough new ideas to pave the way toward a rigorous proof of conformal invariance for the Ising and dimer models in well suited random environments.

\subsection{Definition of the FK model and the general criticality condition}

In the present work we focus on statements related to the so-called FK representation (introduced by Fortuin and Kasteleyn in \cite{fortuin1972random}) of the nearest neighbour Ising model with partition function given by \eqref{eq:intro-Zcirc}. Start with a weighted planar graph $(G,x)$ embedded in the plane or in the sphere (in the finite case) up to homeomorphism preserving the cyclic ordering of edges. Denote its vertices by $G^{\bullet}$ and its faces by $G^{\circ}$. The bipartite graph $\Lambda(G) := G^{\bullet} \cup G^{\circ}$ has edges connecting each vertex to the faces it belongs to. Each quad $z_e = (v^{\bullet}_0 v^{\circ}_0 v^{\bullet}_1 v^{\circ}_1)$ of $\Lambda(G)$ corresponds to an edge $e$ of $G$, linking $v^{\bullet}_0$ and $v^{\bullet}_1$ and separating the faces $v^{\circ}_0$ and $v^{\circ}_1$. We denote $e^\star$ the dual edge linking $v^{\circ}_0$ and $v^{\circ}_1$ in the graph $G^{\circ}$, and have $(e^\star)^\star =e$. Under this identification, one can parametrise the coupling constant $x(e)$ using the abstract angle as
\begin{equation}
\label{eq:x=tan-theta} \theta_{z(e)}\ :=\ 2\arctan x(e)\ \in\ (0,\tfrac{1}{2}\pi), \quad x(e):=\exp[-2\beta J_e].
\end{equation}
Using the classical Kramers-Wannier duality, set the dual weight 
\begin{equation}\label{eq:Kramers-Wannier}
(x_{e})^\star:= \frac{1-x_e}{1+x_e}.
\end{equation}
When $G$ is a finite planar graph, the model with \emph{wired} boundary conditions can be embedded into the sphere, where a distinguished face $v_{\mathrm{out}}^{\circ}$ represents \emph{all} boundary vertices and carries a single fixed spin. The FK-Ising model on $G^\circ$ can then be interpreted as a probability measure on even subgraphs, such that for any subgraph $C$ of $G^\circ$ such that
\begin{equation}
	\mathbb{P}^{G^\circ}_{FK}(C):=\frac{1}{Z_{FK}(G^{\circ},(x_e)_{e\in G})} 2^{\#\textrm{clusters}(C)}\prod_{e^\star\in C}(x_{e^\star})^\star\prod_{e^\star\not \in C}(1-(x_{e^\star})^\star),
\end{equation}
where $e^\star$ denotes the dual edge $G$ linking the vertices $v^{\pm}_{e^\star} \in G^{\circ}$, $\#\mathrm{clusters}(C)$ is the number of clusters in the subgraph $C$, and $Z_{FK}(G^{\circ}, (x_e)_{e \in G})$ is a normalization constant. It is standard (see e.g.\ \cite{duminil-parafermions}) to pass to the infinite-volume limit, thereby defining a full-plane FK-Ising measure on $G^\circ$. In this paper, we only consider graphs satisfying the strong box-crossing property (recalled below in the context of $s$-embeddings), ensuring that the infinite-volume limit is unique and independent of the initial choice of wired boundary conditions on finite graphs. The (combinatorial) link between the Ising model with wired boundary conditions and the FK-Ising model with wired boundary conditions is known as the Edwards-Sokal coupling introduced in \cite{edwards1988generalization} reads as follows:

\begin{itemize}
	\item Ising model to FK-Ising model: start with a spin configuration $\sigma \in \{\pm 1 \}^{G^{\circ}}$ and sort independently for each pair of aligned neighbouring faces $v_{\pm}^\circ\in G^{\circ}$
	  some Bernoulli random variable of parameter $(x_{e^\star})^\star$. The faces $v_{\pm}^\circ$ are connected in the random cluster model if and only if the associated Bernoulli variable is $1$. This constructs a random graph in $G^{\circ}$. 
	\item FK-Ising model to Ising model: For each cluster $C$ in $G^\circ$, sort (independently from other clusters) some fair $\pm 1 $ random variable and assign as a spin the result to \emph{all} spins attached to $C$. 
\end{itemize}

In the present paper, we always work with $s$-embeddings that satisfy some property called  \Unif\, in \cite{Che20} and recalled below. In words, this means working with an $s$-embedding $\cS$ where all the angles remain bounded away from $0$ and $\pi$ while all the edge-lengths are comparable. In what follows,  $\cS$ is a proper $s$-embedding (see Section \ref{sub:semb-definition} for a precise definitions).
\begin{definition}[Assumption \Unif\,]
	We say that $\cS$ satisfies the assumption $\Unif\,=\Uniff\ $ for some parameters $\delta,r_0,\theta_0$ if all edge-lengths in $\cS$ are comparable to $\delta$, meaning that for any  neighbouring $v^{\bullet}\in G^\bullet$ and $v^{\circ}\in G^\circ$ one has
	\begin{equation}
		r_0^{-1}\cdot \delta \leq  |\cS(v^{\bullet})- \cS(v^{\circ})| \leq r_0\cdot \delta,
	\end{equation}
and all the geometric angles in the quads $\cS$ are bounded from below by $\theta_0$.
\end{definition}
In particular, it is easy to see that in the general formalism introduced in Section \ref{sub:notation} to define in full generality the scale of an $s$-embedding, there exist constants $\kappa<1$ and $C_0$, only depending on $r_0,\theta_0$ such that grids satisfying the assumption \Uniff\, have to satisfy the assumptions \LipKd\ and \ExpFat\ hold for some scale $\rho=C_0\cdot \delta $. 

We are now ready to state a simplified version of the main result of \cite{Mah23}, which states that $s$-embeddings satisfying \Unif\, are critical regarding the so-called strong box-crossing property. Given an $s$-embedding $\cS^{\delta}$ of a graph $(S, (x_e)_{e\in E})$ satisfying the assumption \Uniff\, and any $\rho>0$, fix a square $\Lambda^{\delta}_{\rho}$ of width $\rho$ drawn over $\cS^{\delta}$. Consider the FK-Ising model on $\Lambda^{\delta}_{\rho}$, where random cluster weights are given by the Edwards-Sokal coupling of the Ising weights on $\cS^{\delta} $ defined by \eqref{eq:x=tan-theta}. In the following statement, one then denotes by $\mathbb{P}^{\textnormal{free}}_{\textnormal{FK}}$ the measure with $\textnormal{free}$ boundary conditions on the annulus $\Lambda^{\delta}_{\rho}$. 
\begin{theo}[Theorem 1.2 in \cite{Mah23}]\label{thm:RSW-s-embeddings}
In the previous setup, there exist $c(r_0,\theta_0)>0$, only depending on $r_0,\theta_0$, such that
\begin{equation}
\mathbb{P}^{\textnormal{free}}_{\textnormal{FK}} \big( \textnormal{There exist an open circuit in } \mathrm{A}_{\frac{\rho}{2},\rho} \big) > c(r_0,\theta_0).		
\end{equation}
A similar estimate holds for the dual model.
\end{theo}

\subsection{Main results}

The first major breakthroughs in the study of the planar nearest-neighbour Ising model was the exact computation of the critical temperature by Onsager in \cite{onsager1944crystal} for the homogeneous square lattice, which corresponds in the notation of the present paper to $x_c = \sqrt{2} - 1 = \tan(\frac{\pi}{8})$. Over the past sixty years, the phase transition of the model has been extensively studied (see e.g.\ \cite{mccoy2013two} featuring explicit determinantal computations), with truncated spin correlations decaying exponentially fast in the off-critical regime, while decaying polynomially at criticality. In the late 2000s, significant progress was made on the sharpness of the phase transition for Potts and FK models (see \cite{beffara-duminil}). A landmark for the FK-Ising model result was the proof by Duminil-Copin, Hongler, and Nolin in \cite{DCHN} of the strong box-crossing property on the square lattice, as the first example where Theorem~\ref{thm:RSW-s-embeddings} was established for the Ising model. A central follow-up question was the identification of the so-called \emph{correlation length} (see \cite{FK_scaling_relations} for a precise formulation and applications), which quantifies how far the model can deviate from criticality while still resembling critical behaviour in terms of crossing probabilities. Roughly speaking, in the box $\Lambda_n = [-n;n]^2$, the model behaves critically as long as all coupling constants deviate by no more than $O(n^{-1})$ from the critical value. It was shown in \cite{DuGaPe-near-crit} (for the square lattice) and \cite{park-iso} (for Z-invariant isoradial grids) that the strong box-crossing property holds for both the primal and dual models under such deviations, while deviations of a larger order of magnitude yield off-critical models. The first contribution of the deformation procedure in the $s$-embedding framework is to provide a new proof of the strong box crossing property up to the critical window. In what follows, fix $m>0$ and a collection of masses $(m_e)_{e\in \mathbb{Z}^2}$, all bounded by $m$. Consider the FK-Ising model on $\Lambda_n$, where the coupling constant on the edge $e\in \mathbb{Z}^2$ is $x^{(n)}_e=x_c + \frac{m_e}{n}$, where $x_c=\sqrt{2}-1$ is the critical value for the homogeneous model. In the statement below, one denotes by $\mathbb{P}^{\textnormal{free}}_{\textnormal{FK}}$ the measure with $\textnormal{free}$ boundary conditions on the annulus $\mathrm{A}_{\frac{n}{2},n}$.
\begin{theo}[Duminil-Garban-Pete;Park]\label{thm:near-critical-RSW}
 In the previous setup, there exist $c(m)>0$, only depending on $m$, such that
\begin{equation}
\mathbb{P}^{\textnormal{free}}_{\textnormal{FK}} \big( \textnormal{There exist an open circuit in } \mathrm{A}_{\frac{n}{2},n} \big) > c(m).
\end{equation}
A similar estimate holds for the dual model.
\end{theo}
A remarkable feature of this new proof is the appearance of the scaling window. In essence, the embedding of a near-critical model in the box $\Lambda_n$ can be constructed via a linear ODE of the form $\cY_n'(t) = A(t) \cY_n(t)$, where $\cY_n$ encodes the coordinates of the fermion generating the embedding of $\Lambda_n$, where the initial condition $\cY_n(0)$ corresponds to an $s$-embedding of the critical square lattice. Using computation of \cite{HS-energy-Ising} recalled in Section \ref{sec:massive-deformation}, one sees that $||A(0)|| \asymp n$. Therefore, when continuously moving the Ising weights at a bounded speed, standard ODE theory suggests that the embedding at time $t = O(n^{-1})$ should remain comparable to the critical one, resulting in a \Unif-like grid where Theorem~\ref{thm:RSW-s-embeddings} applies. This deformation strategy naturally extends to any $s$-embedding satisfying a \Unif-type assumption.
\begin{theo}\label{thm:unif-rsw}
Fix $m>0$, a proper $s$-embedding $\cS^{\delta}$ of a graph $(S, (x_e)_{e\in E})$ satisfying \Uniff\,, and a collection of masses $(m_{e})_{e\in E}$ all bounded by $m$. Fix a square $\Lambda^{\delta}_{\rho}$  of width $\rho>0$ drawn over $\cS^{\delta}$. Consider the FK-Ising model on $\Lambda^{\delta}_{\rho}$, where the coupling constant on the edge $e\in \Lambda^{\delta}_{\rho}$ is $x^{(\delta)}_e=x_e + \delta m_e$, and denote by $\mathbb{P}^{\textnormal{free}}_{\textnormal{FK}}$ the measure with $\textnormal{free}$ boundary conditions on the annulus $\Lambda^{\delta}_{\rho}$. Then, there exists $c(m,r_0,\theta_0)>0$, only depending on $m,r_0,\theta_0$, such that
\begin{equation}
\mathbb{P}^{\textnormal{free}}_{\textnormal{FK}} \big( \textnormal{There exist an open circuit in } \mathrm{A}_{\frac{\rho}{2},\rho} \big) > c(m,r_0,\theta_0).		
\end{equation}
A similar estimate holds for the dual model.
\end{theo}
In particular, this approach implies that the near-critical window is universal across (near-)critical Z-invariant isoradial grids with bounded angle conditions, as well as across critical doubly periodic graphs.
\smallskip

We now turn to the main results of the present paper, which focuses on the near-critical square lattice endowed with random weights. A key outcome is that the naively centred randomness around the critical point averages so effectively around the critical point that the random near-critical scaling window can be extended the \emph{cubic root} of the deterministic one. More precisely, with high probability with respect to the random environment, the near-critical window in a random environment has size at least $O(n^{-\frac{1}{3}})$ (up to logarithmic corrections). To formalise this statement, fix a family of i.i.d.\ standard gaussian variables $\omega \mapsto (\mathcal{N}_{e_k}(\omega))_{k \in E(\mathbb{Z}^2)}$, and let $\mathbf{P}$ denote the underlying probability measure. For a realisation $\omega$, define the \emph{$t$-weakly random Ising model} on $\mathbb{Z}^2$ by assigning coupling constants to the angles \eqref{eq:x=tan-theta} as
\[
\theta_{e}^{(t)}(\omega) := \frac{\pi}{4} + t \cdot \mathcal{N}_{e_k}(\omega).
\]
We denote by $\mathbb{P}_{\Lambda_n,\omega,t}$ the FK-Ising measure on $\Lambda_n$ associated with the coupling constants defined by $\omega$.
\begin{theo}\label{thm:near-critical-RSW-random}
Consider the $t$-weakly random Ising model on $\mathrm{A}_{\frac{n}{2},n} $, with free boundary conditions at the inner and the outer boundaries. There exist positive constants $c_{1,2,3}>0$ such that for any $0\leq t \leq c_3\cdot(n\log^{\frac{1}{2}}(n))^{-\frac13}$, one has
\begin{equation}
	\mathbf{P}\Bigg[ \mathbb{P}^{\textnormal{free}}_{\Lambda_n,\omega,t} \big( \textnormal{There exist an open circuit in } \mathrm{A}_{\frac{n}{2},n} \big) > c_1  \Bigg] \geq 1- \frac{c_2}{n^4}.
\end{equation}
\end{theo}
Before passing to more refined statements (i.e.\ conformal invariance of the $t$-weakly random Ising model), let us make additional comments. It is explained in Section \ref{sub:optimality} that one could in principle see some optimality in this statement, at least from a self-duality perspective. This comes from an analogy with the case of Bernoulli percolation \cite{AveMah25}, where some moment and self-duality condition is enough to derive some unique criticality condition in random environment, including those with macroscopic deviations around the critical point that don't scale to $0$ as $n\to \infty$. 

Let us emphasise once again that, in any case, if one replaces $\mathcal{N}_{e_k}$ by $|\mathcal{N}_{e_k}|$, the model at time $t=(n\log^\frac12(n))^{-\frac13}\gg n^{-1}$ would be off-critical by a fair margin. The overall machinery extends (up to logarithmic corrections) to independent random variables (not necessarily identically distributed) with light enough tails centred around the critical points, using the previous result together with the Skorokhod embedding Theorem (see \cite[Section 3.2]{AveMah25} for a complete derivation). Let us also emphasise that starting from an $s$-embedding with no particular symmetries and only satisfying  \Unif\, kind of assumption, one can once again deform (with high $\mathbf{P}$-probability) the original model with random coupling constants centred around the original model with a standard deviation  $O(\delta^{\frac{1}{2}})$, while remaining in the class of (near)-critical Ising models.

When starting from a critical and conformally invariant model, the randomness under $ \mathbf{P}$ linearises so effectively that not only do the macroscopic box-to-box crossing properties remain bounded away from $0$ and $1$, but also some very sensitive and microscopic details of the model, including of the scaling limit of the law of the interface separating primal and dual FK clusters as well as the second order expansion of the energy density random variables. In particular, the scaling limit of the model remains conformally invariant in the limit. To lighten notations, denote $\delta_n=\frac{\sqrt{2}}{n}$ and let $\Omega \subset [-\frac{1}{2},\frac{1}{2}]^2$ be a simply connected domain with two marked boundary points $a, b \in \partial \Omega$ considered as prime ends. Let $(\Omega_{\delta_n}, a^{(\delta_n)}, b^{(\delta_n)})_{n \geq 1}$ be a sequence of discretisation (in the Carathéodory sense), converging in the Carathéodory sense to $(\Omega, a, b)$ on the isoradial lattice $\delta_n \mathbb{Z}^2$. Consider the FK-Ising model on $(\Omega_{\delta_n}, a^{(\delta_n)}, b^{(\delta_n)})$ with wired boundary conditions along the arc $(a^{(\delta_n)} b^{(\delta_n)})^{\circ}$ and free boundary conditions along the arc $(b^{(\delta_n)} a^{(\delta_n)})^{\bullet}$. For a realisation $\omega$ under $\mathbf{P}$ and parameter $t$, denote by $(\Omega_{\delta_n}, a^{(\delta_n)}, b^{(\delta_n)}, \omega, t)$ the $t$-weakly random FK-Ising model with coupling constants determined by $\omega$ on the domain $(\Omega_{(\delta_n)}, a^{(\delta_n)}, b^{(\delta_n)})\subset \delta_n \mathbb{Z}^2$. At the discrete level, define the (leftmost) discrete interface $\gamma^{(\delta_n)}_{\omega,t}$ separating primal and dual FK cluster and connecting $a^{(\delta_n)}$ to $b^{(\delta_n)}$, and \emph{drawn on} $\delta_n \mathbb{Z}^2$.
\begin{theo}\label{thm:random-SLE}
Fix $\alpha>0$. In the previous setup, for $t^{(\alpha)}_n= n^{-(\alpha+\frac{1}{3})}$, one has $\mathbf{P}$-almost surely
\begin{equation}
	\gamma^{(\delta_n)}_{\omega,t^{(\alpha)}_n} \underset{n\to \infty}{\overset{(d)}{\longrightarrow}}  \mathrm{SLE}_{16/3}(\Omega,a,b),
\end{equation}
where $\mathrm{SLE}_{16/3}(\Omega,a,b)$ is the standard chordal Schramm–Loewner-Evolution process in $\Omega$ that connects $a$ to $b$ in $\Omega $. 
\end{theo}
In the above theorem, we prove conformal invariance of the FK interfaces requiring some polynomial correction to conditions of Theorem \ref{thm:near-critical-RSW-random}. As discussed in Remark \ref{rem:extension-logarithmic}, this polynomial correction can in principle be replaced by a polylogarithmic correction instead, with a slightly more technical proof. One can also extend conformal invariance of the near-critical model in a random environment to the so-called energy density of the model, whose second term correction in bounded domains approximated by the lattice is known to be conformally covariant. More precisely, set $(\Omega_{\delta_n})_{n \geq 1}$ be a sequence of approximation (in the Hausdorff sense) of $\Omega $ (which we assume for simplicity to have a smooth boundary). Consider the Ising model on $(\Omega_{\delta_n})$ with wired boundary conditions along the arc. For a realisation $\omega$ under $\mathbf{P}$ and parameter $t$, denote by $(\Omega_{\delta_n}, \omega, t)$ the $t$-weakly random FK-Ising model with coupling constants given by $\omega$ inside $\Omega_{\delta_n}$ and $(\delta_n \mathbb{Z}^2, \omega, t)$ the $t$-weakly random FK-Ising model with coupling constants given by $\omega$ inside $\Omega_{\delta_n}$ and uniform critical homogeneous outside of $\Omega_n$. Fix an inner point $a\in \Omega$ approximated by an edge $e^{\delta_n}_a$. We have the following theorem regarding the scaling limit of the energy density $\varepsilon_{a}^{\delta_n}$ encoding the product of the spins separated by $e_a^{\delta_n}$. 

\begin{theo}\label{thm:random-energy-density}
In the previous setup, one has $\mathbf{P}$-almost surely, 
\begin{equation}
	\frac{\mathbb{E}_{(\Omega_{\delta_n},\omega,t^{(\alpha)}_n)}[\varepsilon_{a}^{\delta_n}] - \mathbb{E}_{(\delta_n \mathbb{Z}^2,\omega,t^{(\alpha)}_n)}[\varepsilon_{a}^{\delta_n}]}{\frac{1}{n}} \underset{n\to \infty}{\longrightarrow} \frac{1}{2\pi}\ell_{\Omega}(a),
\end{equation}
where $\ell_{\Omega}(a)$ is the hyperbolic metric element of $\Omega$ seen from $a$ (i.e.\ $\ell_{\Omega}(a):=2|\phi_{a}'(a) |$, where $\phi_{a}$ is any uniformisation of $\Omega$ to the unit disk $\mathbb{D}$ that vanish at $a$). 
\end{theo}
The previous theorems provide very strong indication that the randomness averages enough at each scale to keep exactly the same microscopic details in the scaling limit. It will be visible within the proof that for $t$ close enough to $n^{-\frac13}$, the energy density $\mathbb{E}_{(\delta_n,\mathbb{Z}^2,\omega ,t)}[\varepsilon_{a}^{\delta_n}]$  typically deviates by much more than $\frac{1}{n}$ from the homogeneous critical full-plane energy density $\mathbb{E}_{(\delta_n\mathbb{Z}^2,\omega,0)}[\varepsilon_{a}^{\delta_n}]=\frac{\sqrt{2}}{2}$. Therefore, in order to obtain some meaningful conformal covariance statement, one needs to normalise additively by the (random) full plane value and not by the critical homogeneous one. 

\bigskip
Let us now present another near-critical random near-critical Ising model, which instead features a \emph{logarithmic} near-critical window. This is \emph{not} in contradiction with the belief of optimality of the $O(n^{-\frac{1}{3}})$ near-critical scaling window in \emph{edge-independent} random environments, as the randomness we use here is \emph{not} made of i.i.d.\ coupling constants at each edge. Instead, the random coupling constants we use slightly depend on each other. Their leading order of magnitude are formed by independent processes at each edge while the associated correction depends on the randomness in the entire box $\Lambda_n$. More precisely, consider a family $((B^{(e)}_t)_{t\geq 0})_{e\in \Lambda_n}$ of i.i.d.\ standard Brownian motions defined under some probability measure $\mathbf{P}$. One can then construct, $\mathbf{P}$-almost surely and for $t \geq 0$ small enough, a family of models $\widehat{S}=(\mathbb{Z}^2,(\hat{x}^{(t)}_{\omega,e})_{e\in E(\mathbb{Z}^2)})_{t\geq 0}$, in which all coupling constants are critical outside of $\Lambda_n$, while for $e_k\in \Lambda_n$, the abstract angle $\hat{\theta}^{(t)}_{e_k,\omega}$ given by \eqref{eq:x=tan-theta} is the solution to the SDE
\begin{equation}\label{eq:interacting-model}
	\hat{\theta}^{(t)}_{e_k,\omega}=\frac{\pi}{4}+\cdot B^{(e)}_{t}(\omega)- \frac{1}{2}\int_{0}^{t}\Bigg( \frac{\cos(\hat{\theta}^{(s)}_{e_k,_{\omega}})}{\sin(\hat{\theta}^{(s)}_{e_k,_{\omega}})} -\frac{\mathbb{E}_{S_{\omega}(s)}[\varepsilon_{e_k}]}{\sin(\hat{\theta}^{(s)}_{e_k,_{\omega}})}.  
 \Bigg)ds.
\end{equation}
This model, which we call the \emph{t-weakly random interacting Ising model}, is \emph{not} i.i.d., as $\mathbb{E}_{S_{\omega}(s)}[\varepsilon_{e_k}]$ \emph{depends} on the values of \emph{all} the edges inside $\Lambda_n$. Nevertheless, adding this interacting drift term constructs a \emph{random Ising  fermion which is a local martingale at each corner}.  This is captured in the following theorem.
\begin{theo}\label{thm:near-critical-RSW-random-interacting}
Consider the $t$-weakly random Ising model on $\Lambda_n$ whose coupling constants are given by \eqref{eq:interacting-model}. Then there exist positive constants $c_1>0$ and $c_2>0 $ such that for any $0\leq t \leq c_2.\log(n)^{-2}$, one has
\begin{equation}
	\mathbf{P}\Bigg[ \mathbb{P}^{\textnormal{free}}_{\Lambda_n,\omega,t} \big( \textnormal{There exist an open circuit in } \mathrm{A}_{\frac{n}{2},n} \big) > c_1  \Bigg] \geq 1-O(\frac{1}{n^4}).
\end{equation}
Moreover, for this interacting model, there exist a large enough contant $C>0$ such that if $t_n=\log(n)^{-C}$  the analog of Theorems \ref{thm:random-SLE} and \ref{thm:random-energy-density} hold. 
\end{theo}
As it can be seen in the proofs, for times $0 \leq t \leq c_2\cdot\log(n)^{-2}$, the drift (interacting) term is, up to logarithmic corrections, of order $O(t^{\frac{3}{2}})$, while each Brownian motion is typically of order $\sqrt{t}$. Therefore, we have constructed a (near-critical) model in a random environment such that, with high probability, the edges weights typically deviate by $\asymp \log^{-1}(n)$ from the critical value, formed of some i.i.d.\ Gaussian leading component and a much smaller random correction whose contribution mainly comes from neighboring edges. From our perspective, trying to use the large scale criticality notions developed in \cite{Che20,Mah23,MahPar25a}, this weakly interacting model represents a very good candidate (e.g.\ on the torus) for a critical Ising model in a random environment (not i.i.d.\ but almost for edges far away from each other) with a small but macroscopic random deviation from the critical temperature at each edge. Moreover, this interacting model could help to hint at which random i.i.d.\ process would be the correct correction to a centred process around the critical point to remain within the critical phase.  
\subsection{Other near-critical models and novelty of the approach}

For many planar statistical mechanics models, the ultimate goal is to prove conformal invariance at criticality, in line with predictions from Conformal Field Theory, traditionally focusing on the study of correlation functions or interfaces. Such ambitious results have been achieved only for a few models, notably site percolation on the triangular lattice \cite{smirnov2001critical,smirnov2001criticalperco}, the dimer model \cite{Kenyon-dimer-conformal,kenyon2001dominos,basok2021tau,basok2025nesting,CheRam,berggren2025gaussian}, the planar Ising model \cite{Smirnov-conformal,honglersmirnov,CHI,ChelkakSmirnov_et_al,ChSmi1,Che20}, and the harmonic explorer \cite{HarmExplorer}, each proof developing some very  model specific tools.
A key breakthrough was Schramm's introduction of the SLE processes \cite{schramm2000scaling}, which provided a canonical one-parameter family of conformally invariant random curves, each parameter corresponding to a specific model. The common strategy in these convergence results involves constructing a discrete harmonic/holomorphic observable related to the model and proving its convergence as the mesh size of the discretizing grid tends to zero. Building on these results at criticality, a parallel line of research has emerged focusing on near-critical models, extending the techniques and ideas beyond the critical point. In practice, studying near-critical models involves scaling the model parameters toward their critical values at a suitable rate $\eta(\delta) \to 0$, while simultaneously sending the lattice mesh size $\delta \to 0$. The resulting scaling limit typically differs from the critical one, yet remains non-degenerate. This approach was formalized in a general framework by Makarov and Smirnov in \cite{makarov2010off}, who introduced the notion of \emph{near-critical SLE curves}—a modification of the standard SLE in which the driving term in the Loewner equation is no longer pure Brownian motion, but a perturbation thereof, reflecting the deviation from criticality. This methodology, which also encompasses correlation functions, has proven effective in a variety of models: massive loop-erased random walk \cite{CheWan,berestycki2022near}, massive dimers \cite{chhita2012height,berestycki2022near,WanPHD}, the massive harmonic explorer \cite{papon2023massive}, the Ising model with magnetic perturbation \cite{camia2014ising,papon2024interface}, the Ising model with energy perturbation \cite{DuGaPe-near-crit,park2018massive,park-iso,CIM-universality,WanPHD}, and near-critical percolation \cite{nolin2009asymmetry}.

In all these examples, identifying the correct scaling of the perturbation factor $\eta(\delta)$ typically follows one of two main approaches. The first is to choose $\eta(\delta)$ as the largest order of magnitude for which the Radon–Nikodym derivative of the near-critical measure with respect to the critical one remains under control. This allows one to leverage the well-developed theory of SLE/CLE processes and quasi-conformal mappings to deduce some potential existence, uniqueness and qualitative properties of the near-critical limit. The second approach is to identify the scaling $\eta(\delta)$ that perturbs the harmonicity or holomorphicity of the critical observable into a meaningful, non-degenerate equation—distinct from the discrete holomorphicity observed at criticality.
Even after this step, identifying the full scaling limit remains challenging, particularly in fractal domains typically generated by SLE or massive SLE processes. In any case those limit remain conformally invariant.

Two examples illustrate the subtleties of near-critical regime. First, in the energy-perturbed FK-Ising model, convergence of the martingale observable in arbitrarily rough domains is known \cite{park-iso}, as is the precompactness of the interfaces \cite{DuGaPe-near-crit,park-iso}. Yet, these results are not sufficient to fully characterise the limiting near-critical process, which is conjectured to be \emph{absolutely continuous} with respect to $\mathrm{SLE}(16/3)$ —the known scaling limit at criticality. This conjecture is supported by arguments  provided by Garban and Kupiainen in \cite{garban2025energy}. In contrast, the work of Nolin and Werner \cite{nolin2009asymmetry} shows that the near-critical percolation in the correct non-trivial and non-degenerate regime is \emph{not} absolutely continuous with respect to $\mathrm{SLE}(6)$, even though both models are supported on sets of curves with the same Hausdorff dimension $7/4$. These examples highlight the delicate and subtle nature of the local behavior in near-critical models. Fortunately, the stability of the conformal structure associated to near-critical random environments allows to still draw some interesting conclusions.

In this paper, we propose a novel approach that departs from classical combinatorial methods by leveraging the geometric interpretation of $s$-embeddings for weighted planar graphs. Rather than studying which near-critical deformations remain tractable from the probabilisitic perspective, we investigate which geometric deformations of the $s$-embedding yield non-degenerate Ising model. This shift offers two major advantages. First, it translates the problem into the realm of ordinary and stochastic differential equations, relating the Ising model to the non-degeneracy of ODE/SDE, building a new connection.  Second, and more importantly, the framework developed in \cite{CLR1,CLR2,Che20,Mah23,MahPar25a,ChePar} enables working with highly irregular grids, that vary locally and may in principle include several mesoscopic regions far from criticality. Although the present paper focuses on regular grids as a toy model, the techniques naturally extend to more general configurations, including those with localised degeneracies. A related project with Avérous \cite{AveMah25} investigates near-critical FK models with $1 \leq q \leq 4$, focusing on the stability of crossing probabilities under near-critical i.i.d.\ perturbations, typically larger than the deterministic critical window. While this enhances the results of \cite{FK_scaling_relations} for random deformations, it is limited to setups where criticality is preserved \emph{at each scale} (except for percolation, where some additional noise sensitivity argument allows to conclude for random environments with macroscopic deformations at least at large scale). In contrast, our embedding-based method captures the deformation of the entire discrete conformal structure, governing how all discrete fermionic observables evolve across scales. Notably, it can in principle accommodate local mixtures of supercritical and subcritical regions that balance out spatially to be critical on average.
\subsection{Related works and open questions towards the random bond conjecture}
From our perspective, this paper represents a first step toward extending the universality meta-principle —widely believed in the statistical mechanics community but only proven in very special cases— beyond the realm of integrable models. We conclude by outlining several applications of our approach in different contexts.
\begin{itemize}
    \item A first possible generalisation of this work concerns the convergence of massive fermionic observables in bounded domains beyond the massive Z-invariant setting, for which Park \cite{park2018massive,park-iso} developed the theory. Instead, to study directly the massive model, it is easier to deform continuously both the conformal structure and the fermions towards the massive setup, allowing to transfer along the way the improved knowledge the boundary behaviour of critical fermions to their massive counterparts. Let us note that for a \emph{generic} set of masses in $\Lambda_n$, finding the correct $s$-embedding of the massive model remained an open question until the present paper. Our construction allows to apply discrete complex analysis techniques to all massive models, at least when all the masses are small (but macroscopic) enough. Another interesting deformation is along the variety of critical doubly-periodic graphs identified in \cite{cimasoni-duminil}. In that case, almost all the terms of \eqref{eq:ODE-embedding} cancel when regrouped around some given fundamental domain, allowing a continuous exploration of the algebraic variety of  critical doubly-periodic models (see Remark \ref{rem:doubly-periodic-deformation}). Simultaneous deformation of the embedding and fermions provides some statement of universality for discrete Ising fermions, offering a new explanation for universality statements on critical periodic graphs, even before taking scaling limits. These questions are addressed in ongoing work \cite{Mah25b}.
    
    \item A similar deformation strategy applies to the dimer model, continuously deforming the associated $t$-embeddings. Here, the role of two-point fermions in \eqref{eq:ODE-embedding} is played by the inverse Kasteleyn operator, whose mismatch can be used to deform the $t$-embedding gauges when one changes the dimer weights. This enables the study of universality for periodic dimer models on the torus, as well as for massive and weakly random dimers—beyond the traditional Temperley-integrable framework. This direction is being pursued jointly with Basok and Laslier in \cite{BaLaMa25}. In particular, for both for the Ising and dimer contexts, the deformation process encodes the \emph{correct} conformal structure attached to any near-critical massive model (in particular with non-homogeneous weights), and provides some meaningful and exact discrete complex analysis framework to study the massive models.    
    \item For general critical FK models with $1\leq q \leq 4$, one can again consider near-critical random deformations to study the weakly random scaling window, focusing in particular on crossing probabilities. This makes it possible to extend the scaling relations of \cite{FK_scaling_relations} to the near-critical random setting. Concretely, if the deterministic near-critical window has size $n^{-\nu}$ (for some conjectured $\nu(q)>0$), then in the random case it becomes $n^{-\nu/3}$. To obtain this cubic-root reduction of the critical window, one additionally needs to assume conformal invariance of the critical model, which guarantees, via CLE techniques, that the model's mixing rate is sufficiently large. In that paper, special attention is given to Bernoulli percolation, where independence implies that the random critical window scales instead like a negative power of $\log(n)$. This is asymptotically much larger than any polynomial window in $n$, and thus substantially exceeds the deterministic window. Moreover, if one is only interested in large-scale properties, noise sensitivity arguments allow the use of even macroscopic deformations while still preserving asymptotic criticality. In particular, this enables a proof of Cardy's formula in a random environment. For general FK models, where no graphical representation is available, the methods rely crucially on the stability of crossing estimates at \emph{all scales}, together with the framework developed in \cite{FK_scaling_relations}. However, the work with Avérous \cite{AveMah25} does not, in its present form, capture the stability of the limiting interface process, nor can it handle spatial mixtures of locally off-critical configurations that average out critically.
    \end{itemize}	
Finally, let us mention that we hope that the deformation procedure presented in this paper offers promising avenues to approach two longstanding conjectures concerning the critical Ising model in random environments. 
 We now state these open problems and explain their potential relation to our method.

\textbf{Open Question 1:} Fix two positive coupling constants $J_1, J_2$, and assign to each edge an independent random variable (e.g.\ with probability $1/2$) taking values in $\{J_1, J_2\}$. Then with high probability with respect to the environment, there should exist some critical temperature $\beta=\beta(J_1, J_2)$ such that, at large scales, the model satisfies the strong box crossing property. The continuity of the phase transition for the random bond model was settled in \cite{aizenman1990rounding}, with no identification of the critical point. In the last decades, the study of the random bond model got some serious attention  (including many simulations by physicists \cite{cho1997criticality,shalaev1994critical,merz2002two,dotsenko1995renormalisation}). Still, to the best of our knowledge, it remains a challenge to get a rigorous understanding of what should happen a the critical point. In our framework, one could attempt to tackle this question by first sampling a fair coin at each edge, moving all the heads edges at some given positive speed, moving all tails edges at some negative speed, tuning the relative speeds such that the embedding remains within the \LipKd\ class, at least starting from some mesoscopic scale. One challenge lies in determining whether such a deformation can be carried out for a time bounded from below uniformly in the system size $n$. 
One can alternatively try to find some suitable random process to replace the naive choice of Brownian motions in the SDE \eqref{eq:SDE-embedding}, hoping to replace the continuously the naive centred deformation (in the spirit of the weakly interacting random model) by a more suited process.

Let us mention that this first conjecture looks to be in contradiction with the work \cite{dotsenko1983critical}, where it is conjectured that the magnetisation exponent is supposed to go from $\frac{1}{4}$ in the deterministic critical environment to $0$. Still, this conjecture of \cite{dotsenko1983critical} appears itself to be at least partially in contradiction with \cite{shankar1987exact} and with the annealed quadrichotomy of \cite[Section 7]{gunaratnam2024existence} and \cite{Trish}, which would ensure some strong box crossing property at the critical point of the random environment. 
We now turn to a second open problem that could similarly benefit from this approach.

\textbf{Open Question 2:} Fix $\frac12 < p < 1$ large enough, and perform some independent Bernoulli percolation on the edges of $\mathbb{Z}^2$ with parameter $p$. Then, almost surely with respect to the environment, there exists some unique infinite percolation cluster $\mathcal{C}(p)$. Consider the homogeneous Ising model on vertices of $\mathcal{C}(p)$ at temperature $\beta$. Then, with high probability with respect to the environment, there should exist some $\beta=\beta(p)$ such that, at large scales, the homogeneous model at temperature $\beta(p)$ satisfies the strong box crossing property. One potential route would to be work with $s$-embeddings attaching spins to the primal graph, and  adapt the above idea by sending the abstract angle $\theta$ corresponding to deleted edges to $0$ (effectively forbidding them in FK clusters), while tuning up the couplings on retained edges at a same positive speed. This philosophy could also be used to study the Blume-Capel model, whose critical phase has recently been studied in \cite{gunaratnam2024existence}. The goal would again be to make a continuous deformation of weights, to maintain the embedding within the \LipKd\ class, ensuring non-degeneracy of the crossing probabilities.

For both problems, the potential advantage of this \emph{quenched} deformation-based approach is that it circumvents the lack of knowledge of the critical temperature in the random environment. Instead, it navigates within the space of non-degenerated embeddings and thus critical models. An idea of the same spirit already proved to be effective \cite{giuliani2012scaling,cava2025scaling,antinucci2023energy} (in the full-plane, the half-plane and cylinders) to prove convergence statements for the critical non-nearest neighbour Ising models by deforming continuously (via some renormalisation group based techniques) the Ising weights starting from the nearest neighbour model and simultaneously moving the temperature and the emerging non-nearest neighbour coupling, while remaining within the critical phase.

Both proposed strategies still require additional probabilistic input (e.g.\ the uniqueness and deterministic nature of the critical temperature) and using some more refined statement that up to constant bounds on the two points fermions, which \emph{are knwon in a fairly general context via} \cite{MahPar25a}, we believe the present paper might  provide meaningful progress toward understanding these conjectures.

\medskip
\textbf{Acknowledgement:} This paper is the fruit of many discussions and insights shared many people. The author is grateful for all those external inputs. I would like to thank  Emile Avérous for countless discussions on percolation models in random environments while working on \cite{AveMah25}, Dmitry Krachun for viewing Lemma \ref{lem:deformation-discrete} as a differential process, Dmitry Chelkak for suggesting to feed the deformation with a continuous random process, as well as Mikhail Basok, Nikolai Bobenko, Benoit Laslier, Sung-Chul Park and Ekin Arikok for discussion around discrete complex analysis methods and potential applications to the dimer context. I would also like to thank David Cimasoni, Hugo Duminil-Copin, Ioan Manolescu and Sanjay Ramassamy for encouraging this research. The author is also grateful to Béatrice de Tilière, Christophe Garban, Trishen S. Gunaratnam, Léonie Papon, Stanislav Smirnov and Yijun Wan for enlightening discussions and useful references. This project has received funding from the Swiss National Science Foundation and the NCCR SwissMAP. art of this research was performed while all authors were visiting the Institute for Pure and Applied Mathematics (IPAM), which is supported
by the National Science Foundation (Grant No. DMS-1925919).
\section{Notations and crash intro into the s-embedding formalism}\label{sec:notations}
\setcounter{equation}{0}

We concisely recall the general construction of s-embeddings introduced in \cite[Section 3]{Che20}, along with the regularity theory of the so-called s-holomorphic functions, both derived from a complexification of the standard Kadanoff-Ceva formalism. Our notation follows precisely that of \cite{Che20, Mah23,MahPar25a} and is consistent with \cite[Section~3]{CCK}, \cite{Ch-ICM18}, and \cite{CLR1,CLR2}. As we do not provide proofs, we refer the reader to \cite[Section 2]{Che20} for further details. Chelkak’s original idea was to construct a class of embeddings associated with a given weighted abstract graph, where the weights carry a geometric interpretation, enabling the application of discrete complex analysis techniques.

\begin{figure}
\begin{minipage}{0.325\textwidth}
\includegraphics[clip, width=1.2\textwidth]{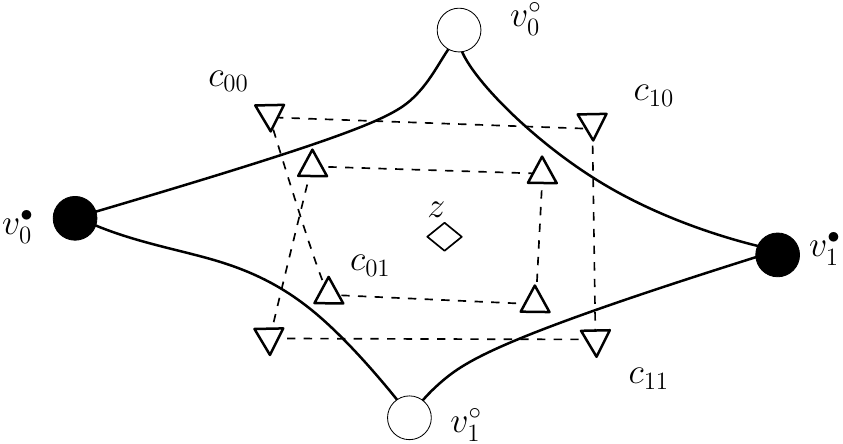}
\end{minipage}\hskip 0.10\textwidth \begin{minipage}{0.33\textwidth}
\includegraphics[clip, width=0.9\textwidth]{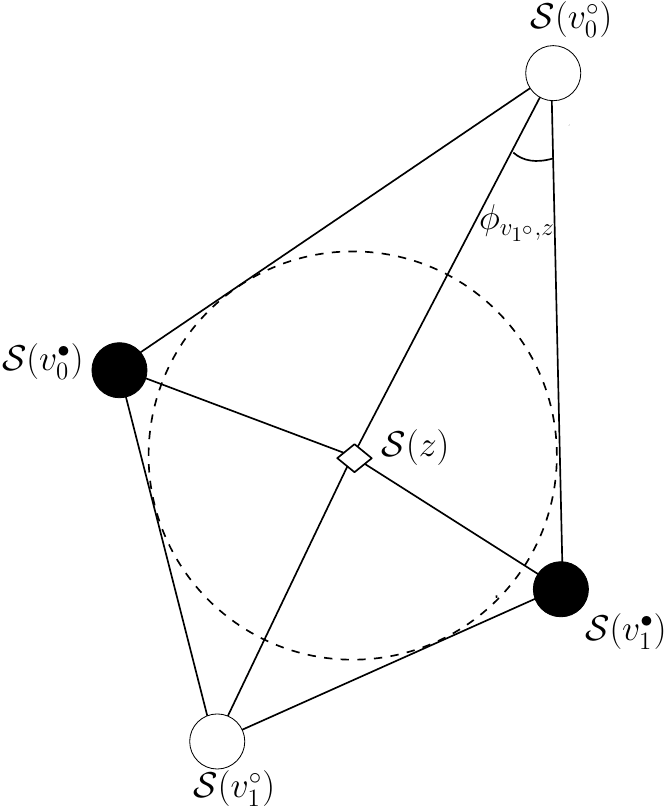}
\end{minipage}
\caption{(Left) Notation for a given quad $z \in \diamondsuit(G)$ with an arbitrary embedding in the plane. Vertices of the primal graph $G^\bullet$ are shown as black dots, while vertices of the dual graph $G^\circ$, corresponding to the faces of $G$, are represented as white dots. The so-called corners, corresponding to the edges of the bipartite graph $\Lambda(G) = G^\bullet \cup G^\circ$, are depicted as triangles. This figure illustrates a portion of the \emph{double cover} of the corner graph, branching around $z$. Corners that are neighbors \emph{in this double cover} are connected by dashed lines. (Right) A portion of the associated $s$-embedding containing the quad $\cS^{\diamondsuit}(z)$, tangent to a circle with radius $r_z$ centered at $\cS(z)$. The Ising weight of the edge between the vertices $v_0^\bullet$ and $v_1^\bullet$ can be recovered using the four angles $\phi_{v,z}$ associated with the quad $\cS^{\diamondsuit}(z)$, following the formula in \eqref{eq:theta-from-S}.}
\label{fig:graph-notations}
\end{figure}

\subsection{Notation and Kadanoff--Ceva formalism}\label{sub:notation}

Let us fix G as a planar graph, allowing multi-edges and vertices of degree 2 but not loops or vertices of degree 1, with the combinatorics of either the plane or the sphere. The graph G is considered up to homeomorphisms that preserve the cyclic order of edges around each vertex. In the spherical case, one designates a particular face of G as the outer face.

We denote by $G=G^\bullet$ the original graph, where vertices are represented by $v^\bullet\in G^\bullet$, and by $G^\circ$ its dual, where vertices $v^\circ\in G^\circ$ correspond to the faces of G. The faces of the graph $\Lambda(G):= G^\circ\cup G^\bullet$, which has a natural incidence relation as a bipartite graph, are in straightforward bijection with the edges of G. Additionally, we denote by $ \Dm(G) $ the graph dual to $\Lambda(G)$, where vertices, labeled as $ z\in\Dm(G) $, correspond to its faces. This graph is commonly referred to as the quad graph. Finally, we define $\Upsilon(G)$ as the medial graph of $\Lambda(G)$, whose vertices—called the corners of G—are in direct bijection with the edges $(v^\bullet v^\circ)$ of $\Lambda(G)$.

To ensure the full consistency of the Kadanoff-Ceva formalism, it is generally necessary to work with various double covers of the graph $\Upsilon(G)$. For relevant illustrations of these double covers, see, for example, \cite[Fig.~27]{Mercat-CMP} or \cite[Fig 3.A]{Che20}. In this paper, we denote by $\Upsilon^\times(G)$ the double cover that branches over \emph{all} the faces of $\Upsilon(G)$, meaning around each element of the type $v^\bullet \in G^\bullet, v^\circ\in G^\circ$, and $ z\in \Dm(G) $. When G is finite, this definition remains meaningful since the quantity  $\#(G^\bullet)+\#(G^\circ)+\#(\Dm(G))$ is always even.
Given a set $ \varpi=\{\vbullet{m},\vcirc{n}\}\subset \Lambda(G) $, where both integers n and m are even, we define $\Upsilon^\times_\varpi(G)$ as the double cover of $\Upsilon(G)$ that branches over all its faces \emph{except} those in $\varpi$. Similarly, we denote by $\Upsilon_\varpi(G)$ the double cover of $\Upsilon(G)$ that branches \emph{only} over the faces in $\varpi$. A function defined on any of these double covers is called a \emph{spinor} if its values at two different lifts of the same corner differ solely by a sign, i.e., by a multiplicative factor of $-1$.

In this paper, we work with the Ising model on the faces of the graph $G$, including the outer face in the disc case, which corresponds to starting with \emph{wired} boundary conditions. This statistical mechanics model generates a random assignment of $\pm1$ variables to the vertices of $G^\circ$, governed by the partition function \eqref{eq:intro-Zcirc}. The associated low-temperature expansion \cite[Section 1.2]{CCK} maps a spin configuration $\sigma : G^\circ\to\{\pm 1\}$ to a subset $C$ of edges in $G$ that separate spins of opposite sign. This mapping is, in fact, a $2$-to-$1$ correspondence, depending on the value assigned to the spin at the outer face.

One can fix an even number $n$ of vertices $\vcirc{n}\in G^\circ$ and consider a subgraph $\gamma^\circ=\gamma_{[\vcirc{n}]}\subset G^\circ$ that has odd degree only at the vertices of $\vcirc{n}$ and even degree at all other vertices of $G^\circ$. Such a configuration can be interpreted as a collection of paths on $G^\circ$ that pairwise connect the vertices in $\vcirc{n}$. Denoting
\[
x_{[\vcirc{n}]}(e)\ :=\ (-1)^{e\cdot\gamma_{[\vcirc{n}]}}\,x(e),\quad e\in E(G),
\]
where $e\cdot\gamma=0$ if the edge $e$ doesn't cross the path $\gamma$ and $e\cdot\gamma=1$ otherwise. It is possible to reconstruct the correlation formula
\begin{equation}
\label{eq:Esigma}
\textstyle \mathbb E\big[\svcirc{n}\big]\ =\ {x_{[\vcirc{n}]}(\cE(G))}\big/{x(\cE(G))},
\end{equation}
where $x(C):=\prod_{e\in C}x(e)$, $x(\cE(G)):=\sum_{c\in\cE(G)}x(C)$, and similarly for product of the kind $x_{[\vcirc{n}]}$.

If $m$ is again even and \( \vbullet{m} \in G^\bullet \), one can fix a subgraph \( \gamma^\bullet = \gamma^{[\vbullet{m}]} \subset G^\bullet \) with even degree at all the vertices of \( G^\bullet \), except those belonging to \( \vbullet{m} \). In the spirit of the Kadanoff-Ceva formalism \cite{kadanoff-ceva}, one can change the signs of the interaction constants \( J_e \mapsto -J_e \) on edges \( e \in \gamma^\bullet \), which is equivalent to replacing \( x(e) \) by \( x(e)^{-1} \) along the edges of \( \gamma^\bullet \), thereby making the model anti-ferromagnetic near \( \gamma^\bullet \). This leads to the random variable (which still depends on the choice of \( \gamma^\bullet \)):
\[
\textstyle \muvbullet{m}\ :=\ \exp\big[-2\beta\sum_{e\in\gamma^{[\vbullet{m}]}}J_e\sigma_{v^\circ_-(e)}\sigma_{v^\circ_+(e)}\,\big]\,.
\]
The domain walls representation gives that (e.g. \cite[Propositon 1.3]{CCK})
\begin{equation}
\label{eq:Emu}
\textstyle \mathbb E\big[\muvbullet{m}\big]\ =\ x(\cE^{[\vbullet{m}]}(G))\big/{x(\cE(G))},
\end{equation}
where \( \cE^{[\vbullet{m}]} \) is the set of subgraphs with even degree at all vertices, except for those in \( \vbullet{m} \), which have odd degrees. Taking the expectation in \eqref{eq:Emu}, the result no longer depends on \( \gamma^\bullet \). The key observation is that one can generalize \eqref{eq:Esigma} and \eqref{eq:Emu} to the case where both spins and disorder are present simultaneously, which now reads as (e.g. \cite[Propositon 3.3]{CCK}):\begin{equation}
\label{eq:Emusigma}
\textstyle \mathbb E\big[\muvbullet{m}\svcirc{n}\big]\ =\ x_{[\vcirc{n}]}(\cE^{[\vbullet{m}]}(G))\big/{x(\cE(G))},
\end{equation}
where the variable \( \muvbullet{m} \) retains the same definition as above. However, an additional difficulty arises for these mixed correlations. Specifically, the sign of the last expression now depends on the parity of the number of intersections between the paths \( \gamma^\circ \) and \( \gamma^\bullet \). There is no canonical way to fix this sign in \eqref{eq:Emusigma} while staying within the Cartesian product structure \( (G^\bullet)^{\times m} \times (G^\circ)^{\times n} \). To circumvent this issue, one can fix \( \cS : \Lambda(G) \to \mathbb{C} \), an arbitrarily chosen embedding of \( G \), and consider the natural double cover of \( (G^\bullet)^{\times m} \times (G^\circ)^{\times n} \), which branches exactly as the spinor \( \left[ \prod_{p=1}^m \prod_{q=1}^n (\cS(v^\bullet_p) - \cS(v^\circ_q)) \right]^{1/2} \). Following the detailed discussion in \cite[Section 2.2]{CHI-mixed}, the expressions in \eqref{eq:Emusigma} are spinors on \( \left[ \prod_{p=1}^m \prod_{q=1}^n (\cS(v^\bullet_p) - \cS(v^\circ_q)) \right]^{1/2} \). When working with mixed correlations of the form \eqref{eq:Emusigma}, the usual Kramers-Wannier duality (see again \cite[Propositon 3.3]{CCK}) implies that \( G^\bullet \) and \( G^\circ \) play equivalent roles.

Among all possible correlators of the form \eqref{eq:Emusigma}, one can focus on the special case where one of the disorders \( v^\bullet(c) \in G^\bullet \) and one of the spins \( v^\circ(c) \in G^\circ \) are chosen to be neighbors in \( \Lambda(G) \), linked by an edge identified with a corner \( c \in \Upsilon(G) \). In this case, one can formally denote the \emph{fermion} at \( c \) by
\begin{equation}
\label{eq:KC-chi-def}
\chi_c:=\mu_{v^\bullet(c)}\sigma_{v^\circ(c)},
\end{equation}
One can now use \eqref{eq:Emusigma} to construct the  \emph{Kadanoff-Ceva fermion}, in a purely combinatorial manner, setting
\begin{equation}
\label{eq:KC-fermions}
X_{\varpi}(c):=\E[\,\chi_c \mu_{v_1^\bullet}\ldots\mu_{v_{m-1}^\bullet}\sigma_{v_1^\circ}\ldots\sigma_{v_{n-1}^\circ} ].
\end{equation}
Given the above remarks, the fermionic observable \( X_\varpi(c) \) is defined up to a sign, but its definition becomes fully legitimate when working in \( \Upsilon^\times_\varpi(G) \). Around \emph{each} quad \( z = (v_0^\bullet, v_0^\circ, v_1^\bullet, v_1^\circ) \) (listing its vertices in counterclockwise order, as in \cite[Figure 3.A]{Che20} or Figure \ref{fig:graph-notations}), the Kadanoff-Ceva observables satisfy simple local linear equations, with coefficients depending only on the Ising coupling constant attached to the quad \( z \). This propagation equation was first introduced in the works of \cite{dotsenko1983critical}, \cite{perk1980quadratic}, and \cite[Section 4.3]{Mercat-CMP}. 

Let us be more concrete. Let \( \theta_z \) be the abstract angle corresponding to the parametrization in \eqref{eq:x=tan-theta} of the edge in \( G^\bullet \) attached to the quad \( z \). Then, for any triplet of corners \( c_{pq} \) (identified as \( c_{pq} = (v^\bullet_p v^\circ_q) \)), where the lifts of \( c_{pq} \), \( c_{p,1-q} \), and \( c_{1-p,q} \) to \( \Upsilon^\times_\varpi(G) \) are neighbors, one has
 \begin{equation}
\label{eq:3-terms}
X(c_{pq})=X(c_{p,1-q})\cos\theta_z+X(c_{1-p,q})\sin\theta_z,
\end{equation}
where . One can easily show that solutions to \eqref{eq:3-terms} are automatically spinors on $\Upsilon^\times_\varpi(G)$.

\medskip

To conclude this overview of Kadanoff-Ceva correlators, we recall the generalization of the Dirac spinor \( \eta_c \), which represents a special solution to the equation \eqref{eq:3-terms} \emph{on isoradial grids}. Given a fixed embedding \( \cS : \Lambda(G) \to \mathbb{C} \) of \( \Lambda(G) \) into the complex plane, we define (as in \cite{ChSmi2}):
\begin{equation} \label{eq:def-eta}
\eta_c:=\varsigma\cdot \exp\big[-\tfrac{i}{2}\arg(\cS(v^\bullet(c))-\cS(v^\circ(c)))\big],\qquad \varsigma:=e^{i\frac{\pi}{4}},
\end{equation}
where the prefactor \( \varsigma = e^{i \frac{\pi}{4}} \) is chosen for convenience. To avoid the sign ambiguity in \eqref{eq:def-eta}, one can again work on the double cover \( \Upsilon^\times(G) \). Specifically, the products \( \eta_c X_\varpi(c): \Upsilon_\varpi(G) \to \mathbb{C} \) are defined on \( \Upsilon_\varpi(G) \), which only branches over \( \varpi \). Note that we continue to use the notation in \eqref{eq:def-eta} below, even when \( \cS \) is not an isoradial grid.

\subsection{Definition of an s-embeddings}\label{sub:semb-definition}

We now present the embedding procedure introduced by Chelkak in \cite[Section 6]{Ch-ICM18} and further developed in greater detail in \cite{Che20}. First, we recall the definition of an s-embedding as given in \cite[Definition 2.1]{Che20}, which utilizes the Kadanoff-Ceva formalism. The general philosophy here is not based upon finding which weights are naturally attached to a given tilling of the plane by tangential quadrilateral but goes other way around, looking for an embedding that fits the Ising weights. The central idea is to use a solution to \eqref{eq:3-terms} in order to construct a concrete embedding associated with the weighted graph.

\begin{definition}\label{def:cS-def}
Let $(G,x)$ be a weighted planar graph with the combinatorics of the plane, and let $\cX:\Upsilon^\times(G)\to\C$ be a solution to the full system of equations \eqref{eq:3-terms} around each quad. We say that $\cS=\cS_\cX:\Lambda(G)\to\C$ is an s-embedding of $(G,x)$ associated with $\cX$ if, for each $c\in\Upsilon^\times(G)$, we have\begin{equation}
\label{eq:cS-def}
\cS(v^\bullet(c))-\cS(v^\circ(c))=(\cX(c))^2.
\end{equation}
For $z \in \Dm(G)$, the quadrilateral $\cS^\dm(z) \subset \mathbb{C}$ is the region bounded by the edges connecting the vertices $\cS(v_0^\bullet(z))$, $\cS(v_0^\circ(z))$, $\cS(v_1^\bullet(z))$, and $\cS(v_1^\circ(z))$. The s-embedding $\cS$ is called \emph{proper} if the quadrilaterals $\cS^\dm(z) = (\cS(v_0^\bullet(z)) \cS(v_0^\circ(z)) \cS(v_1^\bullet(z)) \cS(v_1^\circ(z)))$ do not overlap. It is called \emph{non-degenerate} if no quadrilateral $\cS^\dm(z)$ degenerates into a segment. No convexity condition is imposed on the quadrilaterals $\cS^\dm(z)$.
\end{definition}
Given a fixed solution $\cX $ to \eqref{eq:3-terms}, it is not clear at all that the obtained embedding $\mathcal{S}_{\cX}$ is proper, meaning that finding a solution to \eqref{eq:3-terms} that leads to a non-degenerate proper picture is a non-trivial step. This will be one of the main innovations of the present paper. Starting from a proper s-embedding for a given set of weights, it will be possible to construct a proper s-embedding for another set a weights in a differential manner. One can also extend the definition of $\cS$ fixing the position of centers of quads $\Dm(G)$, setting as in \cite[Equation (2.5)]{Che20}
\begin{equation}\label{eq:cS(z)-def}
\begin{array}{l}
\cS(v_p^\bullet(z))-\cS(z):=\cX(c_{p0})\cX(c_{p1})\cos\theta_z,\\[2pt]
\cS(v_q^\circ(z))-\cS(z):=-\cX(c_{0q})\cX(c_{1q})\sin\theta_z,
\end{array}
\end{equation}
where $c_{p0}$ and $c_{p1}$ (respectively, $c_{0q}$ and $c_{1q}$) are neighbors on $\Upsilon^\times(G)$.
The propagation equation \eqref{eq:3-terms} directly implies the consistency of both \eqref{eq:cS-def} and \eqref{eq:cS(z)-def}. From a more concrete perspective (see Figure \eqref{fig:graph-notations}), the image $\cS^\dm(z) \in \mathbb{C}$ of a combinatorial quadrilateral $z \in \diamondsuit(G)$, under the embedding $\cS$, is a quadrilateral \emph{tangent to a circle} centred at $\cS(z)$ defined in \eqref{eq:cS(z)-def}. The position of the point $\cS(z)$ corresponds to the intersection of the four bisectors of the sides of the tangential quadrilateral $\cS^\dm(z)$. The radius $r_z$ of this circle can be determined using the values of $\cX$, for instance, through \cite[Equation (2.7)]{Che20}. Let $\phi_{v,z}$ represent the half-angle of the quadrilateral $\cS^\dm(z)$ at $\cS(v)$. It is then possible to reconstruct the abstract Ising weight $\theta_z$ (in the parametrization \eqref{eq:x=tan-theta}) from the angles in the image of $\cS^\dm(z) \subset \mathbb{C}$ using the formula \cite[Equation (2.8)]{Che20}.

\begin{equation}
\label{eq:theta-from-S}
\tan\theta_z\ =\ \biggl(\frac{\sin\phi_{v_0^\bullet ,z}\sin\phi_{v_1^\bullet ,z}}{\sin\phi_{v_0^\circ ,z}\sin\phi_{v_1^\circ ,z}}\biggr)^{\!1/2}.
\end{equation}

In the $s$-embedding framework, the large-scale properties of the origami map associated with an embedding determine whether a planar graph equipped with Ising weights can be interpreted as a (near)-critical system. The definition of this framework is recalled from \cite[Definition 2.2]{Che20} (see also \cite{KLRR, CLR1} for the general definition in the dimer context).

\begin{definition}\label{def:cQ-def}
Given $\cS = \cS_\cX$, the \emph{origami} function, denoted by $\cQ = \cQ_\cX : \Lambda(G) \to \mathbb{R}$, is defined (up to a global additive constant) as a real-valued function. Its increments between two neighboring vertices $v^{\bullet}(c)$ and $v^\circ(c)$ are given by
\begin{equation}
\label{eq:cQ-def}
\cQ(v^\bullet(c))-\cQ(v^\circ(c))\ :=\ |\cX(c)|^2\,=\,|\cS(v^\bullet(c))-\cS(v^\circ(c))|\,.
\end{equation}
We will often shorten $|\cX(c)|^2=\delta_c$ the length of the edge of $\Lambda(G)$ attached to the corner $c$. The alternate sum of edge-lengths in a tangential quad vanishes, which ensures the consistent definition for $\cQ$. One can see $\cQ$ as a folding of the quadrilaterals along their diagonals (see e.g. \cite[Section 8.2]{CLR1}), which makes $\cQ$ a $1$-Lipschitz in the $\cS $ plane. 
\end{definition}

In a general tiling made by tangential quadrilaterals, especially one which is locally very irregular, one should first define some notion of \emph{scale} before talking about large scale properties. We follow here the route taken in \cite{CLR1}, using the assumption \LipKd\ on $\cQ$.
\begin{assumpintro}[\LipKd] We say that $\cS $ satisfies the assumption $\textup{Lip}(\kappa,\delta)$  for some $0\leq \kappa<1$ and $\delta >0 $ if for any $v,v' $ vertices of $\Lambda(G) $
\begin{equation}
\label{eq:LipKd}
|\cQ(v')-\cQ(v)|\le\kappa\cdot |\cS(v')-\cS (v)|\quad \text{if}\quad |\cS (v')-\cS (v)|\ge\delta.
\end{equation}
\end{assumpintro}
With the above assumption, one can define the scale of $\cS$ for the constant $\kappa<1 $.
\begin{definition}
We say that an s-embedding $\cS $ covering the open set $U\subseteq \mathbb{C}$ has a scale $\delta $ for the constant $ \kappa <1$ on $U$ if 
\begin{equation}
\delta = \delta^{\kappa} := \inf \{ \tilde{\delta} >0, \textrm{Lip}(\kappa, \tilde{\delta}) \textrm{ holds} \}.
\end{equation}
\end{definition}
In that case, one writes $\cS = \cS^{\delta} = \cS^{\delta^\kappa} $ (removing the $\kappa $ dependency in notations). In words, the scale of $\cS$ relative to the constant $\kappa$ is the distance at which $\cQ $ becomes a $\kappa$-Lipschitz function. In the present paper, all constants $O()$ are uniform on grids such that $\kappa<1$ and can be in principle written explicitly using following the proofs inside \cite{CLR1,Che20}.

\subsection{S-holomorphic functions and associated primitives}\label{sub:HF-def}
Let us now recall the notion of \emph{s-holomorphic functions}, which was generalized to $s$-embeddings in \cite{Che20}. This concept was originally introduced by Smirnov \cite[Definition 3.1]{Smirnov_Ising} for the critical square lattice and by Chelkak and Smirnov \cite[Definition~3.1]{ChSmi2} for isoradial grids. It plays a crucial role in applying discrete complex analysis techniques to the Ising and dimer models. We recall the general definition of $s$-holomorphic functions as stated in \cite[Definition 2.4]{Che20}. In what follows, $\textrm{Proj}[\cdot, \eta \mathbb{R}]$ denotes the usual projection onto the line $\eta \mathbb{R}$.

\begin{definition}\label{def:s-hol}
A function $F$ defined on a subset of $\Dm(G)$ is called $s$-holomorphic if, for each pair of adjacent quads $z, z’ \in \Dm(G)$ separated by the edge $[\cS(v^\circ(c)); \cS(v^\bullet(c))]$ of $\cS$ attached to the corner $c$, we have
\begin{equation}
\label{eq:s-hol}
\textrm{Pr}[F(z), \eta_c \mathbb{R}] = \textrm{Pr}[F(z’), \eta_c \mathbb{R}].
\end{equation}
\end{definition}

The above definition establishes a direct link between real-valued solutions to \eqref{eq:3-terms} and complex-valued $s$-holomorphic functions. This connection was first presented in \cite[Proposition 2.5]{Che20} and in \cite[Appendix]{CLR1}.
\begin{proposition}\label{prop:shol=3term} Let $\cS=\cS_\cX$ be a proper $s$-embedding and $F$ an $s$-holomorphic on $\Dm(G)$. Given $z\in\Dm(G)$, a corner \mbox{$c\in\Upsilon^\times(G)$} belonging to the quad $z$, one can define the spinor $X$ at $c \in \Upsilon^\times $ by the formula 
\begin{align}
X(c)\ &:=\ |\cS(v^\bullet(c))-\cS(v^\circ(c))|^{\frac{1}{2}}\cdot\Re[\overline{\eta}_c F(z)] \notag\\
 &=\ \Re[\overline{\varsigma}\cX(c)\cdot F(z)]\ =\ \overline{\varsigma}\cX(c)\cdot \mathrm{Proj}[F(z);\eta_c\R].
\label{eq:X-from-F}
\end{align}
The map $c\mapsto X(c) $ satisfies all three terms identities  \eqref{eq:3-terms} around the quad $z$. Conversely given $X:\Upsilon^\times(G)\to\R$ a real valued solution to \eqref{eq:3-terms}, there exists a unique s-holomorphic function $F$ on $\Dm(G)$ such that the identity \eqref{eq:X-from-F} holds.
\end{proposition}
When $F$ and $X$ are linked by \eqref{eq:X-from-F}, one can reconstruct the value of $F(z)$ out of the values of $X$ at any pair of corners $c_{pq}(z)\in\Upsilon^\times(G)$ and the geometry of $\cS$, e.g. using the formula \cite[Corollary 2.6]{Che20}
\begin{equation} \label{eq:F-from-X}
F(z)\ =\ -i\varsigma\cdot\frac{\overline{\cX(c_{01}(z))}\,X(c_{10}(z))-\overline{\cX(c_{10}(z))}\,X(c_{01}(z))} {\Im[\,\overline{\cX(c_{01}(z))}\,\cX(c_{10}(z))}.
\end{equation}

In the $s$-embeddings framework, the behavior of the scaling limit of $s$-holomorphic functions is determined by its local equation and boundary conditions. This can be understood through two different integration procedures at the discrete level. The first is a standard extension of the integration procedure for discrete holomorphic functions to the $s$-embeddings framework, which now accounts for the presence of the origami map. The second is a generalization of Smirnov’s primitive of the square of an $s$-holomorphic function. The former is extensively studied in \cite[Proposition 6.15]{CLR1} and will be useful for deriving the local regularity theory for discrete functions, along with their local equation in the limit. The latter, introduced by Smirnov in \cite{Smirnov_Ising} for the critical square lattice, helps identify a discrete Riemann-Hilbert boundary condition that arises in the Ising context. Let us begin with the primitive $I_{\mathbb{C}}$. Given an $s$-holomorphic function $F$ on $\diamondsuit(G)$, we can define (up to a global additive constant)
\cite[Section 2.3]{Che20}
\begin{equation}\label{eq:def-I_C}
I_{\mathbb{C}}[F]:= \int \big( \overline{\varsigma}Fd\cS + \varsigma \overline{F} d\cQ \big).
\end{equation}
Let $v_{1,2}^{\bullet}, v_{1,2}^{\circ}$ be vertices of the quad $z\in \diamondsuit(G)$. Then one has for $\star \in \{ \bullet, \circ \} $
\begin{equation}
I_{\mathbb{C}}[F](v_{2}^{\star}) - I_{\mathbb{C}}[F](v_{1}^{\star}) = \overline{\varsigma} F(z) [ \cS(v_{2}^{\star}) - \cS(v_{1}^{\star})] + \varsigma \overline{F(z)}[ \cQ(v_{2}^{\star}) - \cQ(v_{1}^{\star})].
\end{equation}
It is possible to extend to origami map to the entire complex plane (see \cite{CLR1}, \cite[Section 2.3]{Che20} or \cite[Section 2.4]{MahPar25a}) which allows to extend the definition \eqref{eq:def-I_C} to the entire complex plane. 

The notion of the primitive of the square $H_X$ can be introduced through a purely combinatorial definition, linked to the Kadanoff-Ceva formalism. This definition does not require any specific embedding into the plane, as long as one works with a spinor $X$ that satisfies \eqref{eq:3-terms}. This generalization of Smirnov’s original work is presented in \cite[Definition 2.8]{Che20}.
\begin{definition}
\label{def:HX-def} Given $X$ a spinor on $\Upsilon^\times(G)$ satisfying \eqref{eq:3-terms}, one can define the function $H_X$ up to a global additive constant on $\Lambda(G)\cup\Dm(G)$ by setting
\begin{equation}
\label{eq:HX-def}
\begin{array}{rcll}
H_X(v^\bullet_p(z))-H_X(z)&:=&X(c_{p0}(z))X(c_{p1}(z))\cos\theta_z, & p=0,1,\\[2pt]
H_X(v^\circ_q(z))-H_X(z)&:=&-X(c_{0q}(z))X(c_{1q}(z))\sin\theta_z,& q=0,1,\\[2pt]
H_X(v^\bullet_p(z))-H_X(v^\circ_q(z))&:=&(X(c_{pq}(z)))^2,
\end{array}
\end{equation}
similarly to~\eqref{eq:cS-def} and~\eqref{eq:cS(z)-def}.
\end{definition}
The consistency of the above definition follows from \eqref{eq:3-terms}. When passing to an $s$-embedding $\cS$ of the graph $(G, x)$, the correspondence between $X$ and $F$, as recalled in Proposition \ref{prop:shol=3term}, allows us to interpret $H_X$ via the $s$-holomorphic function $F$ associated with $X$. More precisely, one can define $H_X$ as in \cite[Equation (2.17)]{Che20}.\begin{equation}
\label{eq:HF-def}
H_F:=\int\Re(\overline{\varsigma}^2F^2d\cS+|F|^2d\cQ)=\int (\Im(F^2d\cS)+\Re(|F|^2d\cQ)),
\end{equation}
on $\Lambda(G) \cup \Dm(G)$. The extension of $\cQ$ described below allows us to extend $H_F$ in a piecewise affine manner to the entire plane. (It is important to note that the extension occurs on each face of the associated $t$-embedding $\cT = \cS$ — see \cite[Proposition 3.10]{CLR1}.) The following lemma establishes the relationship between the definitions \eqref{eq:HX-def} and \eqref{eq:HF-def}, proving that they are, in fact, the \emph{same} function.

\begin{lemma}{\cite[Lemma 2.9]{Che20}} Let $F$ be defined $\Dm(G)$ and $X$ be defined on $\Upsilon^\times(G)$ related by the identity~\eqref{eq:X-from-F}. Then, the functions $H_F$ and~$H_X$ coincide up to a global additive constant.
\end{lemma}
If $\cS$ is an isoradial grid, the origami map $\cQ$ is constant on both $G^\bullet$ and $G^\circ$ (since all the edges of the quads $\cS^{\diamond}(z)$ have the same length). Therefore, $H_F$ is the primitive of $\Im[F^2 d\cS]$, which recovers the original definition given in \cite[Section~3.3]{ChSmi2}.

\subsection{Regularity theory for s-holomorphic functions}\label{sub:regularity}
In this short subsection we recall in a concise way the regularity theory of s-holomorphic functions, which was developed in the dimer context in \cite[Section 6.5]{CLR1} and the Ising context in \cite[Theorem 2.18]{Che20}. That regularity theory can be summarized as some Harnack type inequality for s-holomorphic functions that controls its maximum via one of its primitives $I_{\mathbb{C}}[F]$ or $H_{F} $ defined in the previous section, except in some pathological scenario where the discrete functions blow up exponentially fast in $\delta^{-1} $.

\begin{theorem}{\cite[Theorem 6.17]{CLR1} and \cite[Theorem 2.18]{Che20} } \label{thm:F-via-HF} For each fixed $\kappa<1$, there exist constants $\gamma_0=\gamma_0(\kappa)>0$ and \mbox{$C_0=C_0(\kappa)>0$} such that the following alternative holds. Let $F$ be a s-holomorphic function defined in a ball of radius $r$ drawn over an s-embedding $\cS$ satisfying the assumption $\textup{Lip}(\kappa,\delta)$ and some $\alpha \in \mathbb{T}$. Then, provided that $r\ge\cst\cdot\delta$ for a constant depending only on $\kappa$, one has the two following alternatives:

For the integration procedure \eqref{eq:def-I_C} 
\[
\begin{array}{rcl}
\text{either}\ \max_{\{z:\cS(z)\in B(u,\frac{1}{2}r)\}}|F|&\le& C_0r^{-1}\cdot |\osc_{\{v:\cS(v)\in B(u,r)\}}I_{\alpha\R}[\overline{\varsigma}F]|,\\[4pt]
\text{or}\ \max_{\{z:\cS(z)\in B(u,\frac{3}{4}r)\}}|F|&\ge& \exp(\gamma_0r\delta^{-1})\cdot C_0r^{-1}|\osc_{\{v:\cS(v)\in B(u,r)\}}I_{\alpha\R}[\overline{\varsigma}F]|
\end{array}
\]

For the integration procedure \eqref{eq:HF-def} 
\[
\begin{array}{rcl}
\text{either}\ \max_{\{z:\cS(z)\in B(u,\frac{1}{2}r)\}}|F|^2&\le& C_0r^{-1}\cdot |\osc_{\{v:\cS(v)\in B(u,r)\}}H_{F}|,\\[4pt]
\text{or}\ \max_{\{z:\cS(z)\in B(u,\frac{3}{4}r)\}}|F|^2&\ge& \exp(\gamma_0r\delta^{-1})\cdot C_0r^{-1}|\osc_{\{v:\cS(v)\in B(u,r)\}}H_{F}|
\end{array}
\]
\end{theorem}

Under the assumption \Unif\,, which is a the simplest toy example where the couple of assumptions \LipKd\, and \ExpFatt\, hold, the second alternative of the theorems recalled above are impossible. In particular, this ensures that (see e.g. \cite[Remark 2.12]{Che20}) that if the s-holomorphic functions $F^\delta$ are uniformly bounded in an open set $U$, they form a pre-compact family (in the topology of the uniform convergence on compacts) as $\delta\to 0$. In particular, functions are $\beta(\kappa)$-H\"older (see \cite[Theorem 2.18]{Che20}) starting at a scale comparable to $\delta$. More generally, under the general assumptions \LipKd\, and \ExpFatt\,, bounded s-holomorphic functions are $\beta(\kappa)$-H\"older starting at a scale $O(\max(\delta,\rho(\delta))$.

\subsection{Constructing full-plane fermions and identifying their branching structure}\label{sub:2-point-fermion}

Let us now construct the full-plane energy Kadanoff-Ceva correlator as the limit of the same object in bounded regions. First, fix a connected box $\Lambda_R$ and consider the Ising model on $\Lambda_R$ with wired boundary conditions, meaning that a single spin is attached to the outer face. Fix a corner $
q \in \Upsilon$, whose two lifts in $\Upsilon^{\times}$ are denoted by $q^+$ and $q^-$. We denote by $u^\circ(q) \in G^\circ$ and $v^\bullet(q) \in G^\bullet$ the vertices of $\Lambda(G)$ adjacent to $q$. Now, using the formalism introduced in \eqref{eq:KC-fermions} with $\varpi = { u^\circ(q), v^\bullet(q) }$, we define
\begin{equation}\label{eq:def-KC-energy}
	X_{R}^{(q)}(c)=\langle \chi_c \chi_q \rangle^{\textrm{(w)}}_{\Lambda_R}:=\mathbb{E}^{\textrm{(w)}}\big[ \sigma_{u^{\circ}(c)}\sigma_{u^{\circ}(q)} \mu_{v^{\bullet}(c)} \mu_{v^{\bullet}(q)} \big].
\end{equation}
The correlator $c \mapsto X_{R}^{(q)}(c)$ is a spinor that satisfies the propagation equation \eqref{eq:3-terms} \emph{everywhere} in $\Upsilon^{\times}{(q)} := \Upsilon^{\times}{[u^{\circ}(q), v^{\bullet}(q)]}$, the double cover $\Upsilon^\times_{(q)}$ that branches everywhere except around $u^{\circ}(q)$ and $v^{\bullet}(q)$. As recalled in \cite[Figure 6]{CIM-universality} (see also Figure \ref{fig:identification-double-covers}), it is possible to identify the two double covers $\Upsilon^{\times}$ and $\Upsilon^\times_{(q)}$, except at the corner $q$, where the nearby connections to $q^{\pm}$ must be swapped on one of the quads containing the edge $[u^{\circ}(q),v^{\bullet}(q)]$.

More precisely, the corner $q^+ \in \Upsilon^{\times}{(q)}$  is chosen so that the branching structures of $\Upsilon^{\times}{(q)}$ and $\Upsilon^{\times}$ coincide around the quad $z^{+}_{q}$. In particular, when looking  the correlator $X^{(q)}_{R}$ as a function on the double-cover $\Upsilon^{\times} $, it satisfies the propagation equation around almost all quads, including $z_q^+$, but in general the propagation equations around the quad $z_q^-$. More precisely, around the quad $z_q^-$, each identity of the form \eqref{eq:3-terms} involving the corners $q^{\pm}$ fails, \emph{while it would be true if one had assigned the values } $X^{(q)}_{R}(q^\pm) = \mp1$ \emph{instead of the true values} $X^{(q)}_{R}(q^\pm) = \pm1$. This rather simple observation is the building block of the entire paper and will be used to deform continuously embeddings while changing continuously Ising weights.

It is standard to consider a full-plane Ising model, obtained as a subsequential limit of models on an increasing sequence of bounded domains, for a weighted planar graph $(S, (x_e))_{e\in E}$. We also fix an arbitrary embedding $S$ of $(S, (x_e))_{e\in E}$ into the plane, which is defined up to homeomorphism. Recall that, in general, for any corner $c \in \Upsilon$, the disorder random variable can be written as $\mu_{v^\bullet(c)}\mu_{v^\bullet(q)}=\prod_{e\in \gamma^\bullet_{(q,c)}}x_e^{\varepsilon_e} $, where the product is taken along a disorder line $\gamma^\bullet_{(q,c)}$ connecting $v^{\bullet}(c)$ to $v^{\bullet}(q)$, and $\varepsilon_e$ represents the energy density at edge $e$. Using the identity $x_e^{\varepsilon_e} = \frac{1}{2}\big( (x_e - x_e^{-1})\varepsilon_e  + (x_e + x_e^{-1})) $, one can expand the bounded correlator $X_{R}^{(a)}$, following \cite[Lemma 4.1]{CIM-universality}, as:
\begin{equation}\label{eq:expansion-fermion}
	X_{R}^{(q)}(c) = \sum_{\iota \in I(q,c)} a_{\iota} \mathbb{E}^{(\textrm{w})}_{\Lambda_R}[ \sigma_{A_\iota}],
\end{equation}
where the sum is finite linear combination of spin-correlations where the products $\sigma_{A_\iota} = \prod_{j\in \iota}  \sigma_j $, only involve spins along the disorder line $\gamma^\bullet_{(q,c)}$. In particular, the indices $\iota \in I(q,c)$ do not depend on $R$ (provided that $R$ is chosen sufficiently large), and the sum contains only a bounded number of terms as $R \to \infty$. Since $\mathbb{E}^{(\textrm{w})}_{\Lambda_R}[ \sigma_{A_\iota}]$ is decreasing as $R \to \infty$, the expression $X_{R}^{(q)}(c)$ has a well-defined limit $X_{S}^{(q)}(c)$ as $R$ becomes infinite. In the present paper, the strong box-crossing property recalled in Theorem \ref{thm:RSW-s-embeddings} holds at a large enough scale for all s-embedding we work with. This ensures the uniqueness of the full-plane Gibbs measure for all Ising model studied here and makes the constructions of the fermion $X_{S}^{(q)}$ independent from the wired boundary conditions used in bounded regions.

\begin{figure}
\hskip -0.10\textwidth \begin{minipage}{0.325\textwidth}
\includegraphics[clip, width=1.2\textwidth]{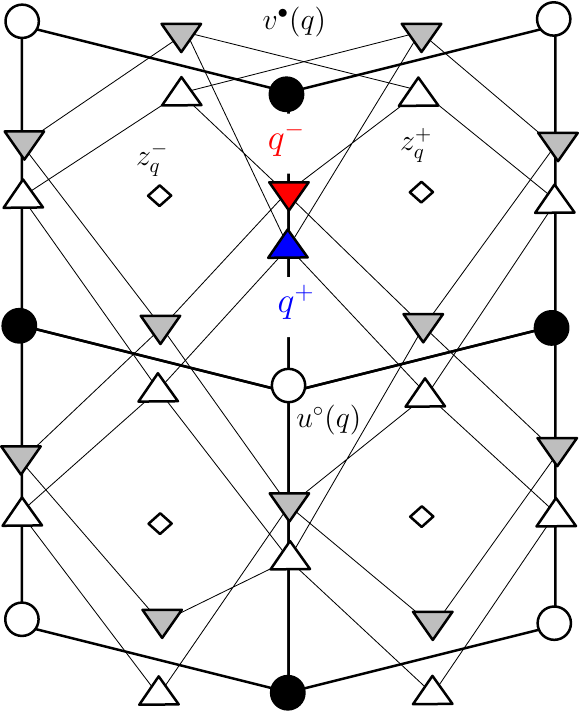}
\end{minipage} \hskip 0.15\textwidth 
\begin{minipage}{0.33\textwidth}
\includegraphics[clip, width=1.2\textwidth]{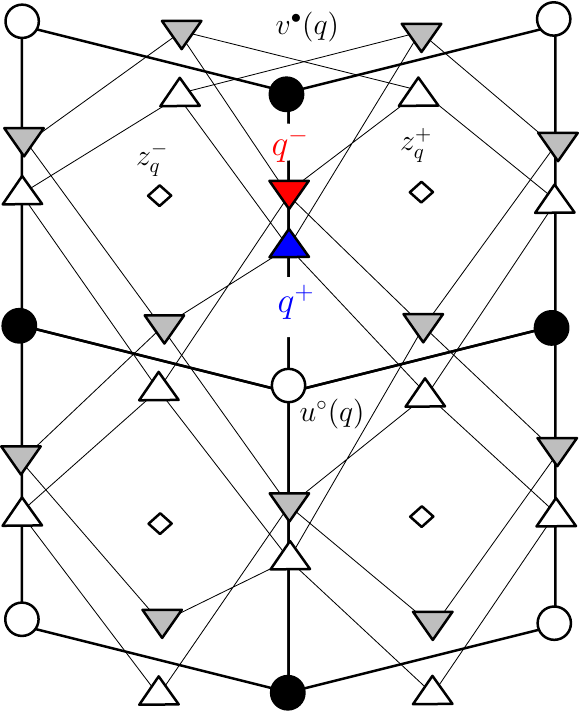}
\end{minipage}
\caption{(Left) The double cover $\Upsilon^\times$ branching around each vertex of $G^\bullet \cup G^{\circ} \cup \diamondsuit(G)$ (Right) The double cover $\Upsilon^\times_{(q)}$ that branches everywhere except around $v^{\bullet}(q)$ and $u^{\circ}(q)$. Those two double covers can be identified with each other away from $q$. The corner $q^+$ is chosen so that the two double-covers have the same branching structure around the quad $z_q^+$.  This figure is similar to \cite[Figure 6]{CIM-universality}.}
\label{fig:identification-double-covers}
\end{figure}

\subsection{Rate of growth of full-plane fermions via the geometry of the embedding}\label{sub:estimate-growth-fermion}

The key ingredient to understand in a quantitative fashion the impact of the deformation process via the ODE \eqref{eq:ODE-embedding-general} or the SDE  \eqref{eq:SDE-embedding} is to get some quantitative estimates regarding the decay of two points fermion $\langle \chi_{\frak{p}} \chi_{\frak{p}'} \rangle_{S} $ associated to the Ising model on $(S,(x_e)_{e\in S})$. The main Theorem of \cite{MahPar25a} together with its proof ensures that one can obtain sharp (up to constant) estimates on the decay of the two-points fermions via the geometry of an associated proper $s$-embedding $\cS$. In the case of a grid satisfying \Unif\, there exists $\Theta$, only depending on $r_0,\theta_0$ such that for any pair of corners $\frak{p}\neq \frak{p}'$, one has

\begin{equation}\label{eq:bound-correlator-distance}
	 |\langle \chi_{\frak{p}} \chi_{\frak{p}'}\rangle_{S}| \leq |\cY_{\cS}(\frak{p})|\cdot |\cY_{\cS}(\frak{p}') | \cdot \frac{\Theta }{|\cS(\frak{p})-\cS(\frak{p}')|}.
\end{equation}

This observation is a direct consequence of the similar estimate when $|\cS(\frak{p})-\cS(\frak{p}')| \geq C_0\cdot \delta $, obtained for the associated complexified fermion within the proof of \cite[Theorem 1.2]{MahPar25a}. Moreover, since all the angles in $\cS$ are bounded from below by $\theta_0$, the formula \eqref{eq:theta-from-S} ensures that all coupling constants in a grid satisfying \Unif\, are bounded away from $0$ and $1$, again depending only on $r_0,\theta_0$. Using the finite energy property, this allows to extend \eqref{eq:bound-correlator-distance} even in the case where $|\cS(\frak{p})-\cS(\frak{p}')| \leq C_0\cdot \delta$.

\subsection{The argument principle to prove properness of the embedding}\label{sub:argument-principle}
The deformation process presented in Section \ref{sec:deformation} is purely algebraic and doesn't provide any statement regarding properness of the deformed embedding. One of the main outputs of the present paper is the construction the $s$-embeddings $ (\cS^{(t)})_{t\geq 0}$ of massive and near-critical i.i.d. models in Sections \ref{sec:massive-deformation} and \ref{sec:iid-deformation}, via \eqref{eq:ODE-embedding}. Still, if one wants to apply embeddings techniques to deduce statements about the associated Ising models, one should check that all embeddings $ (\cS^{(t)})_{t\geq 0}$ remain \emph{proper}, at least for small enough times $t$. We explain here how one can easily check if an $s$-embedding is locally proper, meaning that all faces \emph{do not overlap}. Properness is a priori a non-trivial statement to deduce looking only at the generator $\cY$ associated to an s-embedding $\cS$. More importantly, it looks to be a very local statement that should be checked for any pair of different faces. Fortunately, there exists a standard argument principle that allows to rewrite this very local statement as a more global one, counting the winding number of a concrete curve. The argument principle we present here is very specific to the use of Ising/dimers embeddings, as its proof is based on harmonicity of $s/t$-holomorphic functions on the so-called S-graphs (see e.g. \cite[Section 4]{CLR1} or \cite[Section 2.3]{Che20} for more details.

We recall here simple facts coming from \cite[Section 4.1]{CLR2} and limit ourselves to an $s$-embedding the square lattice to simplify the presentation. All statements presented below apply verbatim in a more general context. Denote the \emph{combinatorial} box $(\Lambda^{\textrm{comb}}_n,(x_e)_{e\in \Lambda_n})$ of size $2n$ centred at the origin, equipped with the Ising weights $(x_e)_{e\in \Lambda_n})$. We denote by $S$ its standard embedding on the square lattice (meaning here that Ising weights \emph{don't} have the interpretation \eqref{eq:theta-from-S}. Denote by $\gamma^{S}_n$ the outer-boundary of $\Lambda^{\textrm{comb}}_n$ oriented in the positive direction, which is a simple curve. Let $\cS$ be an s-embedding of $(\Lambda^{\textrm{comb}}_{2n},(x_e)_{e\in \Lambda_{2n}})$. The argument principle, recalled in \cite[Section 4.1]{CLR2}, reads as follows:
\begin{center}
	If the path $ \cS(\gamma^{S}_{2n})$, oriented in the positive direction (which is not a priori a simple curve), only winds \emph{once} around the image of $\cS(\gamma^{S}_{n})$, then $\cS(\Lambda^{\textrm{comb}}_n)$ is a proper $s$-embedding.
\end{center}
This comes as a byproduct of the following facts:
\begin{itemize}
	\item All the faces of $\cS(\Lambda^{\textrm{comb}}_{2n})$  are oriented in the same direction. This is a consequence of the maximum-principle for harmonic functions on T-graphs \cite[Section 4]{CLR1} together with the fact that the image of each face $f\in \Lambda^{\textrm{abs}}_{2n}$ of the associated $t$-embedding is mapped to a convex polygon.
	\item The images of two different faces $f_1\neq f_2 \in \Lambda^{\textrm{abs}}_n $ don't overlap. Either way, if the faces $f_1,f_2$ overlap around an interior point $z\in \cS(\Lambda^{\textrm{comb}}_n)$, then one has 
\begin{equation}\nonumber
	1=\textrm{wind}(\cS(\gamma^{S}_{2n}),z)= \sum\limits_{f \in  \Lambda^{\textrm{comb}}_{2n}}\textrm{wind}(\cS(f),z) \geq \textrm{wind}(\cS(f_1),z)+\textrm{wind}(\cS(f_2),z) \geq 2,
\end{equation}
as $\cS(\gamma^{S}_{2n})$ only winds once around $z$ (first equality) and the orientation of all faces are the same, which implies that $\textrm{wind}(\cS(f),z)\geq 0$ for all faces of $f\in \Lambda_{2n}$.
\end{itemize}
Let us additionally the for each face $f$, its image $\cS(f)$ is a tangential quadrilateral, whose sum of opposing edge-length are equal. Therefore, it is straightforward that the images of the boundary segments of the face $f$ \emph{cannot} intersect except at the image of the vertices of $f$.

\section{Deforming an s-embedding using differential equations}\label{sec:deformation}
The goal of this section is to give some rigorous meaning to the overall philosophy of the present paper. Assume one starts with a given s-embedding $\cS$ of a graph $(S,(x_e)_{e\in E})$. It is tempting to say that when one moves continuously the family of Ising weights $(x^{(t)}_e)_{e\in E }$ with some continuous time parameter $t$, one should be able to construct a continuous family $\cS^{(t)}$ of $s$-embeddings of the corresponding Ising model $(S,(x^{(t)}_e)_{e\in E})$. This construction is a priori non-trivial, as finding an $s$-embedding $\cS^{(t)}$ requires finding a vector in the kernel of some linear system given by the Ising weights $(x^{(t)}_e)_{e\in E }$. It is not clear that the kernel becomes non-degenerate as time evolves, and even less clear that one can find a vector in the kernel in a continuous manner. To simplify the understanding, we work in the section in the case where $S$ is the square lattice $\mathbb{Z}^2$, equipped with a given set of weights $(x_e)_{e\in G}$. Still, the whole reasoning applies verbatim to general $s$-embeddings. The corners of the square lattice can be splitted into $4$ subsets, depending on their 'geographic' position around vertical edges. Those relative position are denoted respectively - with the natural geographic interpretation- 'north-east' (NE), 'north-west' (NW), 'south-east' (SE) and 'south-west' (SW). In what follows, the SE corners are labeled with the letter $a$, the NE corners are labeled with the letter $b$, the, the NW corners are labeled with the letter $c$ and the SW corners are labeled with the letter $d$.

\subsection{Modifying one Ising weight using Kadanoff-Ceva mismatches}\label{sub:modifying-one-weight}
Let us start with the simplest possible deformation, aiming to construct a new s-embedding $\hat{\cS}$ of a weighted graph $(S,(\hat{x}_e)_{e\in E})$ starting from an $s$-embedding $\cS$ of $(S,(x_e)_{e\in E})$, in the case where the weights $(x_e)_{e\in E}$ and $(\hat{x}_e)_{e\in E}$ only differ for one coupling constant attached to the edge $e_0$. We use here the existence of fermionic correlator that satisfy \eqref{eq:3-terms} almost everywhere on $\Upsilon^{\times}$ but around the edge $e_0$. The original s-embedding $\cS$ of $(S,(x_e)_{e\in E})$ is attached to a propagator $\cY = d\cS $ following Definition \ref{def:cS-def}, meaning that everywhere in $\Upsilon^{\times}$, the fermion $\cY$ satisfies \eqref{eq:3-terms}. From a practical stand-point, around  each quad $z_k$, whose combinatorial weight is parameterised by $x_k=\tan \frac{\theta_k}{2}$ following \eqref{eq:x=tan-theta} , one has the identities (and similarly for the other lifts in the double-cover) for the original embedding $\cY$
\begin{align*} 
\cY(a_k^+) - \cY(d_k^+) \cos(\theta_k) + \cY(b_k^+)\sin(\theta_k) &=  0 & (\textrm{A}_k), \\ 
\cY(b_k^+) - \cY(c_k^+) \cos(\theta_k) + \cY(a_k^+)\sin(\theta_k) &=  0 & (\textrm{B}_k),\\ 
\cY(c_k^+) - \cY(b_k^+) \cos(\theta_k) - \cY(d_k^+)\sin(\theta_k) &=  0 & (\textrm{C}_k),\\ 
\cY(d_k^+) - \cY(a_k^+) \cos(\theta_k) - \cY(c_k^+)\sin(\theta_k) &=  0  & (\textrm{D}_k),
\end{align*}
where the corners are labeled as in Figure \ref{fig:labeling-fermion-mismatch}. Recall that we want to change the Ising weight from $x_{e_0}=\tan \frac{\theta_0}{2}$ to $\hat{x}_{e_0}=\tan \frac{\hat{\theta}_0}{2}$ around $e_0$. Consider two full plane fermionic correlator $X^{(a_0)}_{\hat{S}}$ and $X^{(c_0)}_{\hat{S}}$, defined in Section \ref{sub:2-point-fermion}
as \emph{combinatorial objects} for the \emph{new} Ising model $(S,(\hat{x}_e)_{e\in G})$. The fermionic correlators are normalised such that for $q \in \{ a_0,c_0 \} $, one has
\begin{itemize}
\item $X^{(q)}_{\hat{S}}(q_0^+)=1$
\item $X^{(q)}_{\hat{S}} $ is a spinor that satisfies the propagation equation \eqref{eq:3-terms} \emph{everywhere} in the double-cover $\Upsilon^\times_{(q)} $ that branches everywhere except around $u^{\circ}(q)$ and $v^{\bullet}(q)$, the double cover represented on the right of Figure \ref{fig:identification-double-covers}.
\end{itemize}
For $q \in \{ a_0,c_0 \} $, the labeling of the corner $q_0^+$ around the quad  $z_0$ is made such that $z_{q_0^-}=z_0$. This means that when identifying the double-covers $\Upsilon^\times $ and $\Upsilon^\times_{(q_0)} $, the branching structures \emph{don't match} match around $z_0$ but \emph{do match} around the other quad containing $q_0$. Given the identification of double-covers and the fact that $x_k = \hat{x}_k$ for any $k\neq 0 $, then the correlator $ \frak{p} \mapsto X^{(q)}_{\hat{S}}(\frak{p})$ satisfies the propagation equation \eqref{eq:3-terms} around the quad $z_k$, and for $k\neq 0 $ one has
\begin{align*} \label{eq:original-embedding}
X^{(q)}_{\hat{S}}(a_k^+) - X^{(q)}_{\hat{S}}(d_k^+) \cos(\theta_k) + X^{(q)}_{\hat{S}}(b_k^+)\sin(\theta_k) &=  0 & (\widetilde{\textrm{A}}_k^{(q)}), \\ 
X^{(q)}_{\hat{S}}(b_k^+) - X^{(q)}_{\hat{S}}(c_k^+) \cos(\theta_k) + X^{(q)}_{\hat{S}}(a_k^+)\sin(\theta_k) &=  0 & (\widetilde{\textrm{B}}_k^{(q)}),\\ 
X^{(q)}_{\hat{S}}(c_k^+) - X^{(q)}_{\hat{S}}(b_k^+) \cos(\theta_k) - X^{(q)}_{\hat{S}}(d_k^+)\sin(\theta_k) &=  0 & (\widetilde{\textrm{C}}_k^{(q)}),\\ 
X^{(q)}_{\hat{S}}(d_k^+) - X^{(q)}_{\hat{S}}(a_k^+) \cos(\theta_k) - X^{(q)}_{\hat{S}}(c_k^+)\sin(\theta_k) &=  0  & (\widetilde{\textrm{D}}_k^{(q)}).
\end{align*}
To complete the understanding of the local relations of $X^{(q)}_{\hat{S}}$ on $\Upsilon^{\times}$, one needs to carefully look at them around the quad $z_0$. Let us be more specific, presenting more carefully the case of $q=a_0$. We know that $\frak{p}\mapsto X^{(a_0)}_{\hat{S}}(\frak{p}) $ is a spinor that satisfies the propagation equation \eqref{eq:3-terms} everywhere on the double cover $\Upsilon^\times_{(a_0)}$. As an example, this yields that
\begin{equation}
	X^{(a_0)}_{\hat{S}}(a_0^+)+X^{(a_0)}_{\hat{S}}(d_0^+)\cos (\hat{\theta}_0)-X^{(a_0)}_{\hat{S}}(b_0^+)\sin(\hat{\theta}_0)=0.
\end{equation}
Using the identification between $\Upsilon^\times_{(a_0)}$ and $\Upsilon^\times$ recalled in Figure \ref{fig:labeling-fermion-mismatch}, one can rewrite the local relations but this time on $\Upsilon^{\times}$. As an example
\begin{align*}
	&X^{(a_0)}_{\hat{S}}(a_0^+) - X^{(a_0)}_{\hat{S}}(d_0^+) \cos(\hat{\theta}_0) +X^{(a_0)}_{\hat{S}}(b_0^+)\sin(\hat{\theta}_0) \\
	&= 2X^{(a_0)}_{\hat{S}}(a_0^+) + \Big[ - 2X^{(a_0)}_{\hat{S}}(a_0^+) +X^{(a_0)}_{\hat{S}}(a_0^+)  
   - X^{(a_0)}_{\hat{S}}(d_0^+) \cos(\hat{\theta}_0) + X^{(a_0)}_{\hat{S}}(b_0^+)\sin(\hat{\theta}_0) \Big]\\
   &  = 2X^{(a_0)}_{\hat{S}}(a_0^+) - \Big[ X^{(a_0)}_{\hat{S}}(a_0^+)  
   + X^{(a_0)}_{\hat{S}}(d_0^+) \cos(\hat{\theta}_0) - X^{(a_0)}_{\hat{S}}(b_0^+)\sin(\hat{\theta}_0) \Big]\\
   & = 2X^{(a_0)}_{\hat{S}}(a_0^+) +0 = 2.
\end{align*}
More generally, when evaluating the local relations of $X^{(a_0)}_{\hat{S}}$ around the $z_0$ on $\Upsilon^{\times}$ instead of doing it on $\Upsilon^{\times}_{(a_0)}$, the propagation equation \eqref{eq:3-terms} \emph{would hold if one had} $X^{(a_0)}_{\hat{S}}(a_0^\pm)=\mp1$ instead of $X^{(a_0)}_{\hat{S}}(a_0^\pm)=\pm 1$. This allows to deduce that 
\begin{align*} 	
X^{(a_0)}_{\hat{S}}(a_0^+) - X^{(a_0)}_{\hat{S}}(d_0^+) \cos(\hat{\theta}_0) + X^{(a_0)}_{\hat{S}}(b_0^+)\sin(\hat{\theta}_0) &=  2, &(\textrm{A}_0') \\ 
X^{(a_0)}_{\hat{S}}(b_0^+) - X^{(a_0)}_{\hat{S}}(c_0^+) \cos(\hat{\theta}_0) + X^{(a_0)}_{\hat{S}}(a_0^+)\sin(\hat{\theta}_0) &=  2\sin(\hat{\theta}_0)  ,&(\textrm{B}_0') \\ 
X^{(a_0)}_{\hat{S}}(c_0^+) - X^{(a_0)}_{\hat{S}}(b_0^+) \cos(\hat{\theta}_0) - X^{(a_0)}_{\hat{S}}(d_0^+)\sin(\hat{\theta}_0) &=  0  ,&(\textrm{C}_0')\\ 
X^{(a_0)}_{\hat{S}}(d_0^+) - X^{(a_0)}_{\hat{S}}(a_0^+) \cos(\hat{\theta}_0) - X^{(a_0)}_{\hat{S}}(c_0^+)\sin(\hat{\theta}_0) &=  -2\cos(\hat{\theta}_0)  .&(\textrm{D}_0')
\end{align*}
The same computation for the fermion $X^{(c_0)}_{\hat{S}}$around $z_0$ yields
\begin{align*} 
X^{(c_0)}_{\hat{S}}(a_0^+) - X^{(c_0)}_{\hat{S}}(d_0^+) \cos(\hat{\theta}_0) + X^{(c_0)}_{\hat{S}}(b_0^+)\sin(\hat{\theta}_0) &=  0, &(\textrm{A}_0'')\\ 
X^{(c_0)}_{\hat{S}}(b_0^+) - X^{(c_0)}_{\hat{S}}(c_0^+) \cos(\hat{\theta}_0) + X^{(c_0)}_{\hat{S}}(a_0^+)\sin(\hat{\theta}_0) &=  -2\cos(\hat{\theta}_0), &(\textrm{B}_0'')\\ 
X^{(c_0)}_{\hat{S}}(c_0^+) - X^{(c_0)}_{\hat{S}}(b_0^+) \cos(\hat{\theta}_0) - X^{(c_0)}_{\hat{S}}(d_0^+)\sin(\hat{\theta}_0) &=  2, &(\textrm{C}_0'')\\ 
X^{(c_0)}_{\hat{S}}(d_0^+) - X^{(c_0)}_{\hat{S}}(a_0^+) \cos(\hat{\theta}_0) - X^{(c_0)}_{\hat{S}}(c_0^+)\sin(\hat{\theta}_0) &=  -2\sin(\hat{\theta}_0) .&(\textrm{D}_0'')
\end{align*}

\begin{figure}
\hskip -0.10\textwidth \begin{minipage}{0.325\textwidth}
\includegraphics[clip, width=1.6\textwidth]{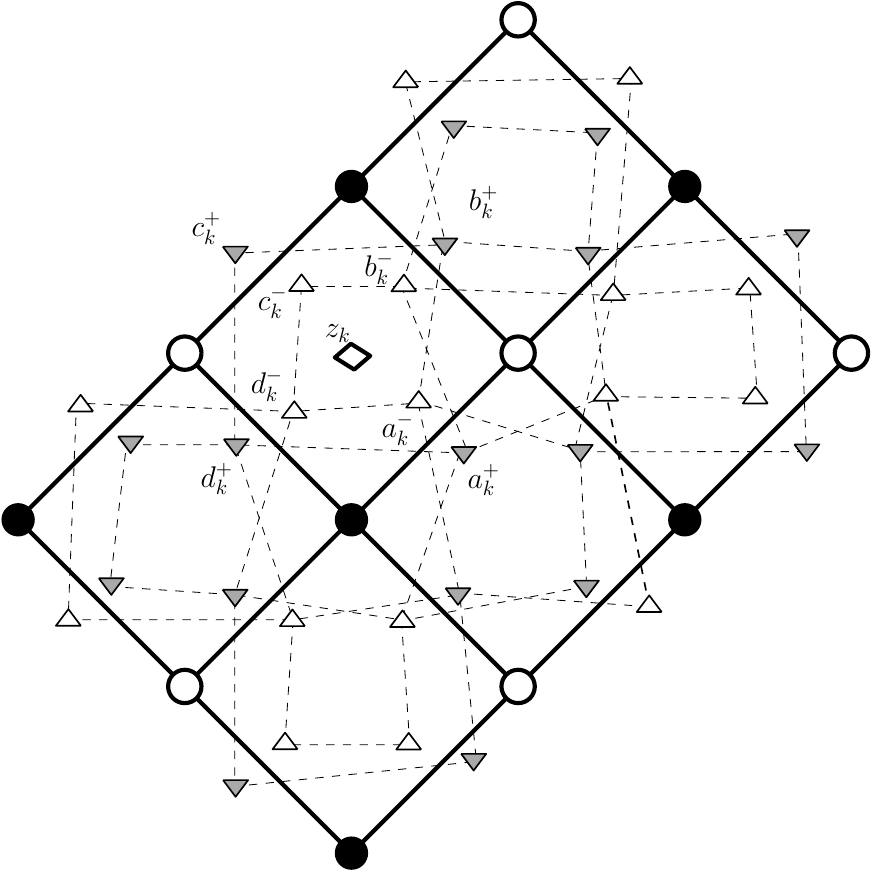}
\end{minipage} \hskip 0.15\textwidth 
\begin{minipage}{0.33\textwidth}
\includegraphics[clip, width=1.6\textwidth]{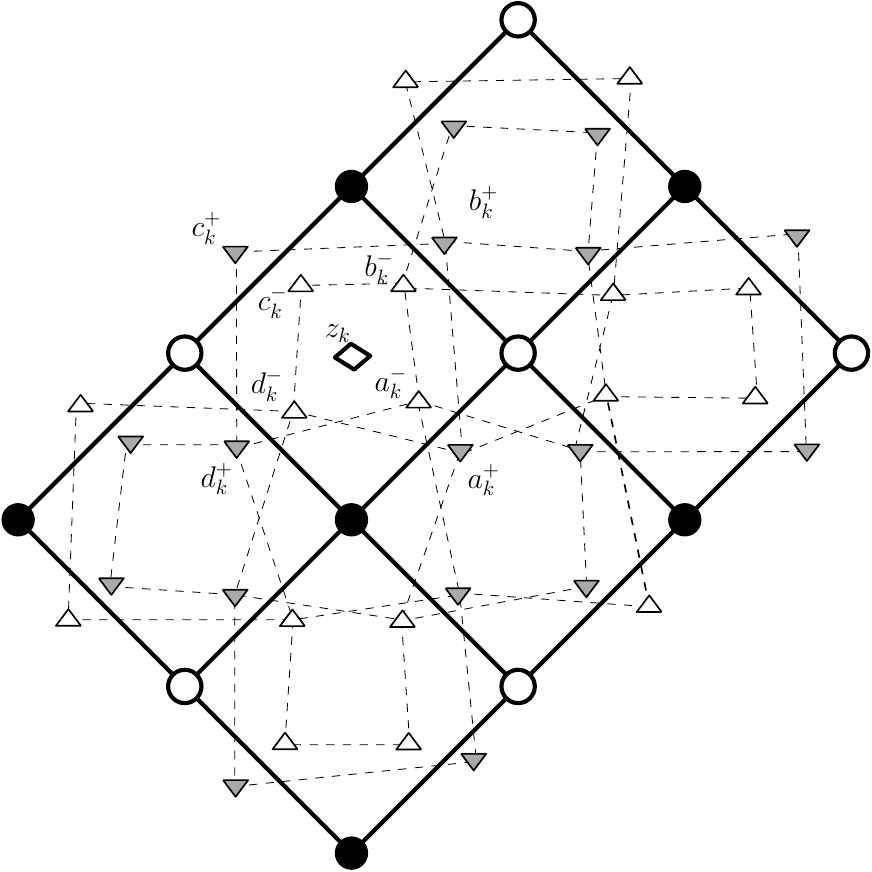}
\end{minipage}
\caption{For both pictures, dashed lines correspond to neighboring relation on double covers. The South-Eastern corners are labeled by $a$, the North-Eastern corners are labeled by $b$, the North-Western corners are labeled by $c$ and the South-Western corners are labeled by $d$. (Left) The double cover $\Upsilon^\times$ branching around each vertex of $G^\bullet \cup G^{\circ} \cup \diamondsuit(G)$ (Right) The double cover $\Upsilon^\times_{(a_k)}$ that branches everywhere except around $v^{\bullet}(a_k)$ and $u^{\circ}(a_k)$. We identify the two double covers away from the quad $z_k$ and swap two connection in the quad $z_k$.  }
\label{fig:labeling-fermion-mismatch}
\end{figure}

We now take advantage of the respective mismatches in $ X^{(a_0)}_{\hat{S}}$ and $ X^{(c_0)}_{\hat{S}}$ to construct an explicit solution to \eqref{eq:3-terms} for the weighted graph $(\hat{S},(\hat{x}_e)_{e\in E}) $. More concretely, we are looking for two complex numbers $y_{a_0}$ and  $y_{c_0}$ such that the propagator $\frak{p} \mapsto \cY(\frak{p})+y_{a_0}\cdot X^{(a_0)}_{\hat{S}}(\frak{p}) + y_{c_0}\cdot X^{(c_0)}_{\hat{S}}(\frak{p})$ satisfies all 3 terms identities \eqref{eq:3-terms} for the weights $(\hat{x}_e)_{e\in E}$. Let us start with a trivial observation for any quad $z\neq z_0$. As recalled above in the equations labeled  $\textrm{A}_k, \widetilde{\textrm{A}}_k^{(q)}, \textrm{B}_k, \widetilde{\textrm{B}}_k^{(q)} \textrm{C}_k, \widetilde{\textrm{C}}_k^{(q)} \textrm{D}_k, \widetilde{\textrm{D}}_k^{(q)}$, all three fermions $\cY$, $X^{(a_0)}_{\hat{S}}$ and $X^{(c_0)}_{\hat{S}}$ satisfy the propagation equation \eqref{eq:3-terms} around $z_k$ for the Ising weight $x_k=\hat{x}_k=\tan \frac{\theta_k}{2}$. Therefore, any linear combination of those three fermions still satisfies \eqref{eq:3-terms} around $z_k$ for the Ising weight $\tan \frac{\hat{\theta}_k}{2}$, leaving complete freedom on the the choice of the coefficients $y_{a_0}$ and $y_{c_0}$. We now prove that one \emph{can} tune the coefficients $y_{a_0}$ and $y_{c_0}$ such that $\cY+y_{a_0}\cdot X^{(a_0)}_{\hat{S}} + y_{c_0}\cdot X^{(c_0)}_{\hat{S}}$ satisfies \eqref{eq:3-terms} around the quad $z_0$ but this time with a coupling constant $\hat{x}_0= \tan \frac{\hat{\theta}_0}{2}$.

Denote by $\widehat{\textrm{A}}_0$ the analog of the equation $\textrm{A}_0$ corresponding to an abstract angle $\hat{\theta}_0 $ in \eqref{eq:x=tan-theta}, and similarly for $\widehat{\textrm{B}}_0$,$\widehat{\textrm{C}}_0$ and $\widehat{\textrm{D}}_0$. Then $\cY+y_{a_0}\cdot X^{(a_0)}_{\hat{S}} + y_{c_0}\cdot X^{(c_0)}_{\hat{S}}$ satisfies $\widehat{\textrm{A}}_0,\widehat{\textrm{B}}_0,\widehat{\textrm{C}}_0$ and $\widehat{\textrm{D}}_0$ if and only if the coefficients $y_{a_0}$ and $y_{c_0}$ are solution to the linear system 
 \begin{align}\label{eq:condition-tuning-coefficients}
    \begin{bmatrix}
2  & 0  \\
2\sin \hat{\theta}_0 &  -2\cos \hat{\theta}_0  \\
0  & 2  \\
-2\cos \hat{\theta}_0 &  -2\sin \hat{\theta}_0  
\end{bmatrix} &\times  \begin{bmatrix}
           y_{a_0} \\
           y_{c_0} \\
         \end{bmatrix}
         &= \begin{bmatrix}
           -\cY(a_0^+)+ \cY(d_0^+)\cos \hat{\theta}_0   - \cY(b_0^+) \sin  \hat{\theta}_0   \\
           -\cY(b_0^+)+ \cY(c_0^+) \cos \hat{\theta}_0   - \cY(a_0^+) \sin  \hat{\theta}_0    \\
           -\cY(c_0^+)+ \cY(b_0^+)  \cos \hat{\theta}_0   +  \cY(d_0^+) \sin  \hat{\theta}_0   \\
           -\cY(d_0^+)+ \cY(a_0^+)  \cos \hat{\theta}_0   +\cY(c_0^+) \sin  \hat{\theta}_0 
         \end{bmatrix},
 \end{align}
 where the vector on the RHS of \eqref{eq:condition-tuning-coefficients} is denoted $V$. Using the relations $\textrm{A}_0,\textrm{B}_0,\textrm{C}_0$ and $\textrm{D}_0$ that link values of $\cY$ via the abstract angle $ \theta_0$ allows to rewrite the values of the fermion at $\cY(a_0^+),\cY(b_0^+),\cY(c_0^+),\cY(d_0^+) $ as some $\pm \cos(\theta_0),\sin(\theta_0) $ linear combination of their neighbours. Factoring respectively by $ [\cos(\hat{\theta}_0) -\cos(\theta_0)] $ and $[\sin(\hat{\theta}_0) -\sin(\theta_0)] $ yields
  \begin{align}
    V &= -2\sin \frac{\hat{\theta}_0-\theta_0}{2} \begin{bmatrix}
             \cY(d_0^+) \sin \frac{\hat{\theta}_0+\theta_0}{2} + \cY(b_0^+) \cos \frac{\hat{\theta}_0+\theta_0}{2}  \\
             \cY(c_0^+) \sin \frac{\hat{\theta}_0+\theta_0}{2} + \cY(a_0^+) \cos \frac{\hat{\theta}_0+\theta_0}{2}   \\
            \cY(b_0^+) \sin \frac{\hat{\theta}_0+\theta_0}{2} - \cY(d_0^+) \cos \frac{\hat{\theta}_0+\theta_0}{2} \\
          \cY(a_0^+) \sin \frac{\hat{\theta}_0+\theta_0}{2} - \cY(c_0^+) \cos \frac{\hat{\theta}_0+\theta_0}{2}
         \end{bmatrix}.
 \end{align}
Applying the relations $\textrm{B}_0$ and $\textrm{D}_0$ allows to replace $\cY(b_0^+) $ and $\cY(d_0^+) $ terms in $V$ only involving $\cY(a_0^+) $ and $\cY(c_0^+)$ and trigonometric functions, which reads as
 \begin{align}
    V &= 2\sin \frac{\hat{\theta}_0-\theta_0}{2}  \begin{bmatrix}
           \cY(a_0^+) \sin \frac{\hat{\theta}_0-\theta_0}{2} + \cY(c_0^+) \cos \frac{\hat{\theta}_0-\theta_0}{2} \\
           \cY(a_0^+) \cos \frac{\hat{\theta}_0+\theta_0}{2} + \cY(c_0^+) \sin \frac{\hat{\theta}_0+\theta_0}{2} \\
           -\cY(a_0^+) \cos \frac{\hat{\theta}_0-\theta_0}{2} + \cY(c_0^+) \sin \frac{\hat{\theta}_0-\theta_0}{2} \\
           \cY(a_0^+) \sin \frac{\hat{\theta}_0+\theta_0}{2} - \cY(c_0^+) \cos \frac{\hat{\theta}_0+\theta_0}{2}
         \end{bmatrix}
  \end{align}
It is straightforward to see that choosing
\begin{align*}
	y_{a_0}& := \sin \frac{\hat{\theta}_0-\theta_0}{2} \big(  \cY(a_0^+) \sin \frac{\hat{\theta}_0-\theta_0}{2} + \cY(c_0^+) \cos \frac{\hat{\theta}_0-\theta_0}{2}) \\
	y_{c_0}&:=\sin \frac{\hat{\theta}_0-\theta_0}{2} \big(  -\cY(a_0^+) \cos \frac{\hat{\theta}_0-\theta_0}{2} + \cY(c_0^+) \sin \frac{\hat{\theta}_0-\theta_0}{2}),
\end{align*}
then $\cY+y_{a_0}\cdot X^{(a_0)}_{\hat{S}} + y_{c_0}\cdot X^{(c_0)}_{\hat{S}}$ indeed satisfies \eqref{eq:3-terms} around the quad $z_0$ for the coupling constant $\hat{x}_0= \tan \frac{\hat{\theta}_0}{2}$. This sequence of rather simple observations are the building block of the deformation procedure, that we precise in the following lemma.
\begin{lemma}\label{lem:deformation-discrete}
	Let $\cS$ be an $s$-embedding of $(S,(x_e)_{e\in E})$ associated to the propagator $\cY$. Consider the Ising model on $(S,(\hat{x}_e)_{e\in E})$, where the set of coupling constant only differ at an edge $e_0$. Using the parametrisation \eqref{eq:x=tan-theta}, define the abstract angles $x_{e_0}=\tan \frac{\theta_0}{2} $ and $\hat{x}_{e_0}=\tan \frac{\hat{\theta}_0}{2}$. Following the identification of double covers made in Section \ref{sub:2-point-fermion} , fix a 'south-eastern' corner $a_0^+ \in \Upsilon^\times $ such that the branching structures of the double-covers $\Upsilon^\times_{(a_0)}$ and $\Upsilon
	^{\times} $ don't coincide on $z_0$. Similarly, fix a 'north-western' corner $c_0^+ \in \Upsilon^\times $ such that the branching structures of the double-covers $\Upsilon^\times_{(c_0)}$ and $\Upsilon
	^{\times} $ don't coincide on $z_0$. Let $X^{(a_0)}_{\hat{S}}$ (respectively $X^{(c_0)}_{\hat{S}}$) the Kadanoff-Ceva fermions normalised such that $X^{(a_0)}_{\hat{S}}(a_0^+)=1$ (respectively $X^{(c_0)}_{\hat{S}}(c_0^+)=1$). Then, as long as $y_{a_0}$ and $y_{c_0}$ are chosen as in \eqref{eq:choice-mismatch-prefactor},\eqref{eq:choice-mismatch-prefactor-2}, then $\hat{\cY}:=\cY+  y_{a_0} X^{(a_0)}_{\hat{S}} + y_{c_0} X^{(c_0)}_{\hat{S}}$ is an s-embedding of $(S,(\hat{x}_e)_{e\in E}) $.	\begin{align}\label{eq:choice-mismatch-prefactor}
	y_{a_0}& := \sin \frac{\hat{\theta}_0-\theta_0}{2} \Big(  \cY(a_0^+) \sin \frac{\hat{\theta}_0-\theta_0}{2} + \cY(c_0^+) \cos \frac{\hat{\theta}_0-\theta_0}{2}\Big) \\
	y_{c_0}&:=\sin \frac{\hat{\theta}_0-\theta_0}{2} \Big(  -\cY(a_0^+) \cos \frac{\hat{\theta}_0-\theta_0}{2} + \cY(c_0^+) \sin \frac{\hat{\theta}_0-\theta_0}{2}\Big).\label{eq:choice-mismatch-prefactor-2}
\end{align}	
\end{lemma}

\subsection{Differential evolution of $s$-embeddings}

One can in principle use the previous lemma in a recursive manner to change not-only one Ising weight but a finite collection of them. From that perspective, this naive recursive approach has several flaws. The first one is that the output of a recursive construction may depend on the order chosen to modify the edge weights. More importantly, one needs to have a refined understanding of the large scale behaviour of the fermions $ X^{(a_0)}_{\hat{S}}$ to understand the deviation between the embeddings. The main result of \cite{MahPar25a} provides an excellent understanding of $ X^{(a_0)}_{S}$ via the geometry of the associated $s$-embedding $\cS$, with explicit scaling factors and a complete description of its continuous counterpart. As the original data is an $s$-embedding $\cS$ of $S$ and not an $s$-embedding of $\hat{S}$, in order to get some precise understanding of the evolution of the embedding, one needs to get some additional estimates regarding the deviation between the combinatorial fermions $ X^{(a_0)}_{S}$ and $X^{(a_0)}_{\hat{S}}$. To avoid both issues, we take a more conceptual approach (first suggested by Dmitry Krachun) which consist in deforming a finite collection of weights in an infinitesimal manner with some time parameter $t$, obtaining the fermion $\cY^{(t)}$ as a solution to a first order ODE. Beyond the possibility to use the very rich literature on differential equations, this approach constructs for free solutions to \eqref{eq:3-terms} when moving continuously the weights, solving (at least in the finite volume case) jumping over one of the main barriers in the embedding setup, which is finding a correct embedding before starting the overall analysis. 

We describe now this continuous time deformation process, still working on the square lattice $S=\mathbb{Z}^2$. Once again, the following formalism works in full generality. Fix two collections $(a_k)_{k\in \Lambda_n} $ and $(c_k)_{k\in \Lambda_n} $ of SE and NW corners inside $\Lambda_n$, and choose respective lifts $(a^{+}_k)_{k\in \Lambda_n} $ and $(c^{+}_k)_{k\in \Lambda_n}$ as in Lemma \ref{lem:deformation-discrete}. For $t\geq 0$ a small enough time parameter, consider the family of Ising models $(S,(x^{(t)}_{e})_{e\in E(\mathbb{Z}^2}))_{t\geq 0}$, where the angle in \eqref{eq:x=tan-theta} for the edge $e\in E(\mathbb{Z}^2)$ is given by
\begin{equation}
	\theta^{(t)}_{e}:=\theta^{(0)}_{e} + m_e \cdot t,
\end{equation}
with a \emph{mass} parameter $m_e$. In what follows, the only input is a given $s$-embedding $\cS^{(0)} $ of $(S,(x^{(0)}_{e})_{e\in E(\mathbb{Z}^2)})$ of the model at the initial time $ t=0$. The following lemma allows to construct a family of $s$-embedding of $(S(t))_{t\geq 0}=(S,(x^{(t)}_{e})_{e\in E(\mathbb{Z}^2)})_{t\geq 0}$ as a solution to some differential equation with initial condition $\cS^{(0)}$. For the rest of this section, one denotes generic corners with the letter $\frak{p}$.
\begin{lemma}\label{lem:ODE-embedding}
Let $\cY^{(0)}:=d\cS^{(0)}$ associated to $(S,(x^{(0)}_{e})_{e\in E(\mathbb{Z}^2)})$. Assume also that the collection of masses $(m_e)_{e\in \mathbb{Z}^2}$ vanishes outside the box $\Lambda$. Consider the differential system defined on corners of $ \Upsilon^{\times} \cap  \Lambda  $, whose initial condition is given by $\cY^{(0)}$ and whose dynamic is given for any $\frak{p} \in \Lambda $ by
		\begin{equation}\label{eq:ODE-embedding}
		\frac{d \cY^{(t)}(\frak{p})}{dt}= \frac{1}{2} \sum\limits_{z_k \in \Lambda}m_k \Bigg[ \cY^{(t)}(c^+_k) \langle \chi_{\frak{p}}\chi_{a^+_k} \rangle_{S(t)} -  \cY^{(t)}(a^+_k) \langle \chi_{\frak{p}}\chi_{c^+_k} \rangle_{S(t)} \Bigg].
	\end{equation}
Then 
\begin{enumerate}
	\item There exist $T_0(\Lambda)>0$ such that the ODE \eqref{eq:ODE-embedding} is well defined on $[0;T_0(\Lambda)] $.
	\item For any time $0\leq t \leq T_0(\Lambda) $, the fermion $\cY^{(t)}$ solution to \eqref{eq:ODE-embedding} constructs an s-embedding $\cS^{(t)}$ of the Ising model $S(t)$.
\end{enumerate}
\end{lemma}
Before diving into the proof, let us detail the intuition leading to \eqref{eq:ODE-embedding}. Using the exact expression  of the coefficients $y_{a_k}$ and $y_{c_k}$ in Lemma \ref{lem:deformation-discrete} when moving infinitesimally from one weight $\theta_k $ to $\hat{\theta}_k$, one sees right away that 
		\begin{equation}
		\frac{d \cY(\frak{p})}{d\theta_k}=\frac{m_k}{2} \Bigg( \cY(c^+_k) \langle \chi_{\frak{p}}\chi_{a^+_k} \rangle_{S} -  \cY(a^+_k) \langle \chi_{\frak{p}}\chi_{c^+_k} \rangle_{S} \Bigg),
	\end{equation}
as the correlators $ \langle \chi_{\frak{p}}\chi_{a^+_k} \rangle_{S}$ and $ \langle \chi_{\frak{p}}\chi_{a^+_k} \rangle_{S}$ are continuous with respect to $\theta_k$ (the correlators involved here correspond to the Ising model with the original weights given by $\theta_k$) . In the case where one recursively moves infinitesimally all the coupling constants in $\Lambda$ as $\theta^{(t)}_j=\theta_j+ t \cdot m_j$ for $t$ small enough, one can extend the last observation and use the continuity of fermions with respect to $t$. Thus, there exist $s$-embedding $\cS^{(t)}$ of $(S,(x^{(t)}_{e})_{e\in \mathbb{Z}^2})$ such that
\begin{equation}
	\cY^{(t)}(\frak{p})=\cY^{(0)}(\frak{p}) + \frac{t}{2} \sum\limits_{z_k \in \Lambda}m_k \Bigg[ \cY^{(0)}(c^+_k) \langle \chi_{\frak{p}}\chi_{a^+_k} \rangle_{S(0)} -  \cY^{(0)}(a^+_k) \langle \chi_{\frak{p}}\chi_{c^+_k} \rangle_{S(0)}\Bigg] + o(t),
\end{equation}
which hints that \eqref{eq:ODE-embedding} indeed constructs an appropriate $s$-embedding.

\begin{proof}[Proof of Lemma \ref{lem:ODE-embedding}]
For a finite box $\Lambda$, the first order linear ODE \eqref{eq:ODE-embedding} rewrites as $Y'(t)=A(t)Y(t)$ for some matrix $A(t)$ and a finite dimensional vector $Y$. The $\mathcal{L}^1$ norm of $A$ is bounded by $O(\sum_{e_k\in \Lambda}|m_{e_k}|)$ at any time as the fermionic correlators are trivially bounded by $1$. Therefore there exist $T_0(\Lambda)>0$ such that the solution to \eqref{eq:ODE-embedding} exists and is unique in $[0;T_0(\Lambda)]$. Let us now check that the solution to \eqref{eq:ODE-embedding} at time $t$ indeed satisfies all propagation identities \eqref{eq:3-terms} for the angles $(\theta_{e}^{(t)})_{e\in \Lambda}$. More concretely, one needs to check that for any corner $ \frak{p}$ neighbouring $ \frak{p}^{\pm}$ in $\Upsilon^{\times}$ in the quad $z_j \in \diamondsuit $, one has
\begin{equation}\label{eq:consistency-embedding}
	\cY^{(t)}(\frak{p})= \cos(\theta^{(t)}_j)\cY^{(t)}(\frak{p}^{+})+\sin(\theta^{(t)}_j)\cY^{(t)}(\frak{p}^{-}).
\end{equation}
Denote by $\cT_{(t)} $ the finite dimensional vector indexed by trios of neighboring corners $\frak{p}, \frak{p}^{\pm}\in \Upsilon^{\times} \cap \Lambda $ around a quad $z_j$ whose coordinates are given by
\begin{equation}
	\cT_{(t)}(\frak{p}) := \cY^{(t)}(\frak{p})- \cos(\theta^{(t)}_j)\cY^{(t)}(\frak{p}^{+})-\sin(\theta^{(t)}_j)\cY^{(t)}(\frak{p}^{-}).
\end{equation}
We are going to prove in \textbf{Steps 1-3} that $\cT$ satisfies some first order ODE of the form $(\cT_{(t)})'=B \times  \cT_{(t)}$ for some fixed matrix $B$. 

\textbf{Step 1: Identifying cancellations for $\cT_{(t)}'(\frak{p)}$:} by construction, the derivative of the LHS of \eqref{eq:consistency-embedding} is given by the RHS of \eqref{eq:ODE-embedding}. The derivative of the RHS of \eqref{eq:consistency-embedding} is given by the
\begin{equation}\label{eq:derrivative-embedding-RHS}
	m_j \Bigg[ -\sin(\theta^{(t)}_j)\cY^{(t)}(\frak{p}^{+}) + \cos(\theta^{(t)}_j)\cY^{(t)}(\frak{p}^{-}) \Bigg] +  \cos(\theta^{(t)}_j) \frac{d}{dt}\cY^{(t)}(\frak{p}^{+}) +\sin(\theta^{(t)}_j) \frac{d}{dt}\cY^{(t)}(\frak{p}^{-}).
\end{equation}
When computing $\cT_{(t)}'(\frak{p})$ at the coordinate corresponding to the trio $\frak{p},\frak{p}^{\pm} \in z_j$, it turns out that all terms involving quads $z_k\neq z_j$ cancel, leaving only few terms to analyse. The goal of this first step is to identify those cancellations. For the last two terms in \eqref{eq:derrivative-embedding-RHS} involving $\frac{d}{dt} \cY^{(t)}(\frak{p}^\pm)$, one has 
\begin{align*}
		\cos(\theta^{(t)}_j) \frac{d}{dt}\cY^{(t)}(\frak{p}^{+}) +\sin(\theta^{(t)}_j) \frac{d}{dt}\cY^{(t)}(\frak{p}^{-})&=\sum\limits_{z_k \in \Lambda}\frac{m_k}{2}  \cY^{(t)}(c^+_k) \langle \chi_{\frak{p}^+}\chi_{a^+_k} \rangle_{S(t)} \cos(\theta^{(t)}_j) & (\star_1)  \\
		& \quad -\sum\limits_{z_k \Lambda} \frac{m_k}{2}  \cY^{(t)}(a^+_k) \langle \chi_{\frak{p}^+}\chi_{c^+_k} \rangle_{S(t)} \cos(\theta^{(t)}_j) & (\star_2)\\
		& \quad +\sum\limits_{z_k \in  \Lambda}\frac{m_k}{2}  \cY^{(t)}(c^+_k) \langle \chi_{\frak{p}^-}\chi_{a^+_k} \rangle_{S(t)} \sin(\theta^{(t)}_j)& (\star_3)\\
		& \quad -\sum\limits_{z_k \in \Lambda} \frac{m_k}{2}  \cY^{(t)}(a^+_k) \langle \chi_{\frak{p}^-}\chi_{c^+_k} \rangle_{S(t)} \sin(\theta^{(t)}_j) & (\star_4)
\end{align*}
Fix $k\neq j$. The \emph{combinatorial correlator} $ \frak{q}\mapsto \langle \chi_{\frak{q}}\chi_{a^+_k} \rangle_{S(t)}  $ satisfies \eqref{eq:3-terms} on $\Upsilon^{\times}$ around $z_j$, with respect to the angle $\theta^{(t)}_{j} $. In particular this ensures that 
\begin{equation}\label{eq:first-matching-term-ODE}
	\langle \chi_{\frak{p}^+}\chi_{a^+_k} \rangle_{S(t)} \cos(\theta^{(t)}_j)+\langle \chi_{\frak{p}^-}\chi_{a^+_k} \rangle_{S(t)} \sin(\theta^{(t)}_j)=\langle \chi_{\frak{p}}\chi_{a^+_k} \rangle_{S(t)}.
\end{equation}
One can then multiply both sides of \eqref{eq:first-matching-term-ODE} by $m_k \cY^{(t)}(c^+_k)$, and sum all equalities corresponding to quads $z_k\neq z_j$. This ensures that
\begin{multline*}
\sum\limits_{\underset{k\neq j}{z_k \in \Lambda}}m_k  \cY^{(t)}(c^+_k) \cdot \Bigg( \langle \chi_{\frak{p}^+}\chi_{a^+_k} \rangle_{S(t)} \cos(\theta^{(t)}_j) +  \langle \chi_{\frak{p}^-}\chi_{a^+_k} \rangle_{S(t)} \sin(\theta^{(t)}_j)\Bigg)=\\ 
\sum\limits_{\underset{k\neq j}{z_k \in \Lambda}}m_k  \cY^{(t)}(c^+_k)  \langle \chi_{\frak{p}}\chi_{a^+_k} \rangle_{S(t)}. 
\end{multline*}
Omitting the contributions around the quad $z_j$ containing $\frak{p},\frak{p}^{\pm}$, the first line of the above equation corresponds to $ (\star_1 + \star_3 )$ while the second line corresponds to the $\cY^{(t)}(c_k^{+})$ terms in \eqref{eq:ODE-embedding}. Similarly one has,
\begin{multline*}
\sum\limits_{\underset{k\neq j}{z_k \in \Lambda}}m_k  \cY^{(t)}(a^+_k) \cdot \Bigg( \langle \chi_{\frak{p}^+}\chi_{c^+_k} \rangle_{S(t)} \cos(\theta^{(t)}_j) +  \langle \chi_{\frak{p}^-}\chi_{c^+_k} \rangle_{S(t)} \sin(\theta^{(t)}_j)\Bigg)=\\ 
\sum\limits_{\underset{k\neq j}{z_k \in \Lambda}}m_k  \cY^{(t)}(a_k^+)  \langle \chi_{\frak{p}}\chi_{c^+_k} \rangle_{S(t)},
\end{multline*}
where omitting once again the contributions around the quad $z_j$, the first line of the above equation corresponds to $ (\star_2 + \star_4 )$ while the second line corresponds to the $\cY^{(t)}(a_k^{+})$ terms in \eqref{eq:ODE-embedding}. All together, this means that in $\cT_{(t)}'(\frak{p})$, all terms involving quads $z_k\neq z_j$ cancel out.

\textbf{Step 2: Identifying the contributions to $\cT_{(t)}'(\frak{p)}$ around $z_j$:}
Given the cancelations identified in \textbf{Step 1}, one can analyse the remaining contributions to  $\cT_{(t)}'(\frak{p)}$ coming from corners belonging to the quad $z_j$. This reads as:
\begin{align*}
\cT_{(t)}'(\frak{p})&= +\frac{m_j}{2} \Big[ \cY^{(t)}(c^+_j) \langle \chi_{\frak{p}}\chi_{a^+_j} \rangle_{S(t)} - \cY^{(t)}(a^+_j) \langle \chi_{\frak{p}}\chi_{c^+_j} \rangle_{S(t)}  \Big]\\
& \quad - \frac{m_j}{2} \cos(\theta^{(t)}_j)\Big[ \cY^{(t)}(c^+_j) \langle \chi_{\frak{p}^+}\chi_{a^+_k} \rangle_{S(t)} 	- \cY^{(t)}(a^+_j) \langle \chi_{\frak{p}^+}\chi_{c_k} \rangle_{S(t)}  \Big]
 \\
 & \quad - \frac{m_j}{2} \sin(\theta^{(t)}_j)\Big[ \cY^{(t)}(c^+_j) \langle \chi_{\frak{p}^-}\chi_{a^+_k} \rangle_{S(t)} - \cY^{(t)}(a^+_j) \langle \chi_{\frak{p}^-}\chi_{c_k} \rangle_{S(t)} \Big]
 \\
	 & \quad -m_j \Big[\sin(\theta^{(t)}_j)\cY^{(t)}(\frak{p}^{+})-\cos(\theta^{(t)}_j)\cY^{(t)}(\frak{p}^{-})\Big], & \\
	 \end{align*}
where the first line comes from $\frac{d}{dt}\cY^{(t)}(\frak{p})$, the second line comes from the term $\cos(\theta^{(t)}_j) \frac{d}{dt}\cY^{(t)}(\frak{p}^{+})$ in \eqref{eq:derrivative-embedding-RHS}, the third line comes from the term $\sin(\theta^{(t)}_j) \frac{d}{dt}\cY^{(t)}(\frak{p}^{-})$ in \eqref{eq:derrivative-embedding-RHS} and the last line comes from the first two terms of \eqref{eq:derrivative-embedding-RHS}.

\textbf{Step 3: Computing $\cT_{(t)}'(\frak{p)}$:}
We are now ready to compute $\cT_{(t)}'(\frak{p)}$. Recall the identifications of branching structures made in Section \ref{sub:2-point-fermion} summarised in Figure \ref{fig:identification-double-covers}. One can now use equations $\textrm{A}'_j,\textrm{B}'_j,\textrm{C}'_j,\textrm{D}'_j $ and $\textrm{A}''_j,\textrm{B}''_j,\textrm{C}''_j,\textrm{D}''_j $ similar to the ones introduced for the quad $z_0$ in Section \ref{sub:2-point-fermion}. There is a natural dichotomy depending on the position SW,SE,NW or NE of the corner $\frak{p}$ around the quad $z_j$ and the two following cases appear.
\begin{itemize}
	\item \textbf{Case 1}: The corner $\frak{p} $ is belongs to SE or NW. We detail here the case where $\frak{p}=a_j^{+}$, the other cases can be treated similarly. In that case one has $\frak{p}^{+}=d_j^{+} $ and $\frak{p}^{-}=b_j^{-}$. We break the contributions of $\cT_{(t)}'(\frak{p})$ into several pieces regrouping the $\cY^{(t)}(a^+_j)$ and $\cY^{(t)}(c_j^+)$ factors. In particular one has
\begin{equation}\label{eq:dichotomy-fermion-1}
	\frac{m_j}{2}\cY^{(t)}(a^+_j)\Bigg[ - \langle \chi_{\frak{p}}\chi_{c^+_j} \rangle_{S(t)} +  \langle \chi_{\frak{p}^+}\chi_{c^+_j} \rangle_{S(t)} \cos(\theta^{(t)}_j) + \langle \chi_{\frak{p}^-}\chi_{c^+_j} \rangle_{S(t)} \sin(\theta^{(t)}_j) \Bigg]=0,
\end{equation}
\begin{multline}\label{eq:dichotomy-fermion-2}
	\frac{m_j}{2}\cY^{(t)}(c_j^+)\Bigg[\langle \chi_{\frak{p}}\chi_{a_j^+} \rangle_{S(t)} - \langle \chi_{\frak{p}^+}\chi_{a_j} \rangle_{S(t)} \cos(\theta^{(t)}_j)-\langle \chi_{\frak{p}^-}\chi_{c_j} \rangle_{S(t)} \sin(\theta^{(t)}_j)\Bigg]\\
	 =m_j\cY^{(t)}(c^+_j).
\end{multline}
where \eqref{eq:dichotomy-fermion-1} comes from $(\textrm{A}_j'') $ and \eqref{eq:dichotomy-fermion-2} comes from $(\textrm{A}_j')$.
Recalling that $\frak{p}=a_j^+$, $\frak{p}^{+}=d_j^{+} $ and $\frak{p}^{-}=b_j^{-}$, one has (adding the last line of the first equation of Step 2)
\begin{equation}
	\cT_{(t)}'(a_j^{+})=m_j \Bigg[ \cY^{(t)}(c^+_j)  + \cos(\theta_j^{(t)})\cY^{(t)}(b^-_j) - \sin(\theta_j^{(t)})\cY^{(t)}(d^+_j)  \Bigg] = -m_j\cT_{(t)}'(c_j^{+}).
\end{equation}
\item \textbf{Case 2}: The corner $\frak{p} $ is belongs to SW or NE. We detail here the case where $\frak{p}=b_j^{+}$, the other cases can be treated similarly. In that case one has $\frak{p}^{+}=c_j^{+} $ and $\frak{p}^{-}=a_j^{-}$. We break once again the contributions of $\cT_{(t)}'(\frak{p})$ into several pieces regrouping the $\cY^{(t)}(a^+_j)$ and $\cY^{(t)}(c_j^+)$ factors. In particular one has
\begin{multline}\label{eq:dichotomy-fermion-3}
	\frac{m_j}{2}\cY^{(t)}(a^+_j)\Bigg[ - \langle \chi_{\frak{p}}\chi_{c^+_j} \rangle_{S(t)} +  \langle \chi_{\frak{p}^+}\chi_{c^+_j} \rangle_{S(t)} \cos(\theta^{(t)}_j) + \langle \chi_{\frak{p}^-}\chi_{c^+_j} \rangle_{S(t)} \sin(\theta^{(t)}_j) \Bigg]\\=-m_j\cos(\theta_j^{(t)})\cY^{(t)}(a^+_j),
\end{multline}

\begin{multline}\label{eq:dichotomy-fermion-4}
	\frac{m_j}{2}\cY^{(t)}(c_j^+)\Bigg[\langle \chi_{\frak{p}}\chi_{a_j^+} \rangle_{S(t)} - \langle \chi_{\frak{p}^+}\chi_{a_j} \rangle_{S(t)} \cos(\theta^{(t)}_j)-\langle \chi_{\frak{p}^-}\chi_{c_j} \rangle_{S(t)} \sin(\theta^{(t)}_j)\Bigg]\\=m_j\sin(\theta_j^{(t)})\cY^{(t)}(c^+_j),
\end{multline}
where \eqref{eq:dichotomy-fermion-3} comes from $(\textrm{B}_j'') $ and \eqref{eq:dichotomy-fermion-4} comes from $(\textrm{B}_j')$. Recalling that $\frak{p}=b_j^+$, $\frak{p}^{+}=c_j^{+} $ and $\frak{p}^{-}=a_j^{-}$ allows to conclude that
\begin{multline}
	\cT_{(t)}'(b_j^{+})=m_j \Bigg[ +\cos(\theta_j^{(t)})\cY^{(t)}(a^+_j) +\sin(\theta_j^{(t)})\cY^{(t)}(c^+_j) \\- \sin(\theta_j^{(t)})\cY^{(t)}(\frak{p}^+) +\cos(\theta_j^{(t)})\cY^{(t)}(\frak{p}^-) \Bigg] = 0.
\end{multline}
	\end{itemize}

\textbf{Step 4: Concluding the proof}
The previous construction ensures that $(\cT_{(t)})'=B \times  \cT_{(t)}$ for some fixed matrix $B$ and $\cT_{(0)}=0$ (as $\cS^{(0)}$ is an $s$-embedding of $S(0)$). Therefore, one can apply the Cauchy-Lipchitz theorem allows to conclude the proof. 
\end{proof}

\begin{remark}\label{rem:deformation-varying-mass}
	To lighten the notations and the proof, we stated the previous Theorem assuming that all the masses $(m_e)_{e\in \Lambda}$ are constant along the deformation process. In full generality, assuming that	\begin{equation}
	\theta^{(t)}_{e}:=\theta^{(0)}_{e} +m_e(t),
\end{equation}
for some collection of $\mathcal{C}^{1}$ smooth functions $t\mapsto m_e(t)$, then Lemma \ref{lem:ODE-embedding} still holds, with a differential equation becoming \begin{equation}\label{eq:ODE-embedding-general}
		\frac{d \cY^{(t)}(\frak{p})}{dt}= \frac{1}{2} \sum\limits_{z_k \in \Lambda}m_k'(t) \Bigg[ \cY^{(t)}(c^+_k) \langle \chi_{\frak{p}}\chi_{a^+_k} \rangle_{S(t)} -  \cY^{(t)}(a^+_k) \langle \chi_{\frak{p}}\chi_{c^+_k} \rangle_{S(t)} \Bigg].
\end{equation}  
The previous proof applies almost verbatim, the only minor change is that $(\cT_{(t)})'=B(t) \times  \cT_{(t)}$	 for some time dependent matrix $B(t)$ where the coefficients $m_j$ are now replaced by $m_j'(t)$. As the Cauchy-Lipchitz theorem is still valid in that setup this allows to conclude exactly in the same way.
\end{remark}

\begin{remark}
	The deformation procedure presented in Lemma \ref{lem:ODE-embedding} provides very basic discrete-level explanations regarding the lack of holomorphicity of the massive models as highlighted by the works of Park \cite{park2018massive,Park2021Fermionic}. At time $t=0$, one can see using \eqref{eq:F-from-X} 
	\begin{equation}
		 \cY^{(0)}(c^+_k) \langle \chi_{\frak{p}}\chi_{a^+_k} \rangle_{S(0)} -  \cY^{(0)}(a^+_k) \langle \chi_{\frak{p}}\chi_{c^+_k} \rangle_{S(0)} =r_{z_k}\overline{F_{\frak{p}}(z_k)}\big(\cos(\theta_k)\sin(\theta_k)\big)^{-1},
	\end{equation}
where $F_{\frak{p}}$ is the $s$-holomorphic complexification of $\frak{q}\mapsto \langle \chi_{\frak{p}}\chi_{\frak{q}}\rangle$. The function $F_{\frak{p}}$ is \emph{discrete holomorphic} in the usual sense on the square lattice corresponding to the $s$-embedding $\cS^{(0)}$. In the uniform mass case, this observation yields that the deformation process using differential equations inserts to the conformal structure attached to the massive Ising model some \emph{anti-holomorphic} component. This hints why the conformal structure of massive models is adapted to fermions satisfying some more evolved equation than $\partial_z f=0$ in the usual Euclidian metric. 
\end{remark}

\section{The embedding of the massive-Ising model via continuous deformation}\label{sec:massive-deformation}

In this section we study in greater details the effective output of the embedding deformation induced by \eqref{eq:ODE-embedding} when passing from the critical square lattice to the so called massive regime. Beyond the results themselves, we see this section as simplified introduction to the method we will use in Section \ref{sec:iid-deformation} to treat the near-critical i.i.d. model. One works in the box $\Lambda_n$, centred at the origin. Let us start at time $t=0$ from the Ising model on the critical square lattice, with uniform critical weights $\theta_k=\frac{\pi}{4}$ for all edges $e_k\in \Lambda_n$. Let $\cS^{(0)}$ be its standard embedding onto $\mathbb{C}$, made a tilling of squares of mesh size $\frac{\sqrt{2}}{n} $. One can see the bipartite graph $\cS(\Lambda(\mathbb{Z}^2))=\cS(G^{\bullet})\cup \cS(G^{\circ})$ as an isoradial grid, where all edges have length $\frac{1}{n}$. Our goal is to move continuously each Ising weight $\theta_{k}=2\arctan(x_{e_k}) $ from $\frac{\pi}{4}$ to $\frac{\pi}{4}+\frac{m_k}{n}$ while capturing how the family of embeddings $(\cS^{(t)})_{t\leq \frac{1}{n}}$ deviates from $\cS^{(0)}$. 
In order to control this deviation along the deformation process, two main difficulties arise. The first one is to check that the embedding $\cS^{(t)} $ is proper, which allows use the two-point fermion estimate \eqref{eq:bound-correlator-distance}. The other point that requires some careful attention is to ensure that $\cS^{(t)} $ didn't deviate already too much from $\cS^{(0)} $, which allows to keep applying \eqref{eq:bound-correlator-distance} with the same pre-factor $\Theta$ attached to the some fixed in advance $r_0,\theta_0$ parameters for \Unif\, grids. Indeed, this pre-factor comes from the regularity theory of $s$-holomorphic functions presented Section \ref{sub:regularity} and could in principle degenerate along the deformation process. Once this is done, controlling the deviation of $\cS^{(t)} $  from $\cS^{(0)} $ becomes some standard Lotka–Volterra predator–prey problem that can be smoothly treated. The output of the construction is that the massive deformation via \eqref{eq:ODE-embedding} with initial condition $\cS^{(0)} $ produces after a time $t=\frac{1}{n}$ a proper s-embedding that satisfies \Uniff\ with $\delta=\frac{1}{n}$ and well chosen constants $r_0,\theta_0$, at least if all the masses are bounded by a small enough $m$. We focus on the effects of the deformation \eqref{eq:ODE-embedding} on the corners $\frak{p}$ such that $\cS^{(0)}(\frak{p})$ belongs to a square $\mathcal{D}_2$ of width 2 centred at the origin. One naturally identifies the isoradial grid of mesh size $\frac{1}{n}$ in $\mathcal{D}_2$ with $\Lambda_{2n}$. There are $O(n^2)$ different corners inside $\mathcal{D}_2$.  For two neighbouring corners $\frak{p},\frak{p}'$ inside $\Lambda_n$ belonging to the quad $z_j$, set (as long as $\cS^{(t)}$ is proper)
\begin{equation}\label{eq:M(p)(t)}
	M^{(t)}(\frak{p}):= |\cY^{(t)}(\frak{p})-\cY^{(0)}(\frak{p})|, 
\end{equation}
\begin{equation}\label{eq:arg(t)}
	\phi^{(t)}_{z_j} ( \frak{p},\frak{p}' ):=\arg \phi_{v(\frak{p},\frak{p}'),z_j}^{(t)},
\end{equation}
where $\phi_{v(\frak{p},\frak{p}'),z_j}^{(t)}$ is the \emph{geometrical} half-angle at $\cS^{(t)}(v_{\frak{p},\frak{p}'})$ along the bisector linking $\cS^{(t)}(v_{\frak{p},\frak{p}'})$ to $\cS^{(t)}(z_j)$.
Finally, set
\begin{equation}\label{eq:M(t)}
	M(t):= \max\limits_{\frak{p}\in \Lambda_{2n}}  M^{(t)}(\frak{p})
\end{equation}
In what follows, the constant $\Theta$ corresponds to the one in \eqref{eq:bound-correlator-distance} for $s$-embeddings satisfying \Uniff\, with $r_0=10$ and $\theta_0=\frac{\pi}{10}$. We are now ready to prove the next proposition, which quantifies the deviation from of $\cS^{(t)}$ to $\mathcal{S}^{(0)}$ for times $0\leq t\leq \frac{1}{n}$, as long as $m$ is chosen small enough. This will be done using standard deviation estimates for ordinary differential equations which allows to keep track of the control of the rate of growth of the full-plane fermion \eqref{eq:bound-correlator-distance} along the deformation. We will always verify the effect of the deformation of the coupling constants inside $\Lambda_{n}$ on the $s$-embedding $\cS(\Lambda_{2n})$ in order to apply the properness principle recalled in Section \ref{sub:argument-principle}.

\begin{prop}\label{prop:effective-cost-deformation}
Assume that the masses $(m_e)_{e\in \mathbb{Z}^2}$ vanish outside $\Lambda_n$ and are bounded by (some small) enough $m>0$ inside $\Lambda_n$. Then for any $0\leq t\leq \frac{1}{n}$, any corner $\frak{p}\in \Lambda_{2n}$, any $v_1\neq v_2 \in \Lambda_{2n}$ one has:
		\begin{enumerate}[label=(\roman*)]
	\item The deviation of $\cY^{(t)}$ can be bounded via
\begin{equation}
| \cY^{(t)}(\frak{p}) - \cY^{(0)}(\frak{p})| \leq  \frac{ 0.1}{n^\frac{1}{2}} \times nt
\end{equation}
\item The embedding  $\cS^{(t)}(\Lambda_{2n})$ is proper, and distances in $\cS^{(t)}$ satisfy
\begin{equation}
	0.9\leq \frac{|\cS^{(t)}(v_1)-\cS^{(t)}(v_2)|}{|\cS^{(0)}(v_1)-\cS^{(0)}(v_2)|} \leq 1.1
\end{equation}
\item The embedding $\cS^{(t)}(\Lambda_{2 n})$ satisfies \Unifff\,.\end{enumerate}
\end{prop}
Before diving into the details of the proof, let us highlight the strategy we repeatedly use. One first trivially bounds the two-points fermions by $1$, and derives that $\cY^{(t)}$ didn't deviate too much from $ \cY^{(0)}$, at least for a small enough time $O(n^{-2})$. The control of the deviation of $\cY^{(t)}$ implies that $\cS^{(t)}(\Lambda_{2n})$ is still proper and still satisfies \Unifff\;. Therefore, one can now use the improved bounds  \eqref{eq:bound-correlator-distance} and obtain, using the Grownwall lemma, a sharp (up to constant) order of magnitude of the effective deviation between $\cY^{(t)}$ and $\cY^{(0)}$ during that $O(n^{-2})$ time span. Applying recursively this procedure $n$ times over intervals of lenght $O(n^{-2})$ allows to conclude.
\begin{proof}
The proof will be made by recursion over indices $0\leq \ell \leq n-1$, showing that $(i),(ii),(iii)$ hold for any time $t\in [\frac{\ell}{n^2};\frac{\ell+1}{n^2}]$. 

\textbf{Initialisation with} $\ell=0$. We prove here that $(i),(ii),(iii)$ hold $0\leq t \leq \frac{1}{n^2}$. The proof splits in several elementary that we highlight here.

\textbf{Step 1: Evaluating the deviation of $\cY^{(t)}$ using trivial bounds on $\langle \chi_{\frak{p}} \chi_{\frak{q}} \rangle$}
Denote first
\begin{equation}\label{eq:def_phi_n}
	\varphi_n(s):= \max\limits_{\frak{p} \in \Lambda_{2n}} \sum\limits_{z_k\in \Lambda_n} |\langle \chi_{\frak{p}} \chi_{a^+_k} \rangle_{S(s)} |	 + |\langle \chi_{\frak{p}} \chi_{c^+_k} \rangle_{S(s)}|.	
	\end{equation}
There are at most $O(n^2)$ Kadanoff-Ceva fermions $\langle \chi_{\frak{p}} \chi_{\frak{q}} \rangle_{S(s)}$ involved each sum used to define \eqref{eq:def_phi_n}. Each Kadanoff-Ceva fermion is trivially bounded by $1$, independently of the time $s$, which ensures that $\varphi_{n}(s)= O(n^2)$, for some absolute constant $O$, independent from $n$. For any corner $\frak{p} \in \Lambda_{2n}$, one can integrate the ODE \eqref{eq:ODE-embedding}, which gives
\begin{align*}
	| \cY^{(t)}(\frak{p}) - \cY^{(0)}(\frak{p})|&= \bigg| \int\limits_{0}^{t} \sum\limits_{z_k\in \Lambda_n} m_k\big(   \cY^{(s)}(c^+_k)\langle \chi_{\frak{p}} \chi_{a^+_k} \rangle_{S(s)}  	- \cY^{(s)}(a^+_k) \langle \chi_{\frak{p}} \chi_{c_k} \rangle_{S(s)} \big) ds \bigg| \\
	& \leq \int\limits_{0}^{t}  m\sum\limits_{z_k\in \Lambda_n} \big| \cY^{(0)}(c^+_k) \big|	\cdot \big| \langle \chi_{\frak{p}} \chi_{a^+_k} \rangle_{S(s)} \big| + \big|\cY^{(0)}(c^+_k) \big| \cdot \big| \langle \chi_{\frak{p}} \chi_{c^+_k} \rangle_{S(s)}\big| \\
	& \quad + \int\limits_{0}^{t} m\sum\limits_{z_k\in \Lambda_n} \big|\cY^{(s)}(c^+_k) - \cY^{(0)}(c^+_k) \big| \cdot  \big|\langle \chi_{\frak{p}} \chi_{a^+_k} \rangle_{S(s)} \big|	  \\
	& \quad + \int\limits_{0}^{t} m\sum\limits_{z_k\in \Lambda_n}  \big|\cY^{(s)}(a^+_k) - \cY^{(0)}(a^+_k) \big| \cdot  \big|\langle \chi_{\frak{p}} \chi_{c^+_k} \rangle_{S(s)}\big|,
\end{align*}
where in the above inequality we used that
\begin{itemize}
	\item For $e_k\in \Lambda_{n}$, its mass is bounded by $m$ i.e.\ $\big| m_k \big| \leq m $.
	\item The triangular inequality allows to decompose for $0\leq s \leq \frac{1}{n^2}$
\begin{align*}
	\big|\cY^{(s)}(a^+_k) \big| &= \big|\cY^{(s)}(a^+_k)- \cY^{(0)}(a^+_k) + \cY^{(0)}(a^+_k)\big|\\
	&\leq \big|\cY^{(0)}(a^+_k)\big|+ \big|\cY^{(s)}(a^+_k)- \cY^{(0)}(a^+_k)\big|.
\end{align*}
\end{itemize}
For the standard square embedding $\cS^{(0)}$ of the critical model, one has $\big|\cY^{(0)}(a^+_k) \big|=\big|\cY^{(0)}(c^+_k) \big|=n^{-\frac{1}{2}}$. Therefore, taking the maximum of the above equality over all corners $\frak{p}\in \Lambda_{2n}$ ensures that
\begin{equation}\label{eq:Grownwall}
	M(t) \leq \frac{ m\cdot K(t)}{n^{\frac{1}{2}}}+ m\int\limits_{0}^{t} M(s) \varphi_n(s) ds,
\end{equation}
with an increasing non-negative function $K(t):= \int\limits_{0}^{t} \varphi_n(s)ds$. One can now apply the standard Grownwall lemma to \eqref{eq:Grownwall} which ensures that 
\begin{equation}\label{eq:Grownwall-2}
	M(t) \leq m\cdot \frac{ K(t)}{n^{\frac{1}{2}}} \times \exp(m\cdot K(t)).
\end{equation}
Since $\varphi_{n}(s)=O(n^2)$, the function $K$ is uniformly bounded in $[0;\frac{1}{n^2}]$, uniformly in $n$. Taking $m$ small enough ensures that for any $s\in [0;\frac{1}{n^2}]$ and $\frak{p}\in \Lambda_{2n}$, one has 
	 \begin{equation}\label{eq:deformation-small-time}
|  \cY^{(s)}(\frak{p}) - \cY^{(0)}(\frak{p}) | \leq M(s) \leq  \frac{0.0001}{n^{\frac{1}{2}}}.   \end{equation}
\smallskip 
\textbf{Step 2: Verifying that $\cS^{(s)}(\Lambda_{2n})$ is proper and comparable to $\cS^{(0)}$} The chore of this second step is to prove the distance comparability $(ii)$, which allows to deduce both properness of $\cS^{(s)}(\Lambda_{2n})$ and that it satisfies \Unifff\,. For any pair of distinct vertices $v,v'\in \Lambda_{G}$, there exist an injective sequence  $(\frak{q}_j)_{j\leq d_{G}(v,v')}$  sequence of $d_{G}(v,v') \leq 4n \cdot |\cS^{(0)}(v)-  \cS^{(0)}(v')|$ neighbouring corners such that
\begin{equation}\nonumber
	\cS^{(s)}(v)-  \cS^{(s)}(v') =\sum\limits_{j=1}^{d_G(v,v')}(-1)^{j}\cY^{(s)}(\frak{q}_j)^2.
\end{equation}
Provided $m$ is chosen small enough (as in Step 1), \eqref{eq:deformation-small-time} implies that for $0\leq s\leq \frac{1}{n^2}$
	 \begin{equation}\label{eq:deformation-growth-square}
\bigg| \big( \cY^{(s)}(\frak{q}_j) \big)^2 -    \big( \cY^{(0)}(\frak{q}_j) \big)^2 \bigg| \leq \frac{0.0025}{n}.
\end{equation} 
Summing brutally the errors along the path $(\frak{q}_j)_{j\leq d_{G}(v,v')}$ ensures that 
\begin{equation}\label{eq:deformation-growth-distance}
	|\frac{\cS^{(s)}(v)-\cS^{(s)}(v')}{\cS^{(0)}(v)-\cS^{(0)}(v')} -1| \leq  \sum\limits_{j=1}^{d_G(v,v')} \frac{0.0025}{n} \leq 0.01,
\end{equation}
concluding that the distance comparability in $(ii)$ holds. This also ensures
that $\cS^{(s)}(\Lambda_{2n})$ is proper. To verify this statement, one first notes that \eqref{eq:deformation-growth-square} ensures that for any quad $z\in \diamondsuit(G)$, the orientations of the tangential quadrilaterals $\cS^{(0),\dm}(z)$ and $\cS^{(s),\dm}(z)$ are the same, as the fermion  $\cY^{(s)}$ didn't move enough from its original value to modify the orientation of the image of a tangential quad. Moreover, one use a specific choice of vertices $v,v'$ in \eqref{eq:deformation-growth-distance}. More precisely, fix first $\cS^{(s)}(v=0)=0_{\mathbb{C}}$ (i.e.\ one should fix a reference point for the embedding and we declare here the image of the origin of $\Lambda_{2n}$ to be the origin of the complex plane). One can then fix $v'\in \gamma_{\Lambda_{2n}}$, the (oriented in the positive direction) boundary  of the \emph{combinatorial} box $\Lambda_{2n}$ (defined in the context of Section \ref{sub:argument-principle}). It is clear from \eqref{eq:deformation-growth-distance} that 
\begin{equation}\nonumber
	0.99\leq \frac{|\cS^{(s)}(v')-\cS^{(s)}(0)|}{|\cS^{(0)}(v')-\cS^{(0)}(0)|} \leq 1.01,
\end{equation}
which ensures that $\cS^{(s)}(\gamma_{\Lambda_{2n}})$ only winds once around the origin of the plane. Moreover, it is a simple curve, as any self-intersection would make \eqref{eq:deformation-growth-distance} would fail. Therefore, one can apply once again the argument principle reasoning presented in Section \ref{sub:argument-principle} to conclude that $\cS^{(s)}(\Lambda_{2n})$ is indeed proper. To conclude on the item $(iii)$,  it is enough to recall that the original square lattice $\cS^{(0)}$ satisfies \Unifff\,  with a pretty fair margin (all edge-lengths are $\frac{1}{n}$ and all angles are $\frac{\pi}{2}$ in $\cS^{(0)}$) to see that the deviation \eqref{eq:deformation-small-time} implies that $\cS^{(s)}(\Lambda_{2n})$ still satisfies \Unifff\,.
\smallskip

\textbf{Step 3: Using the improved bounds \eqref{eq:bound-correlator-distance}} Once we know that all embeddings $\cS^{(s)}(\Lambda_{2n})$ are proper and satisfy \Unifff\,, it is possible, when working with the fermion $\cY^{(s)} $, to replace the trivial bounds of the form $|\langle \chi_{\frak{p}}\chi_{a_k^{+}}\rangle| \leq 1$ by the accurate one \eqref{eq:bound-correlator-distance}. Applying \eqref{eq:deformation-small-time} and the distance comparability of the item $(ii)$ proven in Step 1 ensures that
\begin{equation}\label{eq:bound-correlator-embedding-within-proof}
 	|\langle \chi_{\frak{p}} \chi_{a^+_k} \rangle_{S(s)}| \leq   \Theta \cdot \frac{|\cY^{(s)}(a^+_k) \cY^{(s)}(\frak{p})| }{|\cS^{(s)}(\frak{p})-\cS^{(s)}(a^+_k)|} \leq \Theta \cdot \frac{(1.1)^2}{n} \frac{1}{0.9\cdot|\cS^{(0)}(\frak{p})-\cS^{(0)}(a_k)|}.
 \end{equation}
Taking the maximum over corners $\frak{p}\in \Lambda_{2n}$, one has
\begin{align}\label{eq:bound-phi-optimal}
	\varphi_{n}(s)&= \max_{\frak{p}\in \Lambda_{n}} O\Bigg( \Theta \sum\limits_{z_k\in \Lambda_n} \frac{1}{n} \frac{1}{\cdot|\cS^{(0)}(\frak{p})-\cS^{(0)}(z_k)|}\Bigg)\\
	& = n\times \max_{\frak{p}\in \Lambda_{n}} O\Bigg( \Theta \sum\limits_{z_k\in \Lambda_n} \frac{1}{n^2} \frac{1}{\cdot|\cS^{(0)}(\frak{p})-\cS^{(0)}(z_k)|}\Bigg)\\
	& =n\times O\Bigg( \displaystyle\int_{\mathcal{D}_1}\frac{1}{|z-\frak{p} |}dA(z) \Bigg)\\
	&=O(n), 
\end{align}
where $O$ independent from $m$ and $n$ and we passe from the second to the third line using the approximation of the discrete sum by its continuous area integral. Once can now redo verbatim the reasoning of Step $1$ using this the improved bound $\varphi_n(s) = O(n)$ in the Grownwall lemma \eqref{eq:Grownwall-2}. Hence, provided $m$ chosen small enough, one has $M(s)\leq (0.01\cdot n^{-\frac{1}{2}})\times ns$, which concludes the proof of the initialisation.

\smallskip
\textbf{Heredity} Assume that the proposition has been proven for all $0\leq \ell'\leq \ell$. Let us prove that $(i),(ii),(iii)$ hold for $\frac{\ell}{n^2}\leq t \leq \frac{\ell+1}{n^2}$. The recurrence hypothesis ensures that
\begin{equation*}
	\Big|\cY^{(\frac{\ell}{n^2})}(\frak{p})- \cY^{(0)}(\frak{p})\Big|\leq  \frac{0.1}{n^{\frac{1}{2}}} \cdot n\cdot  \frac{\ell}{n^2} \leq \frac{0.1}{n^{\frac{1}{2}}}.
\end{equation*}
Therefore, the embedding $\cS^{(\frac{\ell}{n^2})}(\Lambda_{2n})$ satisfies \Unifff\ with a fair margin. Therefore, one can redo almost verbatim the proof of the initialisation, running the same ODE for a time $0\leq s \leq n^{-2}$, but this time with a initial condition being the fermion $\cY^{(\frac{\ell}{n^2})}$. Since the deterministic constants on the line-to-line passage (e.g.\ the deterministic coefficients $0.0025$, $0.01 $ ...) were chosen to be far from sharp, it is clear that the initial condition $|\cY^{(0)}(a_k^{+})|=|\cY^{(0)}(c_k^{+})| =n^{-\frac12}$ can be replaced by $0.9 \cdot n^{-\frac12} \leq |\cY^{(\frac{\ell}{n^2})}(a_k^{+})|, |\cY^{(\frac{\ell}{n^2})}(c_k^{+})|\leq 1.1 \cdot n^{-\frac12} $ with no incidence in the proof nor the choice of the small enough mass $m$. This allows to deduce that one can keep the same $m>0$ is in the initialisation step while still having for any $\frac{\ell}{n^2}\leq s \leq \frac{\ell+1}{n^2}$ and any $\frak{p}\in \Lambda_{2n}$ the control
\begin{equation*}
	\Big|\cY^{(s)}(\frak{p})-\cY^{(\frac{\ell}{n^2})}(\frak{p}) \Big|\leq \frac{0.1}{n^{\frac{1}{2}}}\Big(s-\frac{\ell}{n^2}\Big)\cdot n,
\end{equation*}
which allows to conclude that
\begin{equation*}
	\Big|\cY^{(s)}(\frak{p})-\cY^{(0)}(\frak{p}) \Big|\leq \frac{0.1}{n^{\frac{1}{2}}}\cdot ns.
\end{equation*}
\end{proof}

Once Proposition \ref{prop:effective-cost-deformation} ensures that existence of an $s$-embedding of the massive square lattice (for a small enough mass), one can use the technology recalled in Theorem \ref{thm:unif-rsw} together with standard monotonicity properties of the FK-Ising model to provide a new proof of Theorem \ref{thm:near-critical-RSW}.

\begin{proof}[Proof of Theorem \ref{thm:near-critical-RSW}]
Consider the massive FK-Ising model on the annulus $\textrm{A}_n:=\Lambda_n \backslash \Lambda_{\frac{n}{2}}$, with free boundary conditions in both boundaries, assuming that all the masses are bounded by some $m_0>0$ and fix $\varepsilon$ small enough such that $m_0\leq m \varepsilon^{-1}$, where the parameter $m>0$ is given in Proposition \ref{prop:effective-cost-deformation}. Fix a translate of the annulus $x+\textrm{A}_{\varepsilon n}$ centred at $x$. As all the masses in  $x+\textrm{A}_{\varepsilon n}$ are bounded by $m_0\leq m \varepsilon^{-1}$, one can apply Proposition \ref{prop:effective-cost-deformation} to see that one can construct via the ODE \eqref{eq:ODE-embedding} a proper $s$-embedding of $\cS(x+\textrm{A}_{\varepsilon n})$ equipped with the massive weights induced by those in $\Lambda_n $ that satisfies \Unifff\,. In particular, Theorem \ref{thm:RSW-s-embeddings} ensures that 
\begin{equation}\label{eq:open-circuit-bound-proof}
	\mathbb{P}^{\textrm{free}}_{x+\textrm{A}_{\varepsilon n}} \big[ \textrm{there exists an open circuit in } x+\textrm{A}_{\varepsilon n} \big] \geq  \textrm{cst}>0,
\end{equation}
for some positive constant $ \textrm{cst}=\textrm{cst}(r_0,\theta_0)$ for the parameters $(r_0,\theta_0)=(10,\pi \cdot 10^{-1})$ in Theorem \ref{thm:RSW-s-embeddings}. One can now cover $\Lambda_n$ with $O(\varepsilon^{-2})$ translates of $\Lambda_{\varepsilon n}$ such that neighbouring annuli $\Lambda_{\varepsilon n} \backslash \Lambda_{\frac{\varepsilon}{2} n}$ overlap enough each other. Applying the standard machinery of positive association for crossing events based on the FKG inequality and monotonicity (see e.g. \cite[Theorem 2.1]{FK_scaling_relations}) ensures that existence of some (uniform in $n$) lower bound on the probability of the existence of an open circuit in $\textrm{A}_n$. One can apply a similar reasoning to the dual model, which allows to conclude the proof.
\end{proof}

\begin{proof}[Proof of Theorem \ref{thm:unif-rsw}]
We only sketch here the proof, pointing out the main ideas, as it is appears to be a straightforward generalisation of our alternative proof of Theorem \ref{thm:RSW-s-embeddings} using embedding deformations. Let $\cS$ be an $s$-embedding satisfying \Uniff\, for some parameters $ \delta,r_0,\theta_0$. Fix a the approximation of square $\Lambda_{\rho}^{\delta}$ up to $4\delta$, for some large enough $\rho \geq \textrm{cst}(r_0,\theta_0) \cdot \delta$, where the constant $\textrm{cst}(r_0,\theta_0)$ only depends on $r_0,\theta_0$. Denote by $\theta_e$ the angle attached to the edge $e\in \cS$ and consider the massive deformation of the FK-Ising model where each Ising weight $\theta_e$ is replaced by $\theta_e+m_e \cdot \delta$, where the collection of masses $(m_{e})_{e\in \Lambda_{\rho}^{\delta}} $ uniformly by some small enough $m=m(r_0,\theta_0)>0$. We claim that one can generalise the control of the massive deformation presented in  Proposition \ref{prop:effective-cost-deformation} using the deformation ODE \eqref{eq:ODE-embedding}. Indeed, the only crucial point in the analysis is the bound on the rate of decay of the Kadanoff-Ceva fermions in the embedding $\cS$, which remains controlled by some bound of the form \eqref{eq:bound-correlator-distance}, where constants only depend on $r_0,\theta_0$. Therefore, provided the upper bound $m$ on the masses is small enough, one can construct proper $s$-embedding of the massive model in $\Lambda_{\rho}^{\delta}$ satisfying \Uniffff\.. Playing once again with scales and taking $ \varepsilon$ small enough, in each annulus $\Lambda_{\varepsilon \rho}^{\delta} \backslash \Lambda_{\frac{\varepsilon}{2} \rho}^{\delta} $, there exist some constant $\textrm{cst}(2r_0,\frac{\theta_0}{2}) $, only depending on $r_0,\theta_0$ such that the analog of \eqref{eq:open-circuit-bound-proof} holds. One can then conclude as in the case of the massive square lattice.
\end{proof}

\begin{remark}\label{rem:modifying-proof-small-times}
If one assumes that all the masses are bounded by $1$ and that the initial condition $\cY^{(0)}$ corresponds to the standard embedding of the critical square lattice. Adapting the proofs presented here shows that there exist $m_0>0$ small enough such that the ODE \eqref{eq:ODE-embedding} constructs proper $s$-embeddings $\cS^{(t)}(\Lambda_n)$ satisfying \Unifff\ for all instants $t\in [0;\frac{m_0}{n}]$. This means that one can keep the deformation of the entire $\cS^{(0)}(\Lambda_n)$ within the class of proper $s$-embeddings satisfying \Unifff\, if one stops the
deformation at a time $\frac{m_0}{n}$ for some small enough $m_0$.
\end{remark}

\begin{remark}\label{rem:doubly-periodic-deformation}
One can wonder if it is possible to keep running the deformation (in a deterministic setup) for a time longer than $O(n^{-1})$. In the original work of \cite{HS-energy-Ising} using the discrete exponentials of Kenyon \cite{Ken} (see also a more compact rewriting in \cite{CIM-universality}), the explicit expression of $\langle \chi_{a}\chi_{\frak{p}} \rangle_{\cS^{(0)}} $ on the critical square lattice of mesh size $\frac{1}{n}$ (and more generally on isoradial lattices) has the asymptotic   
\begin{equation}\label{eq:asympototic-critical-full-plane}
	\langle \chi_{a}\chi_{\frak{p}} \rangle_{S^{(0)}}= \frac{1}{n\pi} \cdot  \Re [\frac{\overline{\eta_a \cdot \eta_\frak{p}}}{\cS^{(0)}(\frak{p})-\cS^{(0)}(a)} ] + O(\frac{1}{n^3}\frac{1}{|\cS^{(0)}(\frak{p})-\cS^{(0)}(a)|^3}).
\end{equation}
In particular, for a generic set of masses, one expects $\varphi_n(s) \asymp n$ for all times $0\leq t\leq O(n^{-1})$ and the geometry of the embedding obtained via \eqref{eq:ODE-embedding} should start moving macroscopically at time $t\asymp n^{-1}$. Still, there exists a special class of weights where the deformation process remains within som \Unif\, class of embedding up to some $\asymp 1$ time. This special case corresponds to the critical variety of doubly-periodic graphs, derived explicitly in \cite{cimasoni-duminil}. It turns out that traveling along this critical variety 
correspond exactly to the fact that the RHS of \ref{eq:ODE-embedding-general} vanishes for almost all translates of the fundamental domain (i.e.\ all but those containing $\frak{p}$ and the boundary of $\Lambda$), allowing to modify the entire full-plane picture in a periodic fashion. When working in bounded domains discretised by $s$-embeddings corresponds to critical periodic Ising models, this allows to link Ising fermionic observable on different critical lattices even before passing to the limit, providing some \emph{universality at discrete level}. This is in sharp contrast with the main spirit of the universality proofs developed to date, which all go via computing the scaling limit and then noting that they are indeed universal. This remark will be detailed in \cite{Mah25b}
\end{remark}

\section{The embedding of the near-critical i.i.d. model via stochastic differential equations}\label{sec:iid-deformation}

\subsection{Deforming the embedding using SDE}

In this section we study in greater details the impact of the embedding deformation induced by \eqref{eq:ODE-embedding} when passing from the critical square lattice to the t-weakly random model. The analysis is quite similar to the one made in Section \ref{sec:massive-deformation}, but this time we use independent brownian motions to change continuously the coupling constants in \eqref{eq:ODE-embedding}. In particular, the deformation equation becomes a Stochastic Differential Equation (SDE) instead of a standard ODE, taking now into account the corrections terms appearing in the Ito Lemma. One works once again in the box $\Lambda_{2n}$, centred at the origin and starts at time $t=0$ from a uniform critical model $\theta_k=\frac{\pi}{4}$ with its canonical embedding $\cS^{(0)}$ made of squares of mesh size $\sqrt{2} \cdot\frac{1}{n}$. Each Ising weight moves from $\frac{\pi}{4}$ to $\frac{\pi}{4}+B^{(e_k)}_t$, where the Brownian motions $B^{(e_k)}_t$ are independent from each other. The author is grateful to Dmitry Chelkak for suggesting to use Brownian motions as a continuous deformation process instead of using a collection of centred i.i.d. $\pm1$ sorted at time $t=0$. We still work on the square lattice $S=\mathbb{Z}^2$ and keep the notations of Section \ref{sec:deformation}, with the corners $(a^{+}_k)_{k\in \Lambda_n} $ and $(c^{+}_k)_{k\in \Lambda_n}$ as in Lemma \ref{lem:deformation-discrete}. 
Under the probability measure $\mathbf{P}$, fix a collection of independent standard Brownian motions $\omega\mapsto ((B_t^{(e)}(\omega))_{t\geq 0})_{e\in E(\mathbb{Z}^2)}$ and a collection a variances $(\sigma_{e}^2)_{e\in E(\mathbb{Z}^2)}$.
For $t\geq 0$ a small enough time parameter (which depends on $\omega$), consider the family of Ising models $(S,(x^{(t)}_{\omega,e})_{e\in E(\mathbb{Z}^2)},(\sigma_e)_{e\in E(\mathbb{Z}^2)})_{t\geq 0}$, where the angle at the edge $e\in E(\mathbb{Z}^2)$ in \eqref{eq:x=tan-theta} is given by
\begin{equation}
	\theta^{(t)}_{e,\omega}:=\theta^{(0)}_{e} + \sigma_{e}\cdot B^{(e)}_{t}(\omega).
\end{equation}
The following lemma allows to construct $\mathbf{P}$-almost surely a family of $s$-embeddings of the Ising model $(S_{\omega}(t))_{t\geq 0}=(S,(x^{(t)}_{\omega,e})_{e\in \mathbb{Z}^2})_{t\geq 0}$ as a solution to some SDE with initial condition $\cS^{(0)}$. 
\begin{lemma}\label{lem:SDE-embedding}
Assume that the collection of variances $(\sigma_e^2)_{e\in E(\mathbb{Z}^2)}$ vanishes outside of some finite box $\Lambda$. Consider the stochastic differential system defined on corners of $ \Upsilon^{\times} \cap  \Lambda  $, whose initial condition is given by $\cY^{(0)}$ and whose dynamic is given for any $\frak{p} \in \Lambda $ by $ \cY^{(s)}_{\omega}(\frak{p}):=\cY(\frak{p},\omega,s)$
\begin{align}\label{eq:SDE-embedding}
		d \cY^{(s)}_{\omega} (\frak{p})=& \frac{1}{2} \sum\limits_{z_k \in \Lambda} \sigma_k\Bigg[ \cY^{(s)}_{\omega}(c^+_k) \langle \chi_{\frak{p}}\chi_{a^+_k} \rangle_{S_{\omega}(s)} -  \cY^{(s)}_{\omega}(a^+_k) \langle \chi_{\frak{p}}\chi_{c^+_k} \rangle_{S_{\omega}(s)} \Bigg] dB^{(e_k)}_s \nonumber \\
		& + \frac{1}{4} \sum\limits_{z_k \in \Lambda}   \sigma_k^2\Bigg( \frac{\cos(\theta^{(s)}_{e_k,_{\omega}})}{\sin(\theta^{(s)}_{e_k,_{\omega}})} -\frac{\mathbb{E}_{S_{\omega}(s)}[\varepsilon_{e_k}]}{\sin(\theta^{(s)}_{e_k,_{\omega}})}  
 \Bigg)\times \cdots \nonumber \\
 & \cdots \times \Bigg[ \cY^{(s)}_{\omega}(c_k^{+})\langle \chi_{\frak{p}}\chi_{a_k^{+}} \rangle_{S_{\omega}(s)} - \cY^{(s)}_{\omega}(a_k^{+})\langle \chi_{\frak{p}}\chi_{c_k^{+}} \rangle_{S_{\omega}(s)} ds  \Bigg].\end{align}
Then $\mathbf{P}$-almost surely :
\begin{enumerate}
	\item There exist $T_0(\Lambda,\omega)>0$ such that there exist a solution to the SDE \eqref{eq:SDE-embedding} on $[0;T_0(\Lambda,\omega)] $.
	\item For any time $0\leq s \leq T_0 (\Lambda,\omega)$, the propagator $\cY_{\omega}^{(s)} $ (seen as the solution to \eqref{eq:SDE-embedding}) is an s-embedding $\cS^{(s)}_{\omega}$ of the Ising model $S_{\omega}(s)$.
\end{enumerate}
\end{lemma}
As we did in the deterministic case, let us derive once again the intuition behind Lemma \ref{lem:SDE-embedding}. The easiest example to guess which SDE constructs a correct family of $s$-embeddings is the case where all the variances in $\Lambda$ vanish, except for the one at $e_0$ which we assume to be unitary. One can now apply the discrete deformation Lemma \ref{lem:deformation-discrete} with $\hat{\theta}_0-\theta_0=dB_t$ for some small enough time $t$, taking the formal identity $(dB_t)^2=dt$. Differentiating the fermions $\langle \chi_{\frak{p}} \chi_{\frak{q}} \rangle$ with respect to their coupling constants (see Appendix \ref{app:derivative-fermions}) one gets  for $\frak{p}\in \Lambda $
\begin{align*}
	y_{a_0}&=\frac{1}{2}\cY(c_0^+)dB_t + \frac{1}{4}\cY(a_0^+)dt + o(dt) \\
	y_{c_0}&=-\frac{1}{2}\cY(a_0^+)dB_t + \frac{1}{4}\cY(c_0^+)dt  + o(dt) \\
	\langle \chi_{\frak{p}}\chi_{a^+_k} \rangle_{S(dB_t)} &=\langle \chi_{\frak{p}}\chi_{a^+_k} \rangle_{S(0)} + \frac{1}{2\sin(\theta_0)} \Big( \langle \chi_{\frak{p}}\chi_{a^+_k} \rangle_{S(0)} \mathbb{E}_{S(0)}[\varepsilon_{e_0}] - \langle \chi_{\frak{p}}\chi_{a^+_k}\varepsilon_{e_0} \rangle_{S(0)} \Big)dB_t + o(dB_t)\\
	\langle \chi_{\frak{p}}\chi_{c^+_k} \rangle_{S(dB_t)} &=\langle \chi_{\frak{p}}\chi_{c^+_k} \rangle_{S(0)} + \frac{1}{2\sin(\theta_0)} \Big( \langle \chi_{\frak{p}}\chi_{c^+_k} \rangle_{S(0)} \mathbb{E}_{S(0)}[\varepsilon_{e_0}] - \langle \chi_{\frak{p}}\chi_{c^+_k}\varepsilon_{e_0} \rangle_{S(0)} \Big)dB_t + o(dB_t).
	&\end{align*}
Plugging into the local identities around the quad $z_0$ for $\frak{q}\mapsto \langle \chi_{\frak{p}}\chi_{\frak{q}} \rangle_{S(0)}$ (see Appendix \ref{app:derivative-embedding}) all together ensures that provided $t$ is small enough, there exist an $s$-embedding of the Ising model $S(dB_t)$ that satisfies 
\begin{align*}
	\cY^{(dB_t)}(\frak{p})&=\cY^{(0)}(\frak{p}) + \frac12\Bigg[ \cY^{(0)}(c^+_0) \langle \chi_{\frak{p}}\chi_{a^+_0} \rangle_{S(0)} -  \cY^{(0)}(a^+_0) \langle \chi_{\frak{p}}\chi_{c^+_0} \rangle_{S(0)} \Bigg] dB_t\\
	& +\frac{1}{4}\Bigg( \frac{\cos(\theta_{0})}{\sin(\theta_{0})} -\frac{\mathbb{E}_{S(0)}[\varepsilon_{e_0}]}{\sin(\theta_{0})}  
 \Bigg)\Bigg[ \cY^{(0)}(c_0^{+})\langle \chi_{\frak{p}}\chi_{a_0^{+}} \rangle_{S(0)} - \cY^{(0)}(a_0^{+})\langle \chi_{\frak{p}}\chi_{c_0^{+}} \rangle_{S(0)} \Bigg] dt + o(dt).
\end{align*}
When modifying multiple edge-weights, the standard composition rules and the independence of the Brownian motions hints at the SDE of Lemma \ref{lem:SDE-embedding}. One could in principle follow closely the spirit of the proof of Lemma \ref{lem:ODE-embedding} to prove that the local embedding relation in the stochastic environment satisfies itself some simple SDE and vanishes at time $t=0$. Instead, we prefer use the result of Lemma \ref{lem:ODE-embedding} when approximating Brownian motions by piecewise linear functions 
\begin{proof}[Proof of Lemma \ref{lem:SDE-embedding}]
Define the completed filtration $\mathcal{F}^{\Lambda}_t:=\sigma (B^{(e_k)}_s, 0\leq s \leq t,  e_k\in \Lambda_n) $. The functions of the kind $\langle \chi_{\frak{p}}\chi_{\frak{q}} \rangle_{S(s)} $ are Lipchitz in each coupling constants coordinates and all $\mathcal{F}_s$ measurable. Therefore, standard SDE theory ensures the existence of a strong solution $\cY^{(s)}$ to SDE \eqref{eq:SDE-embedding} with $\mathbf{P}$-almost sure continuous trajectories up to some positive time $T_0(\Lambda,\omega)$. For $\omega$ in a set of $\mathbf{P}$-probability $1$,  the associated continuous trajectory is the continuous function in a neighbourhood of $0^+ $ denoted by
\begin{equation}
	s\mapsto \cY^{(s)}_{\omega}(\frak{p}):=\cY(\frak{p},\omega,s).
\end{equation}
Let us now pass to the proof of the local 3 terms identity \eqref{eq:3-terms} in the stochastic environment. Let us state once again that the chore difficulty of the proof is to check that   $\mathbf{P}$-almost surely, the solution $ \cY^{(s)}_{\omega}$ to \eqref{eq:SDE-embedding} indeed provides a family of $s$-embedding $(\cS^{(s)}_{\omega})_{0\leq s \leq T_0(\Lambda,\omega)}$ of the Ising model s $(S_{\omega}(s))_{0\leq s \leq T_0(\Lambda,\omega)}$, which is done by approximating piecewise linearly Brownian trajectories. 

\textbf{Toy example: modifying only one Ising weight.} 

In order to illustrate the strategy and introduce the reasoning in a lighter fashion, let us first work out a toy example where all the variances inside $\Lambda$ vanish except the one $\sigma_0$ at the edge $e_0$. Then $\mathbf{P}$-almost surely, the Brownian trajectory $s \mapsto B^{(e_0)}_{s}(\omega)$ is a continuous function and therefore bounded by $\frac{\pi}{8}$ on some non-trivial interval $[0;T(\omega)]$. This defines a family of full-plane Ising models $(S_{\omega}(s))_{0\leq s \leq T(\omega)}$ where the angle parameters are given by $\theta_k^{(s)}(\omega):=\frac{\pi}{4} + \mathds{1}_{k=0} \cdot B^{(e_0)}_{s}(\omega)$.

\textbf{Step 1: Linear approximation of the Brownian motion.} Fix $L\in \mathbb{N}^{\star}$. One can split $[0;T(\omega)]$ into $L$ consecutive segments $[s_j,s_{j+1}]$ of width $T(\omega).L^{-1}$ and define the piecewise-affine interpolation $s\mapsto \widetilde{B}^{(e_0)}_{s,L}(\omega)$ of $s\mapsto B^{(e_0)}_{s}(\omega) $ on the points $(s_j)_{0\leq j \leq L}$. The local slope of the piecewise-affine function $\widetilde{B}^{(e_0)}_{s,L}$ in the segment $[s_j,s_{j+1}]$ is denoted by $m_{0}^{(s_j)}(\omega):=L^{-1}\big(B^{(e_0)}_{s_{j+1}}(\omega)-B^{(e_0)}_{s_{j}}(\omega)\big)$. One can also define the piecewise-affine approximation of the Ising angles by setting $\widetilde{\theta}_{k,L}^{(s)}(\omega):=\frac{\pi}{4} + \mathds{1}_{k=0} \cdot \widetilde{B}^{(e_0)}_{s,L}(\omega)$, together with the family of Ising models $(\widetilde{S}_{\omega,L}(s))_{0\leq s\leq T(\omega)}$. Since angles $\widetilde{\theta}_{k,L}^{(s)}(\omega)$ are piecewise linear, one can apply right away \eqref{eq:ODE-embedding-general}, there exist a family propagators $(\cY^{(s)}_{\omega,L})_{0\leq s \leq T(\omega)}$ associated to the Ising models $(\widetilde{S}_{\omega,L})_{0\leq s\leq T(\omega)}$. For $\frak{p}\in \Lambda$, $0\leq t\leq T(\omega)$ and $j^{[t]}= \left\lfloor t\cdot L\cdot T(\omega)^{-1} \right\rfloor $, one has
\begin{equation}\nonumber
	\cY^{(t)}_{\omega,L}(\frak{p})-\cY^{(0)}(\frak{p}) = \Big( \sum\limits_{j=0}^{j^{[t]}-1}\cY^{(s_{j+1})}_{\omega,L}(\frak{p})-\cY^{(s_{j})}_{\omega,L}(\frak{p}) \Big) + \cY^{(t)}_{\omega,L}(\frak{p})-\cY^{(s_{j^{[t]}})}_{\omega,L}(\frak{p})
\end{equation}
For each $0\leq j\leq j^{[t]}-1$, one can write
\begin{align*}
	\cY^{(s_{j+1})}_{\omega,L}(\frak{p})-\cY^{(s_{j})}_{\omega,L}(\frak{p})&=\int\limits_{s_j}^{s_{j+1}} \frac{d}{ds} \cY^{(s)}_{\omega,L}(\frak{p}) ds\\
&= \frac{B^{(e_0)}_{s_{j+1}}(\omega)-B^{(e_0)}_{s_{j}}(\omega)}{2}  \times \Big( \cY^{(s_j)}_{\omega,L}(c_0^{+})\langle \chi_{\frak{p}}\chi_{a_0^{+}} \rangle_{S_{\omega}(s_j)} - \cY^{(s_j)}_{\omega,L}(a_0^{+})\langle \chi_{\frak{p}}\chi_{c_0^{+}} \rangle_{S_{\omega}(s_j)} \Big) \\
	&\quad +  \frac{\Big(B^{(e_0)}_{s_{j+1}}(\omega)-B^{(e_0)}_{s_{j}}(\omega)\Big)^2}{4}  \times  \Bigg( \frac{\cos(\theta^{(s_j)}_{0})}{\sin(\theta^{(s_j)}_{0})} -\frac{\mathbb{E}_{S_{\omega}(s_j)}[\varepsilon_{e_0}]}{\sin(\theta^{(s_j)}_{0})} \Big) \Bigg)\times \cdots  \\
	& \quad \cdots \times   \Bigg(  \cY^{(s_j)}_{\omega,L}(c_0^{+})\langle \chi_{\frak{p}}\chi_{a_0^{+}} \rangle_{S_{\omega}(s_j)} - \cY^{(s_j)}_{\omega,L}(a_0^{+})\langle \chi_{\frak{p}}\chi_{c_0^{+}} \rangle_{S_{\omega}(s_j)}\Bigg) \\
	&\quad +O(L^{-\frac{3}{2}}).
\end{align*}
In the above computation, when passing from the first to the second line, one uses the first two terms in the expansion of $\frac{d}{ds} \cY^{(s)}_{\omega,L}(\frak{p})$, which are respectively given in the case where of the masses are constant inside $\Lambda$ by \eqref{eq:ODE-embedding} and \eqref{eq:second-derivative-fermion-for-Ito-formula}. It is not hard to check the announced error $O(L^{-\frac{3}{2}})$, which is uniform when computed in $\mathcal{L}^2$ norm.

\textbf{Step 2: Identifying the limiting process}. The trajectories of the Brownian motion and of $s\mapsto \cY^{(s)}_{\omega,L}$ are almost surely equicontinuous while the Kadanoff-Ceva correlator and the energy density at $e_0 $ are Lipschitz with respect to angle $\theta_0$. Therefore, one can repeat the main steps of the proof of the Ito formula \cite{le2016brownian}. In particular all the variables are uniformly bounded in $\mathcal{L}^2$ uniformly in $L$, and any sub-sequential limit $\cY^{(t)}$ limit of the process $\Big( (\cY^{(t)}_{\omega,L}\Big)_{L\geq 1}$ satisfies almost surely for any $\frak{p}\in \Lambda$
\begin{align}
	\cY^{(t)}(\frak{p})-\cY^{(0)}(\frak{p})=&\frac{1}{2}\int_{0}^{t}\sigma_0\Bigg[ \cY^{(s)}(c^+_0) \langle \chi_{\frak{p}}\chi_{a^+_0} \rangle_{S_{\omega}(s)} -  \cY^{(s)}(a^+_0) \langle \chi_{\frak{p}}\chi_{c^+_0} \rangle_{S_{\omega}(s)} \Bigg] dB_s \nonumber\\
	& \quad + \frac{1}{4} \int_{0}^{t} \sigma_0^2\Bigg( \frac{\cos(\theta^{(s)}_{0})}{\sin(\theta^{(s)}_{0})} -\frac{\mathbb{E}_{S_{\omega}(s)}[\varepsilon_{e_0}]}{\sin(\theta^{(s)}_{0})}  \Bigg) \times \cdots \nonumber\\
& \quad  \cdots \times  \Bigg( \cY^{(s)}(c_0^{+})\langle \chi_{\frak{p}}\chi_{a_0^{+}} \rangle_{S_{\omega}(s)} - \cY^{(s)}(a_0^{+})\langle \chi_{\frak{p}}\chi_{c_0^{+}} \rangle_{S_{\omega}(s)} \Bigg) ds.\label{eq:fermionic-SDE-embedding-within-proof}
\end{align}
Before passing to the limit, $\mathbf{P}$-almost surely, the construction ensures that at the points $s_j$ (which depend on $L$), for any $0\leq s\leq T(\omega)$, the fermionic identities \eqref{eq:3-terms} are satisfied by $\cY^{(s)}_{L,\omega}$ for the Ising model $S_\omega(s)$ (this Ising model coincides with $\widetilde{S}_\omega(s)$ on the points $s_j$). As a counterpart to this information for the discretised $\cY^{(s)}_{L,\omega}$, one can use the almost sure continuity of the Brownian trajectory and the almost sure continuity of the trajectories of the solutions of \eqref{eq:fermionic-SDE-embedding-within-proof}. All together, this ensures that, $\mathbf{P}$-almost surely, for any $0\leq s \leq T(\omega)$, the fermion $\cY^{(s)}_{\omega}$ satisfies all the identities \eqref{eq:3-terms} for the Ising model $S_{\omega}(s)$, as a counterpart made.

\textbf{General case: modifying all the coupling constants in $\Lambda$.}

For any edge $e_k\in \Lambda$, $\mathbf{P}$-almost surely the trajectory $s \mapsto B^{(e_k)}_{s}(\omega)$ is a continuous function and is bounded by $\frac{\pi}{8}$ on some non-trivial interval $0\leq s\leq T(\Lambda,\omega)$, where $T(\omega,\Lambda)$ is universal among edges in $\Lambda$. This allows to construct once again the family of Ising models $(S_{\omega}(s))_{0\leq s \leq T(\omega)}$ whose parameters are given by $\theta_k^{(s)}(\omega)=\frac{\pi}{4} + B^{(e_k)}_{s}(\omega)$ inside $\Lambda$. Splitting $[0;T(\Lambda,\omega)]$ into $L$ consecutive even segments $[s_j,s_{j+1}]$, one defines the piecewise-affine approximations $s\mapsto \widetilde{B}^{(e_k)}_{s,L}(\omega)$ that interpolates $s\mapsto B^{(e_k)}_{s}(\omega) $ on the points $(s_j)_{0\leq j \leq L}$ and the angles $\widetilde{\theta}_{k,L}^{(s)}(\omega)=\frac{\pi}{4} +  \widetilde{B}^{(e_k)}_{s,L}(\omega)$. One can then interpolate the Ising model $(S_{\omega}(s))_{0\leq s \leq T(\omega)}$ by $(\widetilde{S}_{\omega,L}(s))_{0\leq s\leq T(\Lambda,\omega)}$ on the points $s_j$. As the functions $\widetilde{\theta}_{k,L}^{(s)}$ are piecewise-linear, Lemma \ref{lem:ODE-embedding} constructs a family of fermions $(\cY^{(s)}_{\omega,L})_{0\leq s \leq T(\Lambda,\omega)}$ that satisfy \eqref{eq:3-terms} for $(\widetilde{S}_{\omega,L})_{0\leq s\leq T(\Lambda,\omega)}$. For $\frak{p}\in \Lambda$, $0\leq t\leq T(\omega)$ and $j^{[t]}= \left\lfloor t\cdot L\cdot T(\Lambda,\omega)^{-1} \right\rfloor $, one has this time
\begin{align*}
	\cY^{(s_{j+1})}_{\omega,L}(\frak{p})-\cY^{(s_{j})}_{\omega,L}(\frak{p})&= \sum\limits_{z_k\in \Lambda} \Big(\frac{B^{(e_k)}_{s_{j+1}}(\omega)-B^{(e_k)}_{s_{j}}(\omega)}{2} \Big) \times \cdots \\
	& \quad \quad \quad  \cdots \times  \Big( \cY^{(s_j)}_{\omega,L}(c_k^{+})\langle \chi_{\frak{p}}\chi_{a_k^{+}} \rangle_{S_{\omega}(s_j)} - \cY^{(s_j)}_{\omega,L}(a_k^{+})\langle \chi_{\frak{p}}\chi_{c_k^{+}} \rangle_{S_{\omega}(s_j)} \Big) \\
	&\quad +\sum\limits_{z_k\in \Lambda_n}  \frac{\Big( B^{(e_k)}_{s_{j+1}}(\omega)-B^{(e_k)}_{s_{j}}(\omega)\Big)^2}{4}  \times  \Bigg( \frac{\cos(\theta^{(s)}_{k})}{\sin(\theta^{(s)}_{k})} -\frac{\mathbb{E}_{S_{\omega}(s_j)}[\varepsilon_{e_k}]}{\sin(\theta^{(s)}_{k})}  \Bigg)\times \cdots  \\
	& \quad \cdots \times   \Bigg(  \cY^{(s_j)}_{\omega,L}(c_k^{+})\langle \chi_{\frak{p}}\chi_{a_k^{+}} \rangle_{S_{\omega}(s_j)} - \cY^{(s_j)}_{\omega,L}(a_k^{+})\langle \chi_{\frak{p}}\chi_{c_k^{+}} \rangle_{S_{\omega}(s_j)}\Bigg) \\
& + \sum\limits_{z_k\neq z_r \in \Lambda} \Bigg(B^{(e_k)}_{s_{j+1}}(\omega)-B^{(e_k)}_{s_{j}}(\omega)\Bigg) \Bigg(B^{(e_r)}_{s_{j+1}}(\omega)-B^{(e_r)}_{s_{j}}(\omega)\Bigg) \times \cdots \\
&  \quad \cdots \times g^{(\frak{p})}_{k,r}(\omega)\\
	&\quad +O(L^{-\frac{3}{2}})
\end{align*}
where $O()$ is uniform in $\mathcal{L}^2$ norm and the continuous function  $g^{(\frak{p})}_{k,r}$, are given in \eqref{eq:g_{k,r}} for the Ising model $S_{\omega}(s_j)$. Again, the Brownian trajectories and $s\mapsto \cY^{(s)}_{\omega,L}$ are almost surely equicontinuous, the Kadanoff-Ceva correlators and the energy densities are Lipschitz with respect the angle parameters in $\Lambda$. One can mimic this time the proof of the multi-dimensional Ito formula \cite{le2016brownian}. All the variables $\big( \cY^{(t)}_{\omega,L}\big)_{L\geq 1}$ are uniformly bounded in $\mathcal{L}^2$ and any sub-sequential limit of the processes $\big( \cY^{(t)}_{\omega,L}\big)_{L\geq 1}$ as $L\to \infty $ satisfies almost surely
\begin{align*}
	\cY^{(t)}(\frak{p})-\cY^{(0)}(\frak{p})=&\frac{1}{2}\int_{0}^{t} \sum\limits_{z_k\in \Lambda}\sigma_k\Bigg[ \cY^{(s)}(c^+_k) \langle \chi_{\frak{p}}\chi_{a^+_k} \rangle_{S_{\omega}(s)} -  \cY^{(s)}(a^+_k) \langle \chi_{\frak{p}}\chi_{c^+_k} \rangle_{S_{\omega}(s)} \Bigg] dB^{(e_k)}_s\\
	& \quad + \frac{1}{4}\int_{0}^{t} \sum\limits_{z_k\in \Lambda}  \sigma_k^2\Bigg( \frac{\cos(\theta^{(s)}_{k})}{\sin(\theta^{(s)}_{k})} -\frac{\mathbb{E}_{S_{\omega}(s)}[\varepsilon_{e_k}]}{\sin(\theta^{(s)}_{k})}  \Bigg) \times \cdots \\
& \quad  \cdots \times  \Bigg( \cY^{(s)}(c_k^{+})\langle \chi_{\frak{p}}\chi_{a_k^{+}} \rangle_{S_{\omega}(s)} - \cY^{(s)}(a_k^{+})\langle \chi_{\frak{p}}\chi_{c_k^{+}} \rangle_{S_{\omega}(s)} \Bigg) ds.
\end{align*}
Comparing the discrete expansion at a fixed $L$, the off-diagonal terms $z_k\neq z_r$ containing terms of the form $\big(B^{(e_k)}_{s_{j+1}}(\omega)-B^{(e_k)}_{s_{j}}(\omega)\big) \big( B^{(e_r)}_{s_{j+1}}(\omega)-B^{(e_r)}_{s_{j}}(\omega)\big) $ disappear in the limit, as $\mathrm{d}\big[ B^{(e_k)},B^{(e_r)}\big]_s=0$ for $k\neq r$ since the Brownian motions are independent. The fact that $\mathbf{P}$-almost surely, $\cY^{(t)}_{\omega}$ is indeed an $s$-embedding of $S_{\omega}(t)$, can be checked exactly as for the toy example.
\end{proof}

\subsection{Some non-optimal bound when deforming via Brownian motions}\label{sub:non-optimal-deformation}

We are now in a position to use the construction of Lemma \ref{lem:SDE-embedding} to derive quantitative bounds on the deviation of an $s$-embedding from its initial value when the coupling constants are perturbed by independent Brownian motions. Once the orders of magnitude of the coefficients in the SDE \eqref{eq:SDE-embedding} are understood, it requires only relatively simple computations involving Brownian motion and local martingales to determine the maximal time during which, with high probability, the embedding can evolve while remaining within some \Unif\, class of Ising models. Before diving into the proof, let us compare this setting with the deterministic massive case studied in Section \ref{sec:massive-deformation}. Let $Y^{(t)}_n$ denote the vector containing all coordinates $\cY^{(t)}(\frak{p})$ of the fermion evaluated at the corners $\frak{p} \in \Lambda_n$. The differential equation \eqref{eq:ODE-embedding-general} can be rewritten as a matrix system of the form $(Y^{(t)}_n)' = A_n(t) Y^{(t)}_n$. For a generic set of bounded masses, the (sharp up to constants) bound \eqref{eq:bound-correlator-distance} implies that $\|A(0)\|$ is of order $n$. Therefore, standard ODE heuristics suggest that for times $t = o(\|A(0)\|^{-1}) = o(n^{-1})$, we should have $Y^{(t)}_n = Y^{(0)}_n \times (1 + o(1))$, while for $t = \|A(0)\|^{-1}$, the vectors $Y^{(t)}_n$ and $Y^{(0)}_n$ should remain comparable up to a constant factor. However, this analogy is misleading in the random case. For each realization $\omega$, the process $\cY^{(s)}_{\omega}(\frak{p})$ is the sum of a local martingale and a finite-variation process. The local martingale contribution (the right-hand side of the first line of the SDE in Lemma \ref{lem:SDE-embedding}) is a stochastic generalization of \eqref{eq:ODE-embedding-general}, where $ds$ is replaced by $dB^{(e)}_s$ terms. Provided the system remains within a \Unif-like setup, the quadratic variation (bracket) of the local martingale at time $s$ is of order $(s \log(n))^{\frac{1}{2}}$. Attempting to apply a reasoning similar to the deterministic case would suggest that the embedding remains stable up to time $\log(n)^{-\frac{1}{2}}$, which would correspond to a much larger admissible deviation for the coupling constants than that stated in Theorem \ref{thm:near-critical-RSW-random}. This discrepancy arises because the leading-order contribution to the deformation of $\cY^{(s)}_{\omega}(\frak{p})$ comes from the second term in Itô’s formula (i.e., the right-hand side of the second line of the SDE in Lemma \ref{lem:SDE-embedding}).  After careful analysis, this term is found to be of order $n^{-\frac{1}{2}} \cdot n s^{\frac{3}{2}}$ at time $s$, up to some $\log(n)$ corrections, which provides the correct order of magnitude for the standard deviation in the random setting, and thus offers insight to the appropriate near-critical random window. This argument is further supported by external analysis presented in \cite{AveMah25}, which we will be recalled. In the proofs, we will replicate many arguments from the deterministic case, this time carefully ensuring that all estimates hold with high probability. Let us begin by breaking down how this control is obtained, starting with the Doob decomposition of the fermion.
 \begin{itemize}
 	\item Consider the fermion $(\cY^{(s)}(\frak p ))_{\frak{p}\in \Lambda_n}$ as a strong solution to SDE of Lemma \ref{lem:SDE-embedding}. For every $\frak{p}\in \Lambda_n$ one can make the Doob decomposition of the process as $\cY^{(s)}(\frak{p}):= \cM^{(s)}(\frak{p})+\cA^{(s)}(\frak{p})$ where 
 \begin{align*}
 	\cM^{(t)}(\frak{p})&:=\frac{1}{2}\int_{0}^{t} \sum\limits_{z_k\in \Lambda_n}\sigma_k\Bigg[ \cY^{(s)}(c^+_k) \langle \chi_{\frak{p}}\chi_{a^+_k} \rangle_{S(s)} -  \cY^{(s)}(a^+_k) \langle \chi_{\frak{p}}\chi_{c^+_k} \rangle_{S(s)} \Bigg] dB^{(e_k)}_s,\\
 	& :=\cM^{(s)}_{\mathbb{R}}(\frak{p}) + i \cM^{(s)}_{i\mathbb{R}}(\frak{p}) \\
 	\cA^{(t)}(\frak{p})&:= \frac{1}{4}\int_{0}^{t} \sum\limits_{z_k\in \Lambda}  \sigma_k^2\Bigg( \frac{\cos(\theta^{(s)}_{k})}{\sin(\theta^{(s)}_{k})} -\frac{\mathbb{E}_{S(s)}[\varepsilon_{e_k}]}{\sin(\theta^{(s)}_{k})} \Big) \Bigg) \times \cdots\\
 	&\quad  \cdots \times  \Bigg( \cY^{(s)}(c_k^{+})\langle \chi_{\frak{p}}\chi_{a_k^{+}} \rangle_{S(s)} - \cY^{(s)}(a_k^{+})\langle \chi_{\frak{p}}\chi_{c_k^{+}} \rangle_{S(s)} \Bigg) ds\\
 	& :=\cA^{(s)}_{\mathbb{R}}(\frak{p}) + i \cA^{(s)}_{i\mathbb{R}}(\frak{p}).
 \end{align*}
 The process $\cM^{(t)}(\frak{p})$ is a local martingale with  $\cM^{(0)}(\frak{p})=0$ and $\cA^{(t)}(\frak{p})$ is a $ \mathcal{F}_t$ predictable finite-variation process. Their respective real and imaginary parts $\cM^{(t)}_{\mathbb{R}}(\frak{p}),\cM^{(t)}_{i\mathbb{R}}(\frak{p})$ and $\cA^{(t)}_{\mathbb{R}}(\frak{p}),\cA^{(t)}_{i\mathbb{R}}(\frak{p})$ have the same properties. The brackets of $\cM^{(t)}_{\mathbb{R}}(\frak{p})$ and $\cM^{(s)}_{i\mathbb{R}}(\frak{p})$ are denoted respectively $[\cM_{\mathbb{R}}(\frak{p})]^{(t)}$ and $[\cM_{i\mathbb{R}}(\frak{p})]^{(t)}$, with an infinitesimal increment given by
 \begin{align*}
 	d[\cM_{\mathbb{R}}(\frak{p})]^{(t)}&= \frac{1}{4} \sum\limits_{z_k\in \Lambda_n}\sigma_k^2 \Bigg( \cY^{(s)}_{\mathbb{R}}(c^+_k) \langle \chi_{\frak{p}}\chi_{a^+_k} \rangle_{S(s)} -  \cY^{(s)}_{\mathbb{R}}(a^+_k) \langle \chi_{\frak{p}}\chi_{c^+_k} \rangle_{S(s)} \Bigg)^2,\\
 	d[\cM_{i\mathbb{R}}(\frak{p})]^{(t)}&= \frac{1}{4} \sum\limits_{z_k\in \Lambda_n}\sigma_k^2 \Bigg( \cY_{i\mathbb{R}}^{(s)}(c^+_k) \langle \chi_{\frak{p}}\chi_{a^+_k} \rangle_{S(s)} -  \cY_{i\mathbb{R}}^{(s)}(a^+_k) \langle \chi_{\frak{p}}\chi_{c^+_k} \rangle_{S(s)} \Bigg)^2,
 \end{align*}
where $\cY$ naturally splits between its real and imaginary parts $\cY_{\mathbb{R}}$ and $\cY_{i\mathbb{R}}$.
 \end{itemize}
 
This allows us to define the stopping time $\mathbf{T}_{\Lambda_n}$ for the filtration $(\mathcal{F}_s)_{s \geq 0}$, given $\mathbf{P}$-almost surely by
\begin{equation}
	\mathbf{T}_{\Lambda_n}(\omega) := \inf_{t \geq 0} \left\{ \cS_{\omega}^{(t)}(\Lambda_n) \notin \Unifff \right\} \wedge \inf_{t \geq 0} \left\{ \exists e \in \Lambda_n,\ |B^{(e)}_{t}| \geq \frac{\pi}{20} \right\}.
\end{equation}
In words, the stopping time $\mathbf{T}_{\Lambda_n}$ is the first moment when, somewhere in $\Lambda_n$, the geometry of the embedding becomes sufficiently distorted from the original square lattice $\cS^{(0)}$, due to the evolution of the coupling constants driven by Brownian motion. Indeed, if any coupling constant deviates macroscopically from the critical value $\frac{\pi}{4}$, the geometric reconstruction of the Ising weight via \eqref{eq:x=tan-theta} implies that the corresponding Brownian motion has already moved macroscopically away from $0$. We now state a (non-optimal) lower bound on the stopping time $\mathbf{T}_{\Lambda_n}$, valid at least with high $\mathbf{P}$-probability. In particular, it ensures that, with high probability, all Brownian motions $B^{(e_k)}_s$ can be run up to time $t \asymp n^{-1}$. This allows us to construct, with high probability, an embedding whose coupling constants are i.i.d.\ and typically deviate from $\frac{\pi}{4}$ by $\asymp n^{-\frac{1}{2}}$ at each edge in $\Lambda_n$, while still remaining within a \Unif\, class of embeddings—and thus within the space of (near-)critical Ising models. If, on the other hand, one were to replace the deviations $B^{(e_k)}_s$ by $|B^{(e_k)}_s|$, the resulting model would be off-critical by a significant margin. The following provides the first (non-optimal) lower bound on $\mathbf{T}_{\Lambda_n}$: \begin{proposition}\label{prop:bound-T_n-non-optimal}
 	Assume that all the variances in $\Lambda_n$ are unitary and all the other variances vanish. Then there exist some small enough universal constant $c_{1,2}>0$ such that
 	\begin{equation}
 		\mathbf{P}\Big[ \mathbf{T}_{\Lambda_n}\leq \frac{c_1}{n} \Big] \leq c_2\exp(-n^{c_2}). 
 	\end{equation}
 \end{proposition}
The proof is quite similar to the arguments presented in Section \ref{sec:massive-deformation}, relying on the interplay between the geometry of the embeddings—which controls the order of magnitude of the two-point Kadanoff–Ceva fermions—and the resulting control on the infinitesimal deviation of the embedding. To streamline the exposition, we will omit details regarding the properness of $\cS^{(s)}_{\omega}$, the distance comparability with $\cS^{(0)}$, and high-probability estimates. These aspects follow directly from the fact that the flow of the SDE \eqref{eq:SDE-embedding} constructs $s$-embeddings in the \Unifff\, class, which permits the application of the bound on the growth rate of correlators \eqref{eq:bound-correlator-distance}. All such arguments can be reproduced verbatim from the proofs in Section \ref{sec:massive-deformation}. Before proceeding with the proof, let us recall some basic facts about local martingales and Brownian motion computations. Let $M^{(t)}$ be a real-valued local martingale with $M^{(0)} = 0$, and let $[M]^{(t)}$ denote its bracket under some probability measure $\mathbb{P}$. Then there exists a universal constant $C > 0$ such that, for any $x, y > 0$,
 \begin{equation}\label{eq:deviation-local-martingale}
 	\mathbb{P}[\sup_{t\geq 0} | M^{(t)} | \geq x  \cap  |[M]^{\infty}|\leq y] \leq \exp(-C\frac{x^2}{y}).
 \end{equation}
 
\begin{proof}[Proof of Proposition \ref{prop:bound-T_n-non-optimal}]
	Let us first state the simple event dichotomy
\begin{align*}
	\Big\{ \mathbf{T}_{\Lambda_n}\leq \frac{c}{n} \Big\} &=  \Big\{ \mathbf{T}_{\Lambda_n}\leq \frac{c}{n} \Big\} \bigcap \Big\{ \forall 0\leq  t\leq \mathbf{T}_{\Lambda_n} \forall e\in \Lambda_n, |B^{(e)}_{t}| \leq \frac{\pi}{20}  \Big\} \\
	& \quad \bigcup \Big\{ \mathbf{T}_{\Lambda_n}\leq \frac{c}{n} \Big\} \bigcap \Big\{ \exists t\leq \mathbf{T}_{\Lambda_n} \exists e\in \Lambda_n, |B^{(e)}_{t}| \geq \frac{\pi}{20}  \Big\}.
\end{align*}	

\textbf{Step 1: Exclude the case where at least one coupling constant drifted too fast}
One can easily see that the event of the second line of the above equation implies that the event
\begin{equation*}
	\Big\{\exists 0\leq  t\leq \frac{c}{n}, \exists e\in \Lambda_n, |B^{(e)}_{t}| \geq \frac{\pi}{20}  \Big\},
\end{equation*}
holds, which happens with a probability  that decays stretch exponentially fast in $n$ (with at some speed depending on $c$ small enough which will be fixed later). Therefore, it remains to estimate the probability of the event of the first line.
 
\textbf{Step 2: Evaluate the order of magnitude of $\cA^{\mathbf{T}_{\Lambda_n}}(\frak{p})$} 
 
Let us work assuming the event $\{\forall 0\leq t\leq \mathbf{T}_{\Lambda_n} , \forall e\in \Lambda_n, |B^{(e)}_{t}| \leq \frac{\pi}{20}  \}$ holds. Then for any corner $\frak{p}\in \Lambda_n$ and any $0 \leq s\leq  \mathbf{T}_{\Lambda_n} $, one can easily upper bound $ |\cot(\theta^{(s)}_{k})) -\mathbb{E}_{S(s)}[\varepsilon_{e_k}] \cdot \sin^{-1}(\theta^{(s)}_{k}))|$ by some universal constant (this is why one requires that the angles didn't deviate too much from their original value $\frac{\pi}{4}$). On the other hand, repeating verbatim the computations of Section \ref{sec:massive-deformation} implies that for any $0\leq s\leq  \mathbf{T}_{\Lambda_n} $
 \begin{equation}
 	 \Bigg| \cY^{(s)}(c_k^{+})\langle \chi_{\frak{p}}\chi_{a_k^{+}} \rangle_{S(s)} - \cY^{(s)}(a_k^{+})\langle \chi_{\frak{p}}\chi_{c_k^{+}} \rangle_{S(s)} \Bigg|=O(\frac{1}{n^{\frac{1}{2}}} \times \frac{1}{n}\times \frac{1}{|\cS^{(0)}(a_k)-\cS^{(0)}(\frak{p}) |}).
 \end{equation}
Therefore summing over all the corners $a_k,c_k\in \Lambda_n $ and integrating up to time $\mathbf{T}_{\Lambda_n} $ one gets that
\begin{equation}
	\cA^{\mathbf{T}_{\Lambda_n}}(\frak{p})= \int_{0}^{\mathbf{T}_{\Lambda_n}} \sum_{z_k\in \Lambda_n}O(\frac{1}{n^{\frac{1}{2}}} \times \frac{1}{n}\times \frac{1}{|\cS^{(0)}(a_k)-\cS^{(0)}(\frak{p}) |}) =O(\frac{1}{n^\frac12} \times n\times \mathbf{T}_{\Lambda_n}).
\end{equation}
Therefore, on the event 
\begin{equation*}
	\mathcal{B}_{n,c}:=\Big\{ \mathbf{T}_{\Lambda_n}\leq \frac{c}{n} \Big\} \cap \Big\{ \forall 0\leq  t\leq \mathbf{T}_{\Lambda_n} \forall e\in \Lambda_n, |B^{(e)}_{t}| \leq \frac{\pi}{20}  \Big\},
\end{equation*}
one has for every $\frak{p}\in \Lambda_n $ that $\cA^{\mathbf{T}_{\Lambda_n}}(\frak{p})=O\Big(\frac{c}{n^{\frac12}}\Big)$.

\textbf{Step 3: Evaluate the order of magnitude of $\cM^{\mathbf{T}_{\Lambda_n}}(\frak{p})$} 

Recall that $\mathbf{T}_{\Lambda_n}$ corresponds to the first time the fermion constructed via Lemma \ref{lem:SDE-embedding} deviates macroscopically enough from $\cY^{(0)}$. The embedding $\cS^{\mathbf{T}_{\Lambda_n}}(\Lambda_n)$ doesn't satisfy \Unifff\,, while $|\Re[\cY^{(0)}(\frak{p})]|,|\Im[\cY^{(0)}(\frak{p})]| \geq \cos(\frac{3\pi}{8})$. Therefore, by almost-sure continuity of the solutions $s\mapsto \cY^{(s)}_{\omega}(\frak{p}) $, there exist at least one corner $\frak{p}\in \Lambda_n$ such that 
\begin{align*}
	|\cY_{\mathbb{R}}^{(\mathbf{T}_{\Lambda_n})}(\frak{p}) -\cY_{\mathbb{R}}^{(0)}(\frak{p})| &=|\cM_{\mathbb{R}}^{(\mathbf{T}_{\Lambda_n})}(\frak{p})+\cA_{\mathbb{R}}^{(\mathbf{T}_{\Lambda_n})}(\frak{p})| \geq \frac{1}{50n^{\frac{1}{2}}} \textrm{ or }\\
	|\cY_{i\mathbb{R}}^{(\mathbf{T}_{\Lambda_n})}(\frak{p}) -\cY_{i\mathbb{R}}^{(0)}(\frak{p})| &=|\cM_{i\mathbb{R}}^{(\mathbf{T}_{\Lambda_n})}(\frak{p})+\cA_{i\mathbb{R}}^{(\mathbf{T}_{\Lambda_n})}(\frak{p})| \geq \frac{1}{50n^{\frac{1}{2}}}.
\end{align*}
In particular, provided $c$ is chosen small enough (but independent from $n$), on the event $\mathcal{B}_{n,c}$, there exist at least one corner $\frak{p}\in \Lambda_n$ such that
\begin{align*}
\Big|\cM_{\mathbb{R}}^{(\mathbf{T}_{\Lambda_n})}(\frak{p})\Big| \geq \frac{1}{4n^{\frac{1}{2}}} &\textrm{ or } \\
\Big|\cM_{i\mathbb{R}}^{(\mathbf{T}_{\Lambda_n})}(\frak{p})\Big| \geq \frac{1}{4n^{\frac{1}{2}}}.
\end{align*}
Moreover, for any $\frak{p}\in \Lambda_n$, $(\cM_{\mathbb{R}}^{(t\wedge \mathbf{T}_{\Lambda_n})}(\frak{p}))_{t\geq 0} $ is a local martingale and 
\begin{equation}
	[\cM_{\mathbb{R}}^{(t\wedge \mathbf{T}_{\Lambda_n})}(\frak{p})]^{\infty}=\int_{0}^{t\wedge \mathbf{T}_{\Lambda_n}} \frac{1}{4} \sum\limits_{z_k\in \Lambda_n} \Bigg( \cY_{\mathbb{R}}^{(s)}(c^+_k) \langle \chi_{\frak{p}}\chi_{a^+_k} \rangle_{S(s)} -  \cY_{\mathbb{R}}^{(s)}(a^+_k) \langle \chi_{\frak{p}}\chi_{c^+_k} \rangle_{S(s)} \Bigg)^2 ds.
\end{equation}
Since the embeddings  $(\cS^{(s)}(\Lambda_n))_{0\leq s\leq t\wedge \mathbf{T}_{\Lambda_n} }$ satisfy almost surely \Unifff\, one has for any $0\leq s\leq t\wedge \mathbf{T}_{\Lambda_n}$
\begin{equation}
	\sum\limits_{a^+_k\in \Lambda_n} |\langle \chi_{\frak{p}} \chi_{a^+_k} \rangle_{S(s)}|^2=O\bigg( \sum\limits_{a_k\in \Lambda_{n}} \frac{1}{n^2} \frac{1}{\cdot|\cS^{(0)}(\frak{p})-\cS^{(0)}(a_k^{+})|^2}\bigg) =O(\log(n)),
\end{equation}
as
\begin{equation}	
\sum\limits_{a^+_k\in \Lambda_n} \frac{1}{n^2} \frac{1}{\cdot|\cS^{(0)}(\frak{p})-\cS^{(0)}(a^+_k)|^2}=O\big( \displaystyle\int_{\mathcal{D}_1 \backslash \mathcal{D}_{\frac{1}{n}}}\frac{1}{|z-\frak{p}|^2}dA(z) \big)=O(\log(n)).
\end{equation}
while for any $0\leq s\leq t\wedge \mathbf{T}_{\Lambda_n}$ and any $c_{k}^+\in \Lambda_n$ one has $\cY^{(s)}(c_{k}^+)=O(n^{-\frac12})$. Using a similar reasoning for the other terms involved in the sum ensures that 
\begin{equation}
	[\cM_{\mathbb{R}}^{(t\wedge \mathbf{T}_{\Lambda_n})}(\frak{p})]^{\infty}=\int_{0}^{t\wedge \mathbf{T}_{\Lambda_n}} O(\frac{\log(n)}{n})ds= O(\mathbf{T}_{\Lambda_n}\frac{\log(n)}{n}).
\end{equation}
The equation upper bound also holds for $[\cM_{i\mathbb{R}}^{(t\wedge \mathbf{T}_{\Lambda_n})}(\frak{p})]^{\infty}$.

\textbf{Step 4: Conclude using large deviation principles for $\cM^{\mathbf{T}_{\Lambda_n}}(\frak{p})$} 
One can now conclude using the previous estimates. Taking some union bounds over the $O(n^2)$ corners $ \frak{p} \in \Lambda_n$ one gets
\begin{align*}
	\mathbf{P}\Bigg[\mathcal{B}_{n,c}\Bigg]&\leq O(n^2)\times \sup\limits_{\frak{p}\in \Lambda_n} \mathbf{P}\Bigg[ \big\{ \mathbf{T}_{\Lambda_n}\leq \frac{c}{n} \big\}\bigcap \big\{ \Big|\sup_{t\geq 0}\cM_{\mathbb{R}}^{(t\wedge \mathbf{T}_{\Lambda_n} )}(\frak{p})\big| \geq \frac{1}{50n^{\frac{1}{2}}} \cdots \\
	&  \quad  \quad \quad\quad \quad \quad \quad \cdots  \bigcap \Big|[\cM_{\mathbb{R}}^{(t\wedge \mathbf{T}_{\Lambda_n})}(\frak{p})]^{\infty}\big| \leq O(\mathbf{T}_{\Lambda_n}\frac{\log(n)}{n})\Bigg]\\
&\quad + O(n^2)\times \sup\limits_{\frak{p}\in \Lambda_n} \mathbf{P}\Bigg[ \big\{ \mathbf{T}_{\Lambda_n}\leq \frac{c}{n} \big\}\bigcap \big\{ \Big|\sup_{t\geq 0}\cM_{i\mathbb{R}}^{(t\wedge \mathbf{T}_{\Lambda_n} )}(\frak{p})\big| \geq \frac{1}{50n^{\frac{1}{2}}} \cdots \\
	&  \quad  \quad \quad\quad \quad \quad \quad \cdots  \bigcap \Big|[\cM_{i\mathbb{R}}^{(t\wedge \mathbf{T}_{\Lambda_n})}(\frak{p})]^{\infty}\big| \leq O(\mathbf{T}_{\Lambda_n}\frac{\log(n)}{n})\Bigg].
\end{align*}
One can now see that
\begin{equation}\nonumber
	\mathbf{P}\Bigg[ \big|\sup_{t\geq 0}\cM_{\mathbb{R}}^{(t\wedge \mathbf{T}_{\Lambda_n} )}(\frak{p})\Big| \geq \frac{1}{50n^{\frac{1}{2}}} \bigcap \Big|[\cM_{\mathbb{R}}^{t\wedge (\mathbf{T}_{\Lambda_n})}(\frak{p})]^{\infty}\Big| \leq O(\frac{c\log(n)}{n^2})\Bigg]
\end{equation}
decays stretch exponentially fast in $n$ using \eqref{eq:deviation-local-martingale}, as long as $c$ has been chosen small enough in Step $2$. A similar reasoning concerning $\cM_{i\mathbb{R}}^{(t\wedge \mathbf{T}_{\Lambda_n} )}(\frak{p})$ concludes the proof.
\end{proof}

\subsection{Improving the lower bound of $\mathbf{T}_{\Lambda_n} $ when running \eqref{eq:SDE-embedding}.}
The goal of this section is to derive the optimal order of magnitude by which one can perturb the coupling constants in a Brownian fashion such that the stochastic flow of $s$-embeddings remains within the class of (near-)critical lattices. We show that, up to logarithmic corrections, the SDE \eqref{eq:SDE-embedding} can be run up to a time of order $\asymp n^{-\frac{1}{3}}$ while staying within a \Unif\, class of embeddings—thus transforming the deterministic critical window of size $n^{-1}$ into its cube root in the weakly random setting. The notion of a near-critical window in a random environment is informal and not rigorously defined here, but in our context, it refers to the largest standard deviation of the coupling constants for which \emph{the entire conformal structure remains stable}. In particular, at every scale, crossing probabilities in annuli should remain uniformly bounded away from $0$ and $1$. This does not exclude the possibility that the SDE \eqref{eq:SDE-embedding} may average the conformal structure at some mesoscopic scale (i.e., polynomial in $n$ inside the box $\Lambda_n$), with fluctuations in the origami map of polynomial order within $\Lambda_n$. From a technical perspective, all the tools—such as criticality, precompactness, and the asymptotic behavior of the two-point fermion—developed in \cite{MahPar25a, Mah23, Che20} remain applicable far beyond the \Unif\, setup. This raises the question of whether there exists a broader random scaling window that only governs macroscopic events in $\Lambda_n$, rather than at every scale. Coming back to our hands-on problem, let us now state the following proposition, which serves as the main input for proving Theorem \ref{thm:near-critical-RSW-random}.
\begin{proposition}\label{prop:bound-T_n-optimal}
 	Assume that all the variances in $\Lambda_n$ are unitary and all the other variances vanish. Then there exist some small enough universal constants $c_{1}>0$ such that
 	\begin{equation}
 		\mathbf{P}\Big[ \mathbf{T}_{\Lambda_n}\leq \frac{c_1}{n^\frac{2}{3}\log^{\frac{1}{3}}(n)} \Big] \leq O(\frac{1}{n^4}). 
 	\end{equation}
 \end{proposition}
 \begin{proof}
This proof closely follows the argument of Proposition \ref{prop:bound-T_n-non-optimal}, and we retain exactly the same notations. We highlight only the main additional ingredient that allows the SDE \eqref{eq:SDE-embedding} to be run for a longer time. Since the bound \eqref{eq:bound-correlator-distance} is sharp up to a constant, it is not difficult to see that nearly all the estimates used in the proof of Proposition \ref{prop:bound-T_n-non-optimal} are optimal up to constant factors. However, there remains room for improvement. In that proof proof, we crudely bounded
\[
\left|\cot(\theta^{(s)}_{k}) - \mathbb{E}_{S(s)}[\varepsilon_{e_k}] \cdot \sin^{-1}(\theta^{(s)}_{k})\right|
\]
by a uniform constant. At time $t = 0$, for the homogeneous critical square lattice, we have for every edge $e_k \in \Lambda_n$:
\begin{equation}
	\cot(\theta^{(0)}_{k})=1 \textrm{ and } \frac{\mathbb{E}_{S(0)}[\varepsilon_{e_k}]}{\sin(\theta^{(0)}_{k})}=\frac{\frac{\sqrt{2}}{2}}{\frac{\sqrt{2}}{2}}=1,
\end{equation}
as $\theta^{(0)}_{k}=\frac{\pi}{4}$ and it is well known \cite{mccoy-wu-book} that 
$\mathbb{E}_{S(0)}[\varepsilon_{e_k}]=\frac{\sqrt{2}}{2} $. Therefore, for edge edge $e_k \in \Lambda_n$, the random process $ |\cot(\theta^{(s)}_{k})) -\mathbb{E}_{S(s)}[\varepsilon_{e_k}] \cdot \sin^{-1}(\theta^{(s)}_{k}))|$ vanishes at time $t=0$. The improvement we use lies into showing that this brutal constant bound can be replaced by some functions growing as $\sqrt{s}$ at time $s$, allowing to keep the geometry of the embedding constructed via \eqref{eq:SDE-embedding} for a larger amount of time. Let us detail this result.

\textbf{Step 0: Evaluate accurately $ |\cot(\theta^{(s)}_{k})) -\mathbb{E}_{S(s)}[\varepsilon_{e_k}] \cdot \sin^{-1}(\theta^{(s)}_{k}))|$}

One can apply the Ito formula to $\mathbb{E}_{S(s)}[\varepsilon_{e_{k}}]$ and deduce that
\begin{equation*}
	\mathrm{d} \Big(\mathbb{E}_{S(s)}[\varepsilon_{e_{k}}]\Big)=\sum\limits_{e_r \in \Lambda_n} \frac{\partial}{\partial_{\theta_r}}\mathbb{E}_{S(s)}[\varepsilon_{e_{k}}] dB^{(e_r)}_s + \frac{1}{2} \sum\limits_{e_r \in \Lambda_n} \frac{\partial^2}{\partial_{\theta_r}^2}\mathbb{E}_{S(s)}[\varepsilon_{e_{k}}] ds.
\end{equation*}
The correlation computations recalled in Appendix \ref{app:derivative-fermions} allow to deduce some (sharp up to constant) estimates on the energy density as long as all the Ising angles didn't deviate too much from $\frac{\pi}{4}$, which read as
\begin{align*}
	\frac{\partial}{\partial_{\theta_r}}\mathbb{E}_{S(s)}[\varepsilon_{e_{k}}]&=O\Bigg( \mathbb{E}_{S(s)}[\varepsilon_{e_{k}}]\mathbb{E}_{S(s)}[\varepsilon_{e_{r}}]-\mathbb{E}_{S(s)}[\varepsilon_{e_{k}}\varepsilon_{e_{r}}] \Bigg),\\
	\frac{\partial^2}{\partial_{\theta_r}^2}\mathbb{E}_{S(s)}[\varepsilon_{e_{k}}]&=O\Bigg( \mathbb{E}_{S(s)}[\varepsilon_{e_{k}}]\mathbb{E}_{S(s)}[\varepsilon_{e_{r}}]-\mathbb{E}_{S(s)}[\varepsilon_{e_{k}}\varepsilon_{e_{r}}] \Bigg).
\end{align*}
One can now use \cite[Theorem 1.3]{MahPar25a} that ensures that for any $s$-embedding that satisfies $\Unifff$ (and whose distances are comparable up to some universal constant to those in $\cS^{(0)}$) one has
\begin{equation}
	\Bigg|\mathbb{E}_{S(s)}[\sigma_{e_{k}}]\mathbb{E}_{S(s)}[\sigma_{e_{r}}]-\mathbb{E}_{S(s)}[\sigma_{e_{k}}\sigma_{e_{r}}] \Bigg| =O\Big( \frac{1}{n^2}\frac{1}{\big|\cS^{(0)}(e_k)-\cS^{(0)}(e_r)\big|^2}\Big).
\end{equation}
Note that this last bound could have also been derived using the Pfaffian structure of Ising fermions together with \eqref{eq:bound-correlator-distance}.
 This ensures for any $0\leq s \leq \mathbf{T}_{\Lambda_n}$ one has (again taking continuous analogs of the associated discrete integral) 
 \begin{align*}
 	\sum\limits_{e_r \in \Lambda_n} \Bigg(\frac{\partial}{\partial_{\theta_r}}\mathbb{E}_{S(s)}[\varepsilon_{e_{k}}]\Bigg)^2&=O \Bigg(\sum\limits_{e_r \in \Lambda_n}  \frac{1}{n^4}\frac{1}{\big|\cS^{(0)}(e_k)-\cS^{(0)}(e_r)\big|^4}\Bigg)\\
 	&=O\Bigg( \frac{1}{n^2}\int_{\mathcal{D}_1 \backslash \mathcal{D}_{\frac{1}{n}}} \frac{dA(z)}{|z-\frak{p}|^4} \Bigg)=O(1),\\
 	\sum\limits_{e_r \in \Lambda_n} \frac{\partial^2}{\partial_{\theta_r}^2}\mathbb{E}_{S(s)}[\varepsilon_{e_{k}}]&=\Bigg(\sum\limits_{e_r \in \Lambda_n}  \frac{1}{n^2}\frac{1}{\big|\cS^{(0)}(e_k)-\cS^{(0)}(e_r)\big|^2}\Bigg)=O(\log(n)).
 \end{align*}
Denote the event
\begin{equation}
	\mathcal{J}_{n,C}:=  \Bigg\{ \sup_{e_k\in \Lambda_n} \sup_{0\leq s \leq (\mathbf{T}_{\Lambda_n}\wedge n^{-\frac{2}{3}})} \Big| \cot(\theta^{(s)}_{k})) - \frac{\mathbb{E}_{S(s)}[\varepsilon_{e_k}]}{\sin^(\theta^{(s)}_{k}))} \Bigg|\geq C\sqrt{s\log(n)}  \Bigg\}.
\end{equation}
For each edge $e_k$ and as long as $0\leq s \leq  \mathbf{T}_{\Lambda_n}\wedge n^{-\frac{2}{3}}$, the random process $\mathbb{E}_{S(s)}[\varepsilon_{e_{k}}]-\frac{\sqrt{2}}{2}$ is the sum of a local martingale with $O(1)$ bracket and a finite variation process bounded by $O(\log(n) s) $ at time s. Since both functions $\theta_{k}\mapsto \sin^{-1}(\theta_k)$ $\theta_{k}\mapsto \cot(\theta_k)$ are smooth around $\frac{\pi}{4}$, it is not hard to see that the large deviation estimate \eqref{eq:deviation-local-martingale} ensures that there exist a large enough constant $C$ such that
 \begin{equation}
 	\mathbf{P}\Bigg[ \mathcal{J}_{n,C} \Bigg]=O(\frac{1}{n^4}).
 \end{equation}

 \textbf{Step 1: Exclude the case where at least one Brownian motions drifted too fast} One can apply verbatim the same proof as in Step 1 of Proposition \ref{prop:bound-T_n-non-optimal} and only focus on the event where all the Brownian motions attached to the edges $e\in \Lambda_n$ remain bounded by $\frac{\pi}{20}$ for all instants $0\leq t\leq \mathbf{T}_{\Lambda_n}$.
 
 \textbf{Step 2: Evaluate the order of magnitude of $\cA^{\mathbf{T}_{\Lambda_n}}(\frak{p})$} 
On the event
 \begin{equation*}
 	\Bigg\{\forall 0\leq t\leq \mathbf{T}_{\Lambda_n} , \forall e\in \Lambda_n, |B^{(e)}_{t}| \leq \frac{\pi}{20}  \Bigg\} \cap \overline{\mathcal{J}_{n,C}} \Bigg\},
 \end{equation*}
for any corner $\frak{p}\in \Lambda_n$ and any $0 \leq s\leq  \mathbf{T}_{\Lambda_n} $, one can repeat the computation of Step 2 of the proof of \ref{prop:bound-T_n-non-optimal} while replacing the universal constant used to bound $ |\cot(\theta^{(s)}_{k})) -\mathbb{E}_{S(s)}[\varepsilon_{e_k}] \cdot \sin^{-1}(\theta^{(s)}_{k}))|$ by $C\sqrt{s\log(n)} $. In that case, integrating up to time  $\mathbf{T}_{\Lambda_n} $  one gets for any $\frak{p}\in \Lambda_n$
\begin{align*}
	\cA^{\mathbf{T}_{\Lambda_n}}(\frak{p})&=O\Bigg (\int_{0}^{\mathbf{T}_{\Lambda_n}}\sqrt{s\log(n)} \sum\limits_{a_k\in \Lambda_n}\frac{1}{n^{\frac{1}{2}}} \times \frac{1}{n}\times \frac{1}{|\cS^{(0)}(a_k)-\cS^{(0)}(\frak{p}) |} ds\Bigg)\\
	&=O\Bigg(\frac{1}{n^{\frac{1}{2}}} \cdot n \cdot \log(n)^{\frac{1}{2}} \cdot \mathbf{T}_{\Lambda_n}^{\frac{3}{2}} \Bigg).
\end{align*}
Set
\begin{equation*}
	\widetilde{\mathcal{B}}_{n,c}:=\Big\{ \mathbf{T}_{\Lambda_n}\leq \frac{c}{(n\log^{\frac{1}{2}}(n)\big)^{\frac{2}{3}}} \Big\} \cap \Big\{ \forall 0\leq  t\leq \mathbf{T}_{\Lambda_n} \forall e\in \Lambda_n, |B^{(e)}_{t}| \leq \frac{\pi}{20}  \Big\}.
\end{equation*}
On the event $\widetilde{\mathcal{B}}_{n,c}$, it is clear that once again, for every corner $\frak{p}\in \Lambda_n $ one has $|\cA^{\mathbf{T}_{\Lambda_n}}(\frak{p})|=O(c\cdot n^{-\frac{1}{2}})$, where $O$ is independent from $c$ and $n$.

 \textbf{Step 3: Concluding using large deviations for $\cM^{\mathbf{T}_{\Lambda_n}}(\frak{p})$} 

One can redo verbatim Step $3$ of the proof of Proposition \ref{prop:bound-T_n-non-optimal} and deduce the existence there exist at leas one corner $\frak{p}\in \Lambda_n$ such that
\begin{equation*}
	|\cM_{\mathbb{R}}^{(\mathbf{T}_{\Lambda_n})}(\frak{p})| \geq \frac{1}{50n^{\frac{1}{2}}} \quad \textrm{ or } \quad |\cM_{i\mathbb{R}}^{(\mathbf{T}_{\Lambda_n})}(\frak{p})| \geq \frac{1}{50n^{\frac{1}{2}}}.
\end{equation*} 
Once again, for every corner $\frak{p}\in \Lambda_n$, $(\cM^{(t\wedge \mathbf{T}_{\Lambda_n})}(\frak{p}))_{t\geq 0} $ is a local martingale and one still has
\begin{equation*}
	[\cM_{\mathbb{R}}^{(t\wedge \mathbf{T}_{\Lambda_n})}(\frak{p})]^{\infty}= O(\mathbf{T}_{\Lambda_n}\frac{\log(n)}{n}),  \textrm{ and }[\cM_{i\mathbb{R}}^{(t\wedge \mathbf{T}_{\Lambda_n})}(\frak{p})]^{\infty}= O(\mathbf{T}_{\Lambda_n}\frac{\log(n)}{n}).
\end{equation*}
Using once again some union bound over corners $\frak{p}\in \Lambda_n$, one can write this time

\begin{align*}
	\mathbf{P}\Bigg[\widetilde{\mathcal{B}}_{n,c}  \cap \overline{ \mathcal{J}_{n,C}} \Bigg]&\leq O(n^2)\times \sup\limits_{\frak{p}\in \Lambda_n} \mathbf{P}\Bigg[ \big\{ \mathbf{T}_{\Lambda_n}\leq \frac{c}{(n\log^{\frac{1}{2}}(n)\big)^{\frac{2}{3}}} \big\}\bigcap \big\{ \Big|\sup_{t\geq 0}\cM_{\mathbb{R}}^{(t\wedge \mathbf{T}_{\Lambda_n} )}(\frak{p})\big| \geq \frac{1}{50n^{\frac{1}{2}}} \cdots \\
	&  \quad  \quad \quad\quad \quad \quad \quad \cdots  \bigcap \Big|[\cM_{\mathbb{R}}^{(t\wedge \mathbf{T}_{\Lambda_n})}(\frak{p})]^{\infty}\big| \leq O(\mathbf{T}_{\Lambda_n}\frac{\log(n)}{n})\Bigg]\\
&\quad + O(n^2)\times \sup\limits_{\frak{p}\in \Lambda_n} \mathbf{P}\Bigg[ \big\{ \mathbf{T}_{\Lambda_n}\leq \frac{c}{(n\log^{\frac{1}{2}}(n)\big)^{\frac{2}{3}}} \big\}\bigcap \big\{ \Big|\sup_{t\geq 0}\cM_{i\mathbb{R}}^{(t\wedge \mathbf{T}_{\Lambda_n} )}(\frak{p})\big| \geq \frac{1}{50n^{\frac{1}{2}}} \cdots \\
	&  \quad  \quad \quad\quad \quad \quad \quad \cdots  \bigcap \Big|[\cM_{i\mathbb{R}}^{(t\wedge \mathbf{T}_{\Lambda_n})}(\frak{p})]^{\infty}\big| \leq O(\mathbf{T}_{\Lambda_n}\frac{\log(n)}{n})\Bigg].
\end{align*}
It is enough to see that 
\begin{equation}
	\mathbf{P}\Bigg[\Big|\sup_{t\geq 0}\cM_{\mathbb{R}}^{(t\wedge \mathbf{T}_{\Lambda_n} )}(\frak{p})\Big| \geq \frac{1}{50n^{\frac{1}{2}}} \bigcap \Big|[\cM_{\mathbb{R}}^{(t\wedge \mathbf{T}_{\Lambda_n})}(\frak{p})]^{\infty}\Big| \leq O(\frac{c\log(n)^{\frac{1}{2}}}{n^{\frac{5}{3}}})\Bigg]
\end{equation}
decays again stretch exponentially fast in $n$ by \eqref{eq:deviation-local-martingale} as long as $c$ is chosen small enough. The same large deviation estimate for $\cM_{i\mathbb{R}}^{(t\wedge \mathbf{T}_{\Lambda_n} )}(\frak{p})$ allows to conclude the proof.
 \end{proof}
 
 \begin{proof}[Proof of Theorem \ref{thm:near-critical-RSW-random}]
 Once Proposition \ref{prop:bound-T_n-optimal} is proven it is enough to see that as long as with probability $\mathbf{P}$ at least $O(n^{-4})$, the family of $s$-embeddings $s\mapsto (\cS^{(t)}_{\omega}(\Lambda_{n}))_{0\leq t \leq t_n} $ belongs to \Unifff\,, with $t_n = c_1(n\log^{\frac{1}{2}}(n))^{-\frac{2}{3}}$ where $c_1>0$ comes from Proposition \ref{prop:bound-T_n-optimal}. It is clear that 
\begin{equation*}
	\theta^{(t)}_{e} \overset{(d)}{=} \frac{\pi}{4} + \sqrt{t}\cdot \mathcal{N}_e(0,1),
\end{equation*}
where the standard Gaussians $\mathcal{N}_e(0,1)$ are independent. This concludes the proof.
 \end{proof}

\subsection{A discussion on the optimality of the scaling window in random environment}\label{sub:optimality}
In this section, we discuss the optimality of the near-critical scaling window $O(n^{-\frac{1}{3}})$ in a random environment when only tracking down some average self-duality, which is the key tool to prove the criticality of the deterministic model. Recall that we used deformation by a Brownian motion only to simplify computations, but similar results hold (using the Skorokhod embedding Theorem, see e.g.\ \cite{AveMah25}) for a general set of random angles $(\theta_k)_{k\in \Lambda_n} $ centred around $\frac{\pi}{4}$, assuming their tail is not too degenerated. Let us start by recalling one of the main results of \cite{AveMah25}, which states that for Bernoulli percolation on the square lattice, the model in random environment remains in the critical phase as long as the random bond environment $(\mathbf{p}_e)_{e\in E(\mathbb{Z}^2)} $ are centred around the critical value. More precisely, assuming that all bond parameters are independent, satisfy $\mathbf{E}[\mathbf{p}_e]=\frac{1}{2}$ and do not degenerate toward $0$ or $1$, the strong box crossing property analog to Theorem \ref{thm:RSW-s-embeddings} holds at large scale with high $\mathbf{P}$-probability. In particular, it is not even necessary to scale the standard deviation of the random environment to $0$ as $\Lambda_n \to \mathbb{Z}^2$ to keep a critical random environment. This provides an exact classifications of critical random environments (where the edge-independence of bond percolation helps as it makes the annealed model \emph{exactly} critical). As discussed in \cite[Section 4.3]{AveMah25}, one could in principle hope for a similar treatment condition holds for generic value of FK-percolation model (when $1\leq q \leq 4$), at least in the near critical regime, as the criticality of all those models is also derived via self-duality arguments \cite{beffara2012self}. In the present context,  this would imply for the FK Ising model that there exist at least one smooth function $f$ such that the natural condition on the (near-critical) random environment to remain in the critical phase (with high $\mathbf{P}$-probability) writes as $\mathbf{E}[f(\theta_e)]=0 $. 

Let us first discuss how self-duality constrains the function $f$. First of all, the random environment shouldn't favour neither the primal nor the dual model (whose weights are given by $\theta_e^{\star}=\frac{\pi}{2} -\theta_e$). For $x\geq 0$ small enough, plugging to the condition the self-dual variable $\theta:=\frac{1}{2}\delta_{\frac{\pi}{4}+x} + \frac{1}{2}\delta_{\frac{\pi}{4}-x}$ ensures that
\begin{equation*}
	f(x+\frac{\pi}{4})=-f(\frac{\pi}{4}-x),
\end{equation*}
meaning that the function $f$ has to be odd near $\frac{\pi}{4}$. Therefore, all the even coefficients of the power series expansion of $f$ have to vanish. Conversely, it is clear that for any self-dual near-critical distribution, written as $\theta_e = \frac{\pi}{4} + X^{(e)}$ where $X^{(e)}$ is small and symmetric around $0$, \emph{any} choice of anti-symmetric function $f$ leads to $\mathbf{E}[f(\theta_e)]=0 $. Fixing the value of $f'(\frac{\pi}{4})$ (which amounts to multiply the function $f$ by some overall constant), the self-duality doesn't a priori constrain any other odd derivative of $f$ at $\frac{\pi}{4}$ starting from $f^{(3)}(\frac{\pi}{4})$.

 We claim that a for a \emph{generic} choice of function $f$ satisfying the self-duality criticality condition, one cannot expect to work with random variables $X^{(e)}=X^{(e)}_n$ whose standard deviation $\sigma_n$ satisfies $\sigma_n^3 \gg n^{-1}$ while remaining within the critical phase. Assume the opposite, and fix a generic function $f$ and a symmetric i.i.d.\ random variables $X^{(e)}_n$ whose variance is $ \sigma_n^3 \gg n^{-1}$, 
 the model remains in the critical phase. Denote such distribution by $(\mathbf{p}_e)_{e\in \Lambda_n}$ and define the random variables 
$\widehat{X}^{(e)}_n := X^{(e)}_n + \sigma_n^3$. If it not hard to see that there exist $\alpha_n \in \mathbb{R}$ (which \emph{depends} on the specific law of $\mathbf{p}^{n}_e$ but doesn't degenerate as $n\to \infty$)
such that the function $\widehat{f}(x)=f(x)-\alpha_n x^3$ is still odd and $\mathbf{E}[\widehat{f}(\frac{\pi}{4}+ \widehat{X}^{(e)}_n)]=0$. Then, by the assumption that generic odd functions around the critical point leave random environment with standard deviation $\sigma_n^3 \gg n^{-1}$ still critical, the strong box crossing property would also hold for the distribution $\frac{\pi}{4}+ \widehat{X}^{(e)}_n$.

This is not possible, as at each edge $\widehat{X}^{(e)}$ and $X^{(e)}_n$ differ by some \emph{deterministic parameter} $\sigma_n^3 \gg n^{-1}$, therefore the main results of \cite{FK_scaling_relations} ensure that both models cannot be simultaneously within the critical phase, as one could first sort the variables $X^{(e)}_n$, and then modify continuously all the weights by $\sigma_n^3$, going way beyond the near-critical scaling window, where all the relevant probabilistic properties remain the same up to multiplicative constant.

  \subsection{A weakly random interacting model with a logarithmic critical window}
In this section, we present an Ising model in a weakly random environment that exhibits a much larger critical window than the $\asymp n^{-\frac{1}{3}}$ case. This does \emph{not} contradict the discussion in Section \ref{sub:optimality}, as the model we consider here is \emph{not} built from i.i.d.\ components near the critical point. Instead, at each edge $e$, the random coupling constant is given by the sum of a scaled (with $n$) independent Gaussian variable $\mathcal{N}_e$, and an additional random drift term. With high $\mathbf{P}$-probability, this drift is much smaller in magnitude than the Gaussian term, but crucially, it depends on the values of \emph{all} other edges in the box. We adopt the notation of Section \ref{sub:non-optimal-deformation}, and let $\cY^{(0)}$ denote the fermion associated with the standard embedding of the critical square lattice. Consider a family $(B^{(e)}_t)_{e\in \Lambda_n}$ of independent standard Brownian motions under some probability measure $\mathbf{P}$. For a sufficiently small time parameter $t \geq 0$ (depending on the realization $\omega$), define the family of Ising models $\widehat{S} = (\mathbb{Z}^2, \left(\hat{x}^{(t)}_{\omega,e}\right)_{e\in E(\mathbb{Z}^2)}))_{t \geq 0}$, where all the angles outside $\Lambda_n$ are fixed at the critical value $\frac{\pi}{4}$, while for $e_k \in \Lambda_n$, the angle $\hat{\theta}^{(t)}_{e_k,\omega}$ in \eqref{eq:x=tan-theta} is defined as the strong solution to the SDE
\begin{equation}\label{eq:weakly-random-interacting}
	\hat{\theta}^{(t)}_{e_k,\omega} = \frac{\pi}{4} + B^{(e)}_{t}(\omega) - \frac{1}{2}\int_{0}^{t} \left( 
	\frac{\cos\left(\hat{\theta}^{(s)}_{e_k,\omega}\right)}{\sin\left(\hat{\theta}^{(s)}_{e_k,\omega}\right)} 
	- \frac{\mathbb{E}_{S_{\omega}(s)}[\varepsilon_{e_k}]}{\sin\left(\hat{\theta}^{(s)}_{e_k,\omega}\right)} 
	\right) ds.
\end{equation}
This model is \emph{not} i.i.d., as the term $\mathbb{E}_{S_{\omega}(s)}[\varepsilon_{e_k}]$ \emph{depends} on all the edges within $\Lambda_n$. Still, it is not hard to see that the bond parameters are almost independent in space, meaning that largest contribution to the correction drift at a given edge comes from the influence of the edges which are close in space. The choice of this particular drift ensures that it \emph{exactly cancels} the finite-variation term from Lemma \ref{lem:SDE-embedding}, creating a random $s$-embedding where each coordinate of the fermion is \emph{exactly a local martingale}. As observed, with high $\mathbf{P}$-probability, the finite variation in the Doob decomposition of solutions to \eqref{eq:SDE-embedding} dominates the evolution. Therefore, adding a drift that cancels this contribution guarantees that only the local martingale part of the process remains. This allows us to run the associated SDE for a significantly longer time while remaining within some \Unif\, class of embeddings—and hence, within the realm of (near)-critical Ising models. The following lemma formalizes this intuition.
\begin{lemma}\label{lem:SDE-embedding-drift}
 Consider the Stochastic Differential System defined on corners of $ \Upsilon^{\times} \cap  \Lambda  $, whose initial condition is given by $\cY^{(0)}$ and whose dynamic is given for any $\frak{p} \in \Lambda $ by $ \widehat{\cY}^{(s)}_{\omega}(\frak{p}):=\widehat{\cY}(\frak{p},\omega,s)$
 \begin{equation}\label{eq:SDE-embedding-drift}
 	d \widehat{\cY}^{(s)}_{\omega} (\frak{p})= \frac{1}{2} \sum\limits_{z_k \in \Lambda} \sigma_k\Bigg[ \widehat{\cY}^{(s)}_{\omega}(c^+_k) \langle \chi_{\frak{p}}\chi_{a^+_k} \rangle_{S_{\omega}(s)} -  \widehat{\cY}^{(s)}_{\omega}(a^+_k) \langle \chi_{\frak{p}}\chi_{c^+_k} \rangle_{S_{\omega}(s)} \Bigg] dB^{(e_k)}_s 
 \end{equation}

Then $\mathbf{P}$-almost surely :
\begin{enumerate}
	\item There exist $T_0(\Lambda,\omega)>0$ such that there exist a strong solution to the SDE \eqref{eq:SDE-embedding} on $[0;T_0(\Lambda,\omega)] $.
	\item For any time $0\leq s \leq T_0 (\Lambda,\omega)$, the propagator $\widehat{\cY}^{(s)}_{\omega} $ the strong solution to \eqref{eq:SDE-embedding} is an s-embedding $\widehat{\cS}^{(s)}_{\omega}$ of the Ising model $\widehat{S}_{\omega}(s)$.
\end{enumerate}
\end{lemma}
\begin{proof}
	We do not put many details in this proof as it is very similar to the proof of Lemma \ref{lem:SDE-embedding}. One can redo the approximation setting of the proof of Lemma \ref{lem:SDE-embedding} by interpolating again the Brownian motion in a piecewise linear manner on the points $s_j$. One gets this time a discrete fermion $\widehat{\cY}_{\omega,L} $ constructed from the solution of Lemma \ref{lem:ODE-embedding} of the discretised solution to \eqref{eq:weakly-random-interacting} which reads as 
\begin{align*}
	\widehat{\cY}^{(s_{j+1})}_{\omega,L}(\frak{p})-\widehat{\cY}^{(s_{j})}_{\omega,L}(\frak{p})&=\int\limits_{s_j}^{s_{j+1}} \frac{d}{ds} \widehat{\cY}^{(s)}_{\omega,L}(\frak{p}) ds\\
&= \frac{B^{(e_0)}_{s_{j+1}}(\omega)-B^{(e_0)}_{s_{j}}(\omega)}{2}  \times \Big( \cY^{(s_j)}_{\omega,L}(c_0^{+})\langle \chi_{\frak{p}}\chi_{a_0^{+}} \rangle_{S_{\omega}(s_j)} - \cY^{(s_j)}_{\omega,L}(a_0^{+})\langle \chi_{\frak{p}}\chi_{c_0^{+}} \rangle_{S_{\omega}(s_j)} \Big) \\
	&\quad +  \frac{\Big(B^{(e_0)}_{s_{j+1}}(\omega)-B^{(e_0)}_{s_{j}}(\omega)\Big)^2}{4}  \times  \Big( \frac{\cos(\theta^{(s_j)}_{0})}{\sin(\theta^{(s_j)}_{0})} -\frac{\mathbb{E}_{S_{\omega}(s_j)}[\varepsilon_{e_0}]}{\sin(\theta^{(s_j)}_{0})}  \Big)\times \cdots  \\
	& \quad \cdots \times   \Big(  \cY^{(s_j)}_{\omega,L}(c_0^{+})\langle \chi_{\frak{p}}\chi_{a_0^{+}} \rangle_{S_{\omega}(s_j)} - \cY^{(s_j)}_{\omega,L}(a_0^{+})\langle \chi_{\frak{p}}\chi_{c_0^{+}} \rangle_{S_{\omega}(s_j)}\Big) \\
	&-\frac{1}{4L}\Big( \frac{\cos(\theta^{(s_j)}_{0})}{\sin(\theta^{(s_j)}_{0})} -\frac{\mathbb{E}_{S_{\omega}(s_j)}[\varepsilon_{e_0}]}{\sin(\theta^{(s_j)}_{0})}  \Big)\Big(  \cY^{(s_j)}_{\omega,L}(c_0^{+})\langle \chi_{\frak{p}}\chi_{a_0^{+}} \rangle_{S_{\omega}(s_j)} - \cY^{(s_j)}_{\omega,L}(a_0^{+})\langle \chi_{\frak{p}}\chi_{c_0^{+}} \rangle_{S_{\omega}(s_j)}\Big)\\
	&\quad +O(L^{-\frac{3}{2}}).
\end{align*}
Passing to the limit $L\to \infty$, any subsequential limit of the discretisation process satisfies
\begin{align*} 
		d \widehat{\cY}^{(s)}_{\omega} (\frak{p})=& \frac{1}{2} \sum\limits_{z_k \in \Lambda} \Bigg[ \widehat{\cY}^{(s)}_{\omega}(c^+_k) \langle \chi_{\frak{p}}\chi_{a^+_k} \rangle_{S_{\omega}(s)} -  \widehat{\cY}^{(s)}_{\omega}(a^+_k) \langle \chi_{\frak{p}}\chi_{c^+_k} \rangle_{S_{\omega}(s)} \Bigg] dB^{(e_k)}_s \nonumber \\
		& + \frac{1}{4} \sum\limits_{z_k \in \Lambda}   \Big( \frac{\cos(\theta^{(s)}_{e_k,_{\omega}})}{\sin(\theta^{(s)}_{e_k,_{\omega}})} -\frac{\mathbb{E}_{S_{\omega}(s)}[\varepsilon_{e_k}]}{\sin(\theta^{(s)}_{e_k,_{\omega}})}  
 \Big)\times \cdots \nonumber \\
 & \cdots \times \Bigg[ \widehat{\cY}^{(s)}_{\omega}(c_k^{+})\langle \chi_{\frak{p}}\chi_{a_k^{+}} \rangle_{S_{\omega}(s)} - \widehat{\cY}^{(s)}_{\omega}(a_k^{+})\langle \chi_{\frak{p}}\chi_{c_k^{+}} \rangle_{S_{\omega}(s)} \Bigg]ds  .\\
 & - \frac{1}{4} \sum\limits_{z_k \in \Lambda}\Big( \frac{\cos(\theta^{(s)}_{e_k,_{\omega}})}{\sin(\theta^{(s)}_{e_k,_{\omega}})} -\frac{\mathbb{E}_{S_{\omega}(s)}[\varepsilon_{e_k}]}{\sin(\theta^{(s)}_{e_k,_{\omega}})}  
 \Big) \times  \cdots \\
 & \quad \cdots \times \Bigg[ \widehat{\cY}^{(s)}_{\omega}(c_k^{+})\langle \chi_{\frak{p}}\chi_{a_k^{+}} \rangle_{S_{\omega}(s)} - \widehat{\cY}^{(s)}_{\omega}(a_k^{+})\langle \chi_{\frak{p}}\chi_{c_k^{+}} \rangle_{S_{\omega}(s)} \Bigg]ds,  
 \end{align*}
which ensures that any subsequential limit of the discretised process satisfies  \eqref{eq:SDE-embedding-drift}. We leave as an exercice to check that indeed $\widehat{\cY}^{(s)}_{\omega}$ is $\mathbf{P}$-almost surely an $s$-embedding for the Ising model $\widehat{S}_{\omega}(t)$ as long as $0\leq t \leq T(\omega,\Lambda)$, exactly as in Lemma \ref{lem:SDE-embedding}.
	\end{proof}

We are now ready to generalise Proposition \ref{prop:bound-T_n-optimal} to this weakly random interacting model. This reads 
	
\begin{proposition}\label{prop:bound-T_n-interacting}
 	Assume that the coupling constants of the edges inside $\Lambda_n$ are given \eqref{eq:weakly-random-interacting}. Then there exist some small enough universal constant $c>0$ such that
 	\begin{equation}
 		\mathbf{P}\Big[ \mathbf{T}_{\Lambda_n}\leq \frac{c}{\log^2(n)} \Big] \leq O(\frac{1}{n^4}). 
 	\end{equation}
 \end{proposition}
 \begin{proof}
 	We keep the same strategy and notation as in the proof of Proposition \ref{prop:bound-T_n-optimal}. Note that Step $0$ and $2$ in the present context are trivial in the present context as the finite variation processes $(\mathcal{A}_{t}(\frak{p}))_{\frak{p}\in \Lambda_n} $ is identically $0$. Moreover, very similarly to Step 1 of the proof of \ref{prop:bound-T_n-optimal}, one can see that provided $c>0$ is chosen small enough one has 
 \begin{equation*}
 \mathbf{P}\Big[ \exists 0\leq  t\leq \frac{c}{\log^2(n)}, \exists e\in \Lambda_n, |B^{(e)}_{t}| \geq \frac{\pi}{20}  \Big]= O(\frac{1}{n^4}).	
 \end{equation*}
One can again redo verbatim Step $3$ of the proof of Proposition \ref{prop:bound-T_n-non-optimal} and prove there exist $\frak{p}\in \Lambda_n$ such that
\begin{equation*}
	|\cM_{\mathbb{R}}^{(\mathbf{T}_{\Lambda_n})}(\frak{p})| \geq \frac{1}{50n^{\frac{1}{2}}} \quad \textrm{ or } \quad |\cM_{i\mathbb{R}}^{(\mathbf{T}_{\Lambda_n})}(\frak{p})| \geq \frac{1}{50n^{\frac{1}{2}}}, 
\end{equation*} 
while for every corner $\frak{p}\in \Lambda_n$ one has 
\begin{equation*}
	[\cM_{\mathbb{R}}^{(t\wedge \mathbf{T}_{\Lambda_n})}(\frak{p})]^{\infty}= O(\mathbf{T}_{\Lambda_n}\frac{\log(n)}{n}) \textrm{ and } [\cM_{i\mathbb{R}}^{(t\wedge \mathbf{T}_{\Lambda_n})}(\frak{p})]^{\infty}= O(\mathbf{T}_{\Lambda_n}\frac{\log(n)}{n})
\end{equation*}
Set this time
\begin{equation*}
	\widehat{\mathcal{B}}_{n,c}:=\Big\{ \mathbf{T}_{\Lambda_n}\leq \frac{c}{\log^2(n)} \Big\} \cap \Big\{ \forall 0\leq  t\leq \mathbf{T}_{\Lambda_n} \forall e\in \Lambda_n, |B^{(e)}_{t}| \leq \frac{\pi}{20}  \Big\}.
\end{equation*}
Using once again some union bound over corners $\frak{p}\in \Lambda_n$, one can write this time
\begin{align*}
	\mathbf{P}\Bigg[\widehat{\mathcal{B}}_{n,c}  \cap \overline{ \mathcal{J}_{n,C}} \Bigg] & \leq  O(n^2)\sup\limits_{\frak{p}\in \Lambda_n} \mathbf{P}\Bigg[ \big\{ \mathbf{T}_{\Lambda_n}\leq \frac{c}{\log^2(n)} \big\}\cap \big\{ \big|\sup_{t\geq 0}\cM_{\mathbb{R}}^{(t\wedge \mathbf{T}_{\Lambda_n} )}(\frak{p})\big| \geq \frac{1}{50n^{\frac{1}{2}}} \big\} \cdots  \\
	&  \quad \cdots \cap \big\{  \big|[\cM_{\mathbb{R}}^{(t\wedge \mathbf{T}_{\Lambda_n})}(\frak{p})]^{\infty}\big| \leq O(\mathbf{T}_{\Lambda_n}\frac{\log(n)}{n})\big\} \Bigg]\\
 & \quad +  O(n^2)\sup\limits_{\frak{p}\in \Lambda_n} \mathbf{P}\Bigg[ \big\{ \mathbf{T}_{\Lambda_n}\leq \frac{c}{\log^2(n)} \big\}\cap \big\{ \big|\sup_{t\geq 0}\cM_{i\mathbb{R}}^{(t\wedge \mathbf{T}_{\Lambda_n} )}(\frak{p})\big| \geq \frac{1}{50n^{\frac{1}{2}}} \big\} \cdots  \\
	&  \quad \cdots \cap \big\{  \big|[\cM_{i\mathbb{R}}^{(t\wedge \mathbf{T}_{\Lambda_n})}(\frak{p})]^{\infty}\big| \leq O(\mathbf{T}_{\Lambda_n}\frac{\log(n)}{n})\big\} \Bigg].
\end{align*}
To conclude, it is enough to see (and similarly for $\cM_{i\mathbb{R}}^{(t\wedge \mathbf{T}_{\Lambda_n})}(\frak{p})$) that
\begin{equation}
	\mathbf{P}\Bigg[\Big|\sup_{t\geq 0}\cM_{\mathbb{R}}^{(t\wedge \mathbf{T}_{\Lambda_n} )}(\frak{p})\Big| \geq \frac{1}{50n^{\frac{1}{2}}} \bigcap \Big|[\cM_{\mathbb{R}}^{(\mathbf{T}_{\Lambda_n})}(\frak{p})]^{\infty}\Big| \leq O(\frac{c}{n\log(n)})\Bigg]
\end{equation}
decays at most as $O(n^{-6})$ (using again \eqref{eq:deviation-local-martingale}) provided $c$ is chosen small enough. This concludes the proof. 	
 \end{proof}
 
 We are now in position to prove the first point of Theorem \ref{thm:near-critical-RSW-random-interacting}. 
 \begin{proof}[Proof of the crossing estimate in Theorem \ref{thm:near-critical-RSW-random-interacting}]
 One can use Proposition \ref{prop:bound-T_n-interacting} to repeat the proof of Theorem \ref{thm:near-critical-RSW-random} via Proposition \ref{prop:bound-T_n-optimal} to prove the crossing estimate announced in Theorem \ref{thm:near-critical-RSW-random-interacting}.
 \end{proof}

\section{Conformal invariance of the weakly random models}\label{sec:random-fermions}

In this section, we prove that the $t_n$-weakly random Ising model and the $t_n$-weakly random interacting Ising models are conformally invariant in the limit, respectively for $t_n=n^{-(\frac{1}{3}+\alpha)}$ -with any fixed $\alpha>0$- for the former and $t_n=\log^{-C}(n)$ -with $C>0$ large enough- for the latter. As already recalled in the introduction, Smirnov's breakthrough work \cite{Smirnov-conformal} was followed by many others proving rigorously conformal invariance of the critical model. This includes the proof of convergence for FK interfaces to the chordal SLE(16/3), which encodes the more refined details on the local behaviour of the boundary separating primal and dual clusters. For the massive model, which is not conformally invariant in the limit, the convergence of the FK interfaces remains an open question, despite some excellent understanding of the correlations functions in bounded domains and in the full-plane (\cite{CIM-universality,park-iso,park2018massive}) confirming predictions of \cite{sato1979holonomic}. In the present section, we show that randomness with respect to the environment averages well enough to preserve the $\textrm{SLE}(16/3)$ limit of the FK interfaces and the conformal covariance of the second term in the energy density, almost surely with respect to the environment. In our framework, proving the convergence of FK interfaces for the models in random environment is the biggest additional challenge compared to the existing literature. The proof goes in two step. First, one should establish pre-compactness of discrete curves, which follows directly from Theorem \eqref{thm:near-critical-RSW-random} via \cite{KemSmi1,KemSmi2,aizenman1999holder}. Then, and it is the most challenging part, one should show that the FK-Ising observable converges in arbitrary rough domains to their continuous counterparts. Smirnov first achieved this task on the square lattice \cite{Smirnov-conformal}, later generalized by Chelkak and Smirnov to isoradial lattices \cite{ChSmi2}, and by Park in the massive case \cite{park-iso}. A key technique in these proofs is showing that the primitive $H_F$ of the FK observable $F$, which satisfies Dirichlet boundary conditions and is almost harmonic at the discrete level, converges to its natural continuous counterpart. In order to study FK-Ising observables in a weakly random environment, the original square lattice $ \cS^{(0)}$ is not an adapted to discrete complex analysis techniques. Therefore, we first prove that FK interfaces whose trace is \emph{drawn over the s-embedding naturally associated to the weakly random environment} indeed converge to SLE(16/3) and then use the fact that the s-embedding of the random environment remains close to the square lattice one. We use here a simplified version of one of the key inputs coming from \cite[Section 4]{Che20}, which ensures that it is enough to find a discrete differential operator associated to the $s$-embedding $\cS$ that approximates well enough the continuous Laplacian. The proof given below follows most of the steps of \cite[Section 4]{Che20} and \cite[Chapter 4]{richards2021convergence}, which we prove to remain valid on our \emph{almost square lattice}. We do not claim bringing many novelties in this section, but only prove that the stability of conformal structure of the $s$-embeddings obtained by a random deformation allows to prove the stability of the associated scaling limit.

Fix once for all some $\alpha>0$. Let $\cS^\delta$ be a proper s-embedding, covering the square $[-1;1]^2$. Fix a discrete simply $\Omega^\delta$, approximating in the Carathéodory sense $\Omega$, together with two corners $a^\delta,b^{\delta}\in \partial \Omega^\delta$, approximating respectively   two marked boundary points $a, b \in \partial \Omega$, considered as prime ends. Consider the FK-Ising model on $(\Omega_{\delta}, a^{(\delta)}, b^{(\delta)})$ with wired boundary conditions along the arc $(a^{(\delta)} b^{(\delta)})^{\circ}$ and free boundary conditions along the arc $(b^{(\delta)} a^{(\delta)})^{\bullet}$. One can define the Kadanoff-Ceva FK correlator in $(\Omega^{\delta},a^\delta,b^\delta)$ by setting for any corner $\frak{p}\in \Omega^\delta $
\begin{equation}\nonumber
	X_{\Omega^\delta}(\frak{p}):=\langle \chi_{\frak{p}}\sigma_{(a^\delta b^\delta)^\circ}\mu_{(b^\delta a^{\delta})^\bullet} \rangle_{\Omega^\delta},
\end{equation}
where the underlying Ising model is the one induced by weight on $\cS^\delta$.  In the case of the $t$-weakly random weights, given for a realisation $\omega$, consider the associated $s$-embedding $\cS^{(t)}_{\omega}$ and the associated FK-observable
\begin{equation}
	X^{(t)}_{(\Omega^{\delta_n},\omega)}(\frak{p}):=\langle \chi_{\frak{p}}\sigma_{(a^\delta b^\delta)^\circ}\mu_{(b^\delta a^{\delta})^\bullet} \rangle_{(\Omega^{\delta_n},\omega)},
\end{equation}
where the Ising model is given by the weakly random weights in $\Omega^{\delta_n} $ for the realisation $\omega$. It is clear that $X^{(0)}_{(\Omega^{\delta_n},\omega)}(\frak{p})=X^{(0)}_{\Omega^{\delta_n}}(\frak{p})$ corresponds to the standard FK observable on the critical square lattice. One can associate to $X^{(t)}_{(\Omega^{\delta_n},\omega)}$ some $s$-holomorphic function $F^{(t)}_{\delta_n,\omega}$ via \eqref{eq:X-from-F} on $\cS^{(t)}_{\omega}$. Before diving into precise statements, let us recall some specific estimate associated to the FK observable. Using the maximum principle together with its boundary jump, a purely combinatorial observation ensures (see e.g.\ \cite[Corollary 2.12]{Che20}) that one can chose the additive constant in the definition of $H_{F^{(t)}_{\delta_n,\omega}} \in [0;1]$ everywhere inside $(\Omega^{\delta_n},a^{\delta_n},b^{\delta_n})$. Moreover, assuming that we are working with a proper $s$-embedding which satisfies \Unifff\, at time $t$, the regularity theory recalled in Section \ref{sub:regularity} ensures that
\begin{equation}\label{eq:boundary-behaviour-fermion}
	\Bigg| F^{(t)}_{\delta_n,\omega}(z) \Bigg| = O\Bigg( \Big(\textrm{dist}(z,\partial \Omega^{\delta_n}_{\omega}\Big)^{-\frac{1}{2}} \Bigg).
\end{equation}

In order to run a simplified version of the methods introduced in \cite[Section 4]{Che20}, we use the fact that $s$-holomorphic functions on the embedding $\cS^{(t)}$ are \emph{Lipchitz} with high $\mathbf{P}$-probability, improving the statement made at the end of Section \ref{sub:regularity} which asserted that $s$-holomorphic functions on an $s$-embedding satisfying \Unifff\ are $\beta>0$ Holder at each scale. At time $t=0$, the embedding $\cS^{(0)}$ is the critical square lattice and it is proven in \cite[Section 3.5]{ChSmi2} (see also \cite[Proposition 4.6]{park-iso}) that $s$-holomorphic functions are Lipchitz. The following Lemma, whose proof is sketched in the Appendix, follows the spirit of many computations already made. In particular, it allows to transfer the Lipchitzness on $\cS^{(0)}$ to all embeddings $(\cS^{(t)}(\Lambda_{n}))_{0\leq t \leq t^{(\alpha)}_n}$ where $t^{(\alpha)}_n:= n^{-(\frac{2}{3}+2\alpha)}  $.
\begin{lemma}\label{lem:Lipchitzness-random-environement}
	Let $F^{(t)}$ be an $s$-holomorphic function on the $s$-embedding $\cS^{(t)}(\Lambda_{n})$ obtained via Lemma \ref{lem:SDE-embedding}. Then, with $\mathbf{P}$-probability at least $1-O(n^{-4})$, there exist a constant $C(\omega)>0$ such that for any $0\leq t \leq t^{(\alpha)}_n$, the function $F^{(t)}$ is Lipchitz at every scale, meaning that for any  $z_{1,2}\in B(u,d)$ one has 
	\begin{equation*}
		\big| F^{(t)}(z_1) -F^{(t)}(z_2) \big| \leq C(\omega) \Bigg( \frac{\big| z_1-z_2 \big| }{d} \Bigg) \max_{z\in B(u,d)} \big| F^{(t)}(z) \big|.
	\end{equation*}
\end{lemma}
The key innovation of \cite{Che20} is the understanding that constructing a good approximation of the standard continuous Laplacian on an $s$-embedding is enough to prove that discrete fermions converge to their natural continuous counterparts. When the lattice is the critical square lattice, the naive discretisation of the continuous Laplacian by the discrete one turns out to be sufficient.
Since our near-critical weakly random model is defined on an s-embedding that is \emph{almost} a critical square lattice, we work with the so-called $s$-Laplacian, adapted to the embedding $\cS$ while also approximating well the standard discrete Laplacian on $\cS^{(0)}$. We keep exactly the notations  \cite[Sections 3 and 4]{Che20}. This imples that
\begin{itemize}
	\item There exist an operator $\Delta_S$ defined in \cite[Lemma 3.7]{Che20} for a function $H $ one $\cS(\Lambda(G))$, given for $v^{\bullet} \in G^{\bullet} $ and $v^{\circ} \in G^{\circ} $ by
\begin{align*}
	\Delta_\cS [H ](v^{\bullet})&=\sum\limits_{v_k^{\bullet} \sim  v^{\bullet}} a_{ v^{\bullet} v_k^{\bullet} }\big( H(v_k^{\bullet})-H(v^{\bullet}) \big) + \sum\limits_{v_k^{\circ} \sim v^{\bullet}} b_{v^{\bullet} v_{k}^{\circ}}\big( H(v_k^{\circ})-H(v^{\bullet}) \big),\\
	 \Delta_\cS [H ](v^{\circ})&= \sum\limits_{v_k^{\circ} \sim v^{\bullet}} b_{v^{\circ} v_{k}^{\bullet}}\big( H(v_k^{\bullet})-H(v^{\circ}) \big) - \sum\limits_{v_k^{\circ} \sim  v^{\circ}} a_{ v^{\circ} v_k^{\circ} }\big( H(v_k^{\circ})-H(v_k^{\circ}) \big).
\end{align*}
where the coefficients $a_{vv'}$ are symmetric and positive, while the coefficients $b_{vv'}$ are also symmetric and real. In particular one has $a_{ v^{\bullet}_0 v_1^{\bullet} } = r_z^{-1}\sin^{2}(\theta_z) $ and $a_{ v^{\circ}_0 v_1^{\circ} } = r_z^{-1}\cos^{2}(\theta_z) $ for que quad $z=(v_{0}^{\bullet}v_{0}^{\circ}v_{1}^{\bullet}v_{1}^{\circ})$. 
 	\item Some special attention is given to $ \Delta_{\cS^{(0)}}$
 , since it corresponds on the isoradial embedding of mesh size $\delta$ (up to some explicit $\delta^{-1}$ and sign factors) to usual Laplacian $\Delta^{\textrm{sq}} $ for the standard random walk.
Following \cite{ChSmi1}), set for $v\in \Lambda^\delta(G)$ 
\begin{equation}
	\Delta^{\textrm{sq}}[H](v) := \frac{1}{4\delta^2} \sum\limits_{v_k \sim v} \Big( H(v_k)-H(v) \Big).
\end{equation}
When $\cS^{(0)}$ is seen as an isoradial lattice of mesh size $ \delta$, the coefficients $b_{vv'}$ vanish, while
\begin{equation}
	\Delta_{\cS^{(0)}} [H ](v)=\pm_{v} \delta \Delta^{\textrm{sq}}[H](v),
\end{equation}
where $\pm_{v} :=+ $ if $v\in G^{\bullet} $ and $\pm_v :=- $ if $v\in G^{\circ}$.\end{itemize}
 In particular, we are going to use the fact (proved in \cite{ChSmi1}) that $ \Delta^{\textrm{sq}}$ provides a very accurate approximation of the standard continuous Laplacian $\Delta=\partial_{xx}+\partial_{yy}$ to deduce that $\Delta_S$ also provides an very accurate (up to the sign) approximation of $\Delta$ on the associated $s$-embedding. 
 We show that the naturally defined discrete differential operators on the $s$-embeddings $(\cS^{(t)})_{0\leq t \leq t_n^{(\alpha)}} $ approximate very well those on the square lattice, which themselves approximate very well the continuous ones. Let us state that, with a slightly more refined analysis, one could in principle take $\alpha=\alpha(n)$ going to $0$ slowly enough, creating some logarithmic corrections in the analysis. Still, deforming up to \emph{any} power below $-\frac{2}{3}$ allows to lighten the writing of the proofs.
 \begin{prop}\label{prop:Laplacian-approximation-local}
	Consider an s-embedding $\cS^{(t)}(\Lambda_{n})$ constructed by Lemma \ref{lem:SDE-embedding} and a real function $\phi$. Then, with $\mathbf{P}$-probability at least $1-O(n^{-4})$, there exist a constant $O=O(\omega,\alpha)>0$ such that for any $0\leq t \leq t^{(\alpha)}_n$,
\begin{equation}\label{eq:expnasion-laplacian}
		[\Delta_{\cS^{(t)}}\phi](v)= \pm_{v} \delta \big[ \Delta\phi + O(\delta^{1+\alpha} \sup|D^{(2)}\phi| ) + O(\delta^{2+\alpha} \sup|D^{(3)}\phi |) \big],
\end{equation}
where the supremum are taken in a disc of size $10\delta$ around $v$.
\end{prop}
\begin{proof}
	Let us start with the intuition behind this statement. For $v\in G^{\bullet}$, one has
		\begin{equation}
		\Delta^{\textrm{sq}}[H](v)= \Delta\phi + O(\delta^2 \sup|D^{(4)}\phi |).
	\end{equation}
This comes from the fact that, on the standard square lattice, when expanding $\phi(v_k)$ via the Taylor expansion around of $\phi$ around $z$, the contribution of the terms of the form $\big(\cS^{(0)}(v_k)-\cS^{(0)}(v) \big)^{\ell =0,1,2,3}$ (together with associated conjugate and absolute values) exactly cancel. This can be seen in \cite[Lemma 2.2]{ChSmi1}, whose result is even better when working exactly on the square lattice rather than simply an isoradial graph. Let us discuss how to obtain $O(\delta^{1+\alpha} \sup|D^{(2)}\phi| ) $ for the almost square lattice $\cS^{(t)}$. Recall that it was proven in \cite[Proposition 3.12 ]{Che20} that $\Delta_{\cS^{(t)}}[1]=\Delta_{\cS^{(t)}}[\cS^{(t)}]=\Delta_{\cS^{(t)}} [\overline{\cS}^{(t)}]=0 $. Moreover, one can write for $v_k \sim v $ and a real function $\phi$
\begin{align*}
	\phi(v_k)&=\phi(v)+ (\cS^{(t)}(v_k)-\cS^{(t)}(v))\partial\phi(v)+ (\overline{\cS}^{(t)}(v_k)-\overline{\cS}^{(t)}(v)) \overline{\partial} \phi(v) \\
	&+C_{\phi}(\cS^{(t)}(v_k)-\cS^{(t)}(v))^2+C^{'}_{\phi}(\overline{\cS}^{(t)}(v_k)-\overline{\cS}^{(t)}(v))^2+ C^{''}_{\phi}|\cS^{(t)}(v_k)-\cS^{(t)}(v)|^2 \\
	&+ O(\delta^3\sup |D^{(3)}\phi |),
\end{align*}
where the real coefficients $C_{\phi},C^{'}_{\phi}$ and $C^{''}_{\phi}$ are bounded by $\sup |D^{(2)}\phi |$. As recalled in the proof of \cite[Proposition 3.12]{Che20} $\cS$ only satisfies some \Unif\ one has $\Delta_{\cS}[\cS^2]=\Delta_{S}[\big( \cS(\cdot) - \cS(z) \big)^2] = O(\delta)$ only using very crude bounds $O(\delta^{-1}) $ on the coefficients $a_{vv'}$ and $b_{vv'} $. If one works with $\cS^{(0)}$, then $\Delta_{\cS^{(0)}}[(\cS^{(0)})^2] =0$ by the discussion recalled above and the link to $\Delta^{\textrm{sq}}$. Using that the local geometry of $\cS^{(t)}$ is almost a square lattice up to an $1+o(\delta^{\alpha})$ multiplicative factor everywhere, one can directly see the the improvement to the crude bound $\Delta_{\cS^{ts)}}[(\cS^{(t)})^2]=O(\delta)$ is by an $O(\delta^{1+\alpha})$ multiplicative factor. This can easily be carried out carefully using e.g.\ the factorisation $\Delta_{S}=-4\overline{\partial}_\omega \partial_{\cS} $ defined in \cite[Proposition 3.7]{Che20} (here $\omega$ is kept for the sake of notations consistency with \cite{Che20} and is not related to the $\mathbf{P}$-randomness). The last terms involving the third derivatives of $\phi $ can be treated similarly, as the contribution $\big(\cS^{(0)}(v_k)-\cS^{(0)}(v) \big)^{3}$ cancel on the square lattice. The analysis is exactly the same for $v\in G^{\circ}$.
\end{proof}

We now prove that, already at discrete level, the function $H_{F^{(t)}_{\delta_n,\omega}}$ is (quantitatively) close to its continuous harmonic continuation, except maybe in a $\delta^{1-\eta}$ layer close to the boundary of a domain, following closely a simplified version \cite[Theorem 4.1]{Che20} and \cite[Chapter 4]{richards2021convergence}. Denote by $\Omega^{\delta}_{\textrm{int}(\eta)} $ the main connected component of the $\delta^{1-\eta}$ interior of $\Omega^{\delta}$.

\begin{theo}\label{thm:quantitative-convergence-h}
	Consider an s-embedding $\cS^{(t)}_{\omega}=\cS^{\delta_n}$ constructed by Lemma \ref{lem:SDE-embedding} and $(\Omega^{\delta},a^{\delta},b^{\delta},\omega,t) \subset \mathbb{C}$ a discrete bounded simply connected domain drawn over $\cS^{\delta_n}$, equipped with Dobrushin boundary conditions. For a realisation $\omega$ under the probability measure $\mathbf{P}$, set $F^{(t)}_{\delta_n,\omega}$ the $s$-holomorphic function associated to the FK-Dobrushin observable $X^{(t)}_{(\Omega^{\delta_n},\omega)}$. Denote by $h^{(t)}_{\mathrm{int}(\eta),\omega} $ the harmonic continuation of $H_{F^{(t)}_{\delta_n,\omega}}$ in $\Omega^{\delta}_{\mathrm{int}(\eta)} $. Then, with $\mathbf{P}$-probability at least $1-O(n^{-4})$, there exist an exponent $\gamma=\gamma(\alpha)>0$ and constant $O=O(\alpha,\omega)>0$ such that for any $0\leq t \leq t^{(\alpha)}_n$ and uniformly in $\Omega^{\delta}_{\mathrm{int}(\eta)} $, one has
	\begin{equation}
		|H_{F^{(t)}_{\delta_n,\omega}} -h^{(t)}_{\mathrm{int}(\eta),\omega} | = O(\delta_n^{\gamma}).
	\end{equation}
\end{theo}
We follow closely the proof of \cite{Che20} the lighten the derivation and mostly highlight the differences with our proof. Fix a non-negative symmetric function $\phi_0 \in C^{\infty}_0 (\mathbb{C}) $ which vanishes outside of $\mathbb{D}(0,\frac{1}{2})$ with $\int_\mathbb{C} \phi_0(u)dA(u)=1 $. Let $0<\varepsilon \ll \eta $ some small parameter to be later. Set 
\begin{equation}
	d_u:= \textrm{dist}(u, \partial \Omega^{\delta})  \textrm{ and } \rho_u:= \delta^{\varepsilon} \textrm{crad}(u, \partial \Omega^{\delta}) \asymp \delta^{\varepsilon}d_u \gg \delta  \textrm{ for } u\in \Omega^{\delta}_{\textrm{int}(\eta)},
	\end{equation} 
where $\textrm{crad}(u, \partial \Omega^{\delta})$ is the conformal radius of $u$ in the domain $\Omega^{\delta}$. The map $u\mapsto \textrm{crad}(u, \partial \Omega^{\delta}) $ is smooth its gradient is uniformly bounded and its second derivative is bounded by $O(d_u^{-1})$. This allows to define the running mollifier $\phi(w,u)$ by setting
\begin{equation}
	\phi(w,u):=\rho_u^{-2}\phi_0(\rho_u^{-1}(w-u))  \textrm{ for } u\in \Omega^{\delta}_{\textrm{int}(\eta)},
\end{equation}
and the mollified approximation of $H_{F^{(t)}_{\delta_n,\omega}}=H_{F^{(t)}_\omega}$ by setting
\begin{equation}
	\widetilde{H}_{F^{(t)}_\omega}(u):= \int\limits_{B(u,\rho_u)} \phi(w,u)H_{F^{(t)}_\omega}(w)dA(w),
\end{equation}
where $H_{F^{(t)}_\omega}$ is continued in a piecewise linear manner from $\Lambda(G) \cup \diamondsuit(G) $ in $\Omega^{\delta}_{\textrm{int}(\eta)}$. Below, we will prove the following key estimate
\begin{equation}\label{eq:key-estimate}
	|\Delta \widetilde{H}_{F^{(t)}_\omega}(u)| = O(\delta^{\gamma} d_u^{-2+\gamma}),
\end{equation}
which can be seen as a optimisation (that we will keep true step after step) of the estimate
\begin{equation}\label{eq:key-estimate-Che20}
	|\widetilde{H}_{F^{(t)}_\omega}(u)| = O(\delta^{p} d_u^{-2-q}) + O(\delta^{1-s} d_u^{-3}), 
\end{equation}
where $p>q(1-\eta) $ and $s<\eta$. Once this is done, one can conclude exactly as for \cite[Theorem 4.1]{Che20}

\begin{proof}[\textbf{Proof of the key estimate} \eqref{eq:key-estimate}]
We work here on the event where all the $s$-embeddings $(\cS^{(t)})_{0\leq t \leq t_n^{(\alpha)}}$ obtained via Lemma \ref{lem:SDE-embedding} satisfy \Unifff\,, the fermions encoding the embedding are close enough to the square lattice and satisfy the Lipchitzness of Lemma \ref{lem:Lipchitzness-random-environement}.
	As the mollifier $\phi(\cdot,u)$ vanishes near the boundary of $B(u,\rho_u)$ one has for $u\in \Omega^{\delta}_{\textrm{int}(\eta)}$
\begin{equation}
	\Delta \widetilde{H}_{F^{(t)}_\omega}(u)=\int\limits_{B(u,\rho_u)} \big(\Delta_u\phi(w,u)\big) H_{F^{(t)}_\omega}(w)dA(w) .
\end{equation}

\textbf{Step 1: Repeat Step 1 in \cite[Section 4.2]{Che20}}: Replace $\Delta_u\phi(w,u)$ by $\Delta_w\phi(w,u) $. Since the $s$-holomorphic functions are Lipchitz and the mollifier is symmetric  one has, exactly as in \cite[Section 4.2]{Che20}
\begin{equation}
	\Delta \widetilde{H}_{F^{(t)}_\omega}(u)=\int\limits_{B(u,\rho_u)} \big(\Delta_w\phi(w,u)\big) H_{F^{(t)}_\omega}(w)dA(w)  +O(\delta^{\varepsilon} \cdot \delta^{2\varepsilon} \cdot \rho_u^{-2} ).
	\end{equation}
	This keeps \eqref{eq:key-estimate-Che20} true with $p=\varepsilon$ and $q=0 $.
	
\medskip
 \textbf{Step 2: Replacing the integral by a discrete approximation}
	Consider for each $v\in \Omega^{\delta}_{\textrm{int}(\eta)} \cap \Lambda(G)$ the quadrilateral $\square_{v}^{(t)}$, which is 'almost' a square of width $ \sqrt{2}^{-1}\delta$ whose extremal vertices are the four quads $z_{1,2,3,4} $ neighbouring $v$. For each $w \in \square_{v} $, one gets via simple computations that
	\begin{equation}
		\Delta_w\phi(w,u)\cdot H_{F^{(t)}_\omega}(w)= \pm_{v} \big[ \Delta_{v}\phi(v,u)  + O(\frac{\delta}{\rho_u^5}) \big] \cdot \big[ H_{F^{(t)}_\omega}(v) + O(\frac{\delta}{d_u}) \big].
	\end{equation}
	Thus one has
	\begin{align*}
		\int\limits_{B(u,\rho_u)} \big(\Delta_w\phi(w,u)\big) H_{F^{(t)}_\omega}(w)dA(w)&=\sum\limits_{v \in B(u,\rho_u) }  \Delta_{v}\phi(v,u)  H_{F^{(t)}_\omega}(v)\textrm{Area}(\square_{v}^{(t)})\\
		&  \quad + O(\frac{\delta}{\rho_u^3}) + O(\frac{\delta}{d_u \times \rho_u^2}) + O(\frac{\delta^{2-3\varepsilon}}{d_u^4}),	\end{align*}
where in the error $\frac{\delta}{\rho_u^3}$ comes from integrating the error $\frac{\delta}{\rho_u^5} $ over a ball of area $\rho_u^2$ and bounding $|H_{F^{(t)}_\omega}|$ by $1$, the error $\frac{\delta}{d_u \times \rho_u^2}$ comes from the trivial bound $|\Delta_{v}\phi(v,u)|=\rho_u^{-4}$ integrated over a ball of area $\rho_u^2$ and the last error is the integrated version of the product of the errors $\frac{\delta}{\rho_u^5}$ and $\frac{\delta}{d_u}$. This keeps \eqref{eq:key-estimate-Che20} true with $s=3\varepsilon $ and $p>1-\eta $. Moreover, it is not hard to see that if $\varepsilon $ is chosen small enough compared to $\alpha$, one can replace all the terms $\textrm{Area}(\square_{v}^{(t)})$ by $\frac{1}{2}\delta^{2}$ (using that $\textrm{Area}(\square_{v}^{(t)})=\frac{1}{2}\delta^{2}+O(\delta^{2+\alpha})$) while keeping the estimate \eqref{eq:key-estimate-Che20} true with $s=3\varepsilon $ and $p>1-\eta $ .

\medskip
 \textbf{Step 3: Replace $\Delta_{v}\phi(v,u)$ by $\Delta_{\cS^{(t)}}[\phi(\cdot,u)]$ and integrate by parts}
One can now use Proposition \ref{prop:Laplacian-approximation-local} writing 
		\begin{equation}
		\Delta_{v}\phi(v,u)=\frac{\pm_v}{\delta} \Bigg[
		\Delta_{\cS^{(t)}}[\phi_u](v) + O( \frac{\delta^{1 + \alpha}}{\rho_u^4} ) + O(\frac{\delta^{2+\alpha}}{\rho_u^5}) \Bigg].
	\end{equation}
Summing over the vertices $v\in B(u,\rho_u) $ allows to include the correction terms (provided $\varepsilon$ is small enough with respect to $\alpha$), in the error $O$, compatible with \eqref{eq:key-estimate-Che20} and allows to rewrite
\begin{align*}
	\sum\limits_{v \in B(u,\rho_u) } \delta^2 \Delta_{v}\phi(v,u)  H_{F^{(t)}_\omega}(v) &=\sum\limits_{v \in B(u,\rho_u) } \delta \Delta_{\cS^{(t)}}[\phi_u](v) \Big(\pm_{v}H_{F^{(t)}_\omega}(v)\Big)\\
	& = \sum\limits_{v \in B(u,\rho_u) } \delta \Delta_{\cS^{(t)}}\Big[\pm_{v}H_{F^{(t)}_\omega}(v)\Big]\phi_u(v).  \\
\end{align*}

 \textbf{Step 4: Conclude}
 We are now in position to conclude. Recalling \cite[(4.14) in Remark 4.2]{Che20}, the fact that the $s$-holomorphic functions are Lipschitz ensures that 
 \begin{equation*}
 	H_{F^{(t)}_\omega}(v)=O(\frac{\delta^2}{d_u^3}),
 \end{equation*}
 which ensures that
 \begin{equation*}
 	\sum\limits_{v \in B(u,\rho_u) } \delta \Delta_{\cS^{(t)}}\Big[\pm_{v}H_{F^{(t)}_\omega}(v)\Big]\phi_u(v)= O(\frac{\delta}{d_u^3}).
 \end{equation*}
 This concludes the proof. Note that only some $\frac{1}{2}+ \varepsilon $ Holder exponent for $s$-holomorphic functions in Lemma \ref{lem:Lipchitzness-random-environement} would have been enough to conclude.
\end{proof}
\begin{remark}\label{rem:extension-logarithmic}
	As already mentioned, one could have replace $\alpha$ by $\alpha(n)$ going to $0$ at a logarithmic speed and still have obtained (modifying continuous statement \cite[Lemma A.2]{Che20}) some bound going to $0$ as $\delta \to 0 $ replacing the polynomial bound $ O(\delta^{\gamma})$ (which depends on $\alpha$) in the statement of Theorem \ref{thm:quantitative-convergence-h}. This would have allowed to conclude similarly for the convergence of SLE curves.
	\end{remark}

We are now in position to prove Theorem \ref{thm:random-SLE}
\begin{proof}[Proof of Theorem \ref{thm:random-SLE}]
This proof will be a simple gathering of all the facts already proven. To simplify the reading, fix some $\alpha >0 $ and write $t_n=t_n^{(\alpha)} $. Consider a realisation $\omega$ obtained under $\mathbf{P}$. Then with probability at least $1-O(n^{-4})$, on the box $\Lambda_n$, one has
\begin{itemize}
	\item The $s$-embedding $\cS_{\omega}^{(t_n)}(\Lambda_n)$ is proper and satisfies $\Unifff$ by Proposition \ref{prop:bound-T_n-optimal}.
	\item The FK-Ising model on $\cS_{\omega}^{(t_n)}(\Lambda_n)$ satisfies the strong box crossing at each scale by Theorem \ref{thm:RSW-s-embeddings} .
	\item There exist $O=O(\omega)$ and $\gamma=\gamma(\omega,\alpha)>0$ such that
\begin{equation*}
|H_{F^{(t_n)}_{\delta_n,\omega}} -h^{(t_n)}_{\mathrm{int}(\eta),\omega} | = O(n^{-\gamma}).	
\end{equation*}
\end{itemize}
One can now apply the Borel-Cantelli lemma, which ensures that there exist a set $A^{(\alpha)}$ of measure $1$ for $\mathbf{P}$ such that for any $\omega \in A^{(\alpha)}$, there exist $n_0(\omega)$ such that the three properties listed above hold for any $n\geq n_0(\omega)$. Fix a simply connected domain $(\Omega,a,b) \subset [-\frac{1}{2};\frac{1}{2}]^2$ with two marked boundary points $a,b\in \partial \Omega$, seen as prime-ends, approximated in the Carathéodory sense by $(\Omega_{\delta_n}, a^{(\delta_n)}, b^{(\delta_n)})\subset \cS^{(t_n)}$ with the associated Dobrushin boundary conditions. For any $\omega \in A^{(\alpha)} $ $n\geq n_0$, the first two properties ensures that the family of curves $\big(\widetilde{\gamma}^{(\delta_n)}_{\omega,t_n}\big)_{n\geq 1} $ \emph{drawn on the embedding } $\cS^{(t_n)}(\Lambda_n)$ is pre-compact. Moreover, one can reproduce verbatim the proof of \cite[Section 4.3]{Che20}, which proves that, for any $\omega \in A^{(\alpha)}$, the family of complexified FK-observables $(F^{(t_n)}_{\delta_n,\omega})_{n\geq 1}$ converge to a function $f_\Omega$ as $n\to \infty$, uniformly on compacts of $\Omega$. The function is given by $f_\Omega(z)= \sqrt{\Phi'(z)}$, where $\Phi : \Omega \to \mathbb{R}\times ]0,1[ $ is the uniformisation of $\Omega$ to the strip that sends $a,b$ to respectively $\pm \infty$.	 It is a classical fact (see \cite{karrila2018limits,KemSmi1,KemSmi2}) that the pre-compactness of the family of curves $\big(\widetilde{\gamma}^{(\delta_n)}_{\omega,t_n}\big)_{n\geq 1} $ already discussed together with the convergence of the associated FK observables ensure that the curves $\big(\widetilde{\gamma}^{(\delta_n)}_{\omega,t_n}\big)_{n\geq 1} $ converge to the chordal $\textrm{SLE}(16/3,\Omega,a,b)$. Moreover, for any $\omega \in A^{(\alpha)}$, when choosing a commun reference point to $\cS^{(0)}(\Lambda_n) \subset \delta_n \mathbb{Z}^2$ and $\cS^{(t_n)}(\Lambda_n)$, the distances are all multiplied by some $1+o_{n\to \infty}(1)$ factor (where $o_{n\to \infty}(1)$ depends on $\omega$). There, one can trivially couple the curves $\widetilde{\gamma}^{(\delta_n)}_{\omega,t_n} \in \cS^{(t_n)}(\Lambda_n)$ and $\gamma^{(\delta_n)}_{\omega,t_n} \in \delta_n \mathbb{Z}^2$ to have exactly the same exploration process. Using this coupling between the curves, the curves $\big(\widetilde{\gamma}^{(\delta_n)}_{\omega,t_n}\big)_{n\geq 1} $ and $\gamma^{(\delta_n)}_{\omega,t_n} \in \delta_n \mathbb{Z}^2$ are drawn at an $o_{n\to \infty}(1)$ distance from each other as $n\to \infty$. Recalling the topology for the convergence of curves \cite{papon2024interface} and using the convergence of the interfaces  $\big(\widetilde{\gamma}^{(\delta_n)}_{\omega,t_n}\big)_{n\geq 1} $  allows to conclude for $\big(\gamma^{(\delta_n)}_{\omega,t_n}\big)_{n\geq 1} $.
\end{proof} 

\begin{proof}[Proof of Theorem \ref{thm:random-energy-density}]
	Once we know that with high probability $\cS^{(t_n^{(\alpha)}}$ and $\cS^{(0)}$ are very close to each other, one can apply verbatim \cite{MahPar25a}. Note that an easy computation coming from the study of the stochastic derivatives (see e.g.\ Step $0$ of the proof of Proposition \ref{prop:bound-T_n-optimal}) that for $\mathbb{E}_{S(t_n^{(\alpha)})}[\varepsilon_{e_k}]=\mathbb{E}_{S(0)}[\varepsilon_{e_k}] + \asymp \sqrt{t_n^{(\alpha)}}$, where $\asymp$ is taken up to logarithmic correction. In particular the deviation of is much larger than the $n$ scaling factor coming from the fermion when $\alpha$ is small enough.
\end{proof}

\begin{proof}[Proof of the second part of Theorem \ref{thm:near-critical-RSW-random-interacting}]
We do not provide a proof here, only a sketch. Assume that $t_n:=\log(n)^{-\beta} $. As mentioned in Remark \ref{rem:extension-logarithmic}, we claim that if $\beta$ is chosen large enough, one can redo the proof of the Laplacian estimate of Proposition \ref{prop:Laplacian-approximation-local} and Theorem \ref{thm:quantitative-convergence-h}, with an exponent $\alpha=\alpha(n)\to 0$ producing a correction decaying logarithmically fast to $0$ (with a large enough exponent). Modifying the proof of \cite[Lemma A.2]{Che20}), one can replace the RHS $O(\delta^{\gamma}_n)$ in the statement of Theorem \ref{thm:quantitative-convergence-h} by some quantity that goes to $0$ as $\delta_n\to 0$. This allows to apply verbatim the proof of chordal SLE converges presented above. Finally, the convergence of the energy density is once again a direct application of \cite{MahPar25a}.	
\end{proof}

\appendix

\section{Appendix}

We first briefly discuss how to prove Lemma \ref{lem:Lipchitzness-random-environement}, which states that one provided that the $s$-embeddings are obtained from the SDE of Lemma \ref{lem:SDE-embedding}, as long as all coordinates of the embedding $\cS^{(s)}$ are close to those of $\cS^{(0)}$, then one can keep the Lipchitzness of $s$-holomorphic functions the embedding $\cS^{(s)}$ similar to the one on $\cS^{(0)}$ which corresponds to the critical square lattice where the result is known. 

\begin{proof}[Sketch of the proof of Lemma \ref{lem:Lipchitzness-random-environement}]
Let $F^{(s)}$ an $s$-holomorphic function on $\cS^{(z)}$.  Recall from \cite[Section 2.5]{Che20} that $\Re[F^{(s)}]$ is \emph{harmonic} for the backward random walk on the $s$-graph $\cS^{(s)}-i\cQ^{(s)}$, for the weights recalled in \cite[Proposition 2.16 and (2.20)]{Che20}. At time $s=0$, this random walk can be seen as the standard random walk on vertices of the square lattice with $\frac{1}{4}$ transitions to the neighbours (see e.g.\ \cite[Section 3.2.2]{CHI}). The underlying Laplacian is denoted by $\Delta^{(0)}$, with no additional $\delta^2$ scaling compared to the one used Proposition \ref{prop:Laplacian-approximation-local}.
We claim that first that provided $\cS^{(s)}$ remains close enough to $\cS^{(0)}$, this random remains can still be identified random walk on the square lattice, with weights (depending on the local geometry) that are almost $\frac{1}{4}$ everywhere. We follow the formalism of the statement \cite[Proposition 38]{park2018massive}. Fix two neighbouring quads $z,z'$ (at a distance $O(\delta)$ from each other and $z''\in \partial B(z,R)$ for $R\geq \text{cst}\cdot \delta$. Denote $\textrm{hm}^{(s),z''}_{B(z,R)}(z)$ the harmonic measure of the point $z''$ seen from $z$ for the walk on  $\cS^{(s)}-i\cQ^{(s)}$, that is the unique harmonic function inside $B(z,R)$ whose boundary value along $\partial B(z,R) $ is $0$ except at $z''$ where it is one. For the critical square lattice corresponding to $\cS^{(0)}$, we it is proven in \cite[Proposition 2.7]{ChSmi1} that $|\textrm{hm}^{(s),z''}_{B(z,R)}(z)-\textrm{hm}^{(s),z''}_{B(z,R)}(z')|=O(\delta^2R^{-2})$. The transition weights of $\textrm{hm}^{(s),z''}_{B(z,R)}$ on the square lattice $\cS^{(0)}$ vary continuously with time, with an explicit dependence on the geometry of the embedding recalled in \cite[Equation (2.20)]{Che20}. Computing stochastic derivatives in time of $\textrm{hm}^{(s),z''}_{B(z,R)}(z)-\textrm{hm}^{(s),z''}_{B(z,R)}(z')$, applying the Gronwall lemma (with high probability with respect to the random environment) similarly to \eqref{eq:Grownwall}, it not hard to see that $\sup_{z\sim z' \in B(z,R/2) }|\textrm{hm}^{(s),z''}_{B(z,R)}(z)-\textrm{hm}^{(s),z''}_{B(z,R)}(z')|$ doesn't get multiplied by more than an absolute multiplicative constant (compared to its original value at time $s=0$) for $0\leq s \leq n^{-\frac{2}{3}+\alpha}$, uniformly in $z\sim z' $ in $B(z,R/2)$. Note that to run the proof of Section \ref{sec:random-fermions}, one only needs a weaker statement and prove that harmonic functions on S-graphs are $(\frac{1}{2} +\varepsilon)$-Holder.
\end{proof}

\subsection{Derivatives of correlations and Kadanoff-Ceva fermions}\label{app:derivative-fermions}
We regroup in this section the computations appearing when differentiating Kadanoff-Ceva fermions with respect to the associated coupling constants. Recall that, given two corners $\frak{p},\frak{q}\in \Upsilon^{\times}$, the two point fermions can be written as  $\langle \chi_\frak{p} \chi_\frak{q} \rangle_S=\mathbb{E}[ \prod\limits_{e\in \gamma} (x_e)^{\varepsilon_e} \sigma_{u^\circ(\frak{p})} \sigma_{u^\circ(\frak{q})}]$ where $\gamma $ is a disorder path linking $v^\bullet(\frak{p}) $ to $v^\bullet(\frak{q})$. This formula allows to compute the derivatives of $\langle \chi_\frak{p} \chi_\frak{q} \rangle_S$ with respect to the Ising weight $x_{e_r}=\tan \frac{\theta_{e_r}}{2}$. The obtained formula in fact depends on whether both corners $\frak{p},\frak{q} $ belong to the quad attached to the edge $x_{e_r}$ or not. It is fairly known statement (see \cite{mccoy-wu-book}) that given a finite set $A$ and setting $\sigma_A:=\prod_{u\in A} \sigma_u$, one has
\begin{align}\label{eq:derivative-spin-correlation}
	\frac{\partial}{\partial x_{e_r}} \mathbb{E}[\sigma_{A}]  &= \frac{1}{2x_{e_r}} \big(\mathbb{E}[\sigma_{A}]\mathbb{E}[\varepsilon_{e_r}] -\mathbb{E}[\sigma_{A}\varepsilon_{e_r}]\big).
\end{align}
Similarly to the construction presented in Section \ref{sub:2-point-fermion}, one can use the identity $x_{e_k}^{\varepsilon_{e_k}} =\frac{1}{2} \big( (x_{e_k}+x_{e_k}^{-1})\varepsilon_{e_k} + (x_{e_k}-x_{e_k}^{-1})  \big)$ to expand the two point fermion as in \eqref{eq:expansion-fermion} to rewrite it as a linear combination of spins correlation. When differentiating with respect to the coupling constant $x_{e_r}$, some dichotomy appears, depending on the possibility to connect $\frak{p}$ to $\frak{q}$ using a disorder path that avoids the edge $e_r$. More precisely, one gets:
\begin{itemize}
	\item If the two corners $\frak{p},\frak{q}$ don't belong \emph{simultaneously} to the quad $z_r$ attached to the edge $e_r$, one can expand $\langle \chi_\frak{p} \chi_\frak{q} \rangle_S$ using a disorder path linking $\frak{p},\frak{q}$ and avoiding the edge $e_r$. Therefore, one can differentiate all the obtained Ising correlations with respect to $x_{e_r}$ using \eqref{eq:derivative-spin-correlation} and then sum them back together, which implies directly that 
\begin{equation}
		\frac{\partial}{\partial x_{e_r}} \langle \chi_\frak{p} \chi_\frak{q} \rangle = \frac{1}{2x_{e_r}} \big( \langle \chi_\frak{p} \chi_\frak{q} \rangle \mathbb{E}[\varepsilon_{e_r}]-  \langle \chi_\frak{p} \chi_\frak{q}\varepsilon_{e_r} \rangle  \big),
\end{equation}
as all the coefficients in the correlation expansion of the products of the form $x_{e_k}^{\varepsilon_{e_k}} =\frac{1}{2} \big( (x_{e_k}+x_{e_k}^{-1})\varepsilon_{e_k} + (x_{e_k}-x_{e_k}^{-1})  \big)$ present in the disorder line from $v^{\bullet}(\frak{p})$ to $v^{\bullet}(\frak{q})$ do \emph{not} depend on the coupling constant $x_{e_r}$ (and therefore one only needs to differentiate the Ising correlations). This allows to rewrite
\begin{align}\label{eq:derivative-fermion-away}
	\frac{\partial}{\partial \theta_{e_r}} \langle \chi_\frak{p} \chi_\frak{q} \rangle &=  \frac{1}{4}(x_{e_r}+x_{e_r}^{-1})\cdot \Big( \langle \chi_\frak{p} \chi_\frak{q} \rangle \mathbb{E}[\varepsilon_{e_r}]-  \langle \chi_\frak{p} \chi_\frak{q}\varepsilon_{e_r} \rangle \Big) \\
	& = \frac{1}{2\sin(\theta_{e_r})}\big( \langle \chi_\frak{p} \chi_\frak{q} \rangle \mathbb{E}[\varepsilon_{e_r}]-  \langle \chi_\frak{p} \chi_\frak{q}\varepsilon_{e_r} \rangle  \big)
\end{align}

	\item There is a slight twist in the previous reasoning in the case where both corners $\frak{p},\frak{q}$ belong to the edge $e_r$, as one can take the simplest disorder being either empty or only containing the edge $e_r$. When differentiating, the coefficients in $(x_{e_r})^{\pm1}$ now also depend on $e_r$. We prefer to use here another approach, with a disorder path from $\frak{p}$ to $\frak{q}$ that avoids the edge $e_r$. The tradeoff of this approach is that the parity of the number of crossings (between primal and dual paths) now changes when using this new path avoids $e_r$, meaning that a path with the same parity of crossing will connect $\frak{p}$ to $\frak{q}^{\star}$, the other corner of $\Upsilon^{\times}$ having the same planar projection as $\frak{q}$. Using the spinor identity of Kadanoff-Ceva fermions leads to 
 \begin{equation}\label{eq:derivative-fermion-close}
		\frac{\partial}{\partial \theta_{e_r}} \langle \chi_\frak{p} \chi_\frak{q} \rangle = -\frac{\partial}{\partial \theta_{e_r}} \langle \chi_\frak{p} \chi_{\frak{q}^\star} \rangle = - \frac{1}{2\sin(\theta_{e_r})}\big( \langle \chi_\frak{p} \chi_\frak{q} \rangle \mathbb{E}[\varepsilon_{e_r}]-  \langle \chi_\frak{p} \chi_\frak{q}\varepsilon_{e_r} \rangle  \big).
\end{equation}
\end{itemize}
As a special case of correlation functions, one can enhance \eqref{eq:derivative-spin-correlation} and deduce that for any edge $e_k$ one has
\begin{align*}
	\frac{\partial}{\partial_{\theta_r}}\mathbb{E}[\varepsilon_{e_{k}}]&=\frac{1}{2\sin(\theta_{e_r})}\Bigg( \mathbb{E}[\varepsilon_{e_{k}}]\mathbb{E}_{S(s)}[\varepsilon_{e_{r}}]-\mathbb{E}[\varepsilon_{e_{k}}\varepsilon_{e_{r}}] \Bigg)\\
	\frac{\partial^2}{\partial_{\theta_r}^2}\mathbb{E}[\varepsilon_{e_{k}}]&=-\frac{\cos(\theta_{e_r})}{2\sin^2(\theta_{e_r})}\Bigg( \mathbb{E}[\varepsilon_{e_{k}}]\mathbb{E}[\varepsilon_{e_{r}}]-\mathbb{E}[\varepsilon_{e_{k}}\varepsilon_{e_{r}}] \Bigg)\\
	&\quad +\frac{1}{2\sin(\theta_{e_r})}\Bigg[ \frac{1}{2\sin(\theta_{e_r})} \cdot  \Bigg( \mathbb{E}[\varepsilon_{e_{k}}]\mathbb{E}[\varepsilon_{e_{r}}]-\mathbb{E}[\varepsilon_{e_{k}}\varepsilon_{e_{r}}] \Bigg)\cdot \mathbb{E}[\varepsilon_{e_{r}}]   \\
	& \quad + \frac{1}{2\sin(\theta_{e_r})} \mathbb{E}[\varepsilon_{e_{k}}] \cdot \Bigg( \mathbb{E}[\varepsilon_{e_{r}}]^2-1 \Bigg) \\
	& \quad - \frac{1}{2\sin(\theta_{e_r})} \Bigg( \mathbb{E}[\varepsilon_{e_{k}}\varepsilon_{e_{r}}] \mathbb{E}[\varepsilon_{e_{r}}] - \mathbb{E}[\varepsilon_{e_{k}}] \Bigg)\\
	&=\Bigg( \mathbb{E}[\varepsilon_{e_{k}}]\mathbb{E}[\varepsilon_{e_{r}}]-\mathbb{E}[\varepsilon_{e_{k}}\varepsilon_{e_{r}}] \Bigg) \cdot \Bigg[  -\frac{\cos(\theta_{e_r})}{2\sin^2(\theta_{e_r})}+ \frac{1}{4\sin^2(\theta_{e_r})}\mathbb{E}[\varepsilon_{e_{r}}]\\
	&\quad + \frac{1}{2\sin(\theta_{e_r})} \mathbb{E}[\varepsilon_{e_{r}}] \Bigg].
\end{align*}

Let us now compute second order derivatives of Kadanoff-Ceva correlator, only focusing on differentiating twice with respect to the angle $\theta_{e_r}$. Assume first that $\frak{p}$ and $\frak{q}$ don't belong to the same quad $z_r$. In that case, one has. Let us differentiate \eqref{eq:derivative-fermion-away} with respect to $\theta_{e_r}$. One has
\begin{align*}
\frac{\partial^2}{\partial \theta^2_{e_r}} \langle \chi_\frak{p} \chi_\frak{q} \rangle	& = -\frac{-\cos(\theta_{e_r})}{2\sin^2(\theta_{e_r})}\Big( \langle \chi_\frak{p} \chi_\frak{q} \rangle \mathbb{E}[\varepsilon_{e_r}]-  \langle \chi_\frak{p} \chi_\frak{q}\varepsilon_{e_r} \rangle  \Big) \\
&\quad + \frac{1}{2\sin(\theta_{e_r})}\Bigg[ \frac{1}{2\sin(\theta_{e_r})}\Big( \langle \chi_\frak{p} \chi_\frak{q} \rangle \mathbb{E}[\varepsilon_{e_r}]-  \langle \chi_\frak{p} \chi_\frak{q}\varepsilon_{e_r} \rangle  \Big)\cdot \mathbb{E}[\varepsilon_{e_r}]\\
& \quad + \langle \chi_\frak{p} \chi_\frak{q} \rangle \cdot \frac{1}{2\sin(\theta_{e_r})}\Big( \mathbb{E}[\varepsilon_{e_{r}}]^2 -1\Big)\\
& \quad - \frac{1}{2\sin(\theta_{e_r})} \cdot \Big( \langle \chi_\frak{p} \chi_\frak{q}\varepsilon_{e_r} \rangle \mathbb{E}[\varepsilon_{e_{r}}]-\langle \chi_\frak{p} \chi_\frak{q} \rangle\Big)\\
&= \frac{1}{2\sin(\theta_{e_r})}\Bigg[ \langle \chi_\frak{p} \chi_\frak{q} \rangle \mathbb{E}[\varepsilon_{e_r}]-  \langle \chi_\frak{p} \chi_\frak{q}\varepsilon_{e_r} \rangle  \Bigg]\cdot \Bigg[ \frac{-\cos(\theta_r)}{\sin(\theta_r)}  +  \frac{\mathbb{E}[\varepsilon_{e_{r}}]}{\sin(\theta_r)}  \Bigg],
\end{align*} 
using that $\varepsilon_{e_r}^2 =1$. Therefore one has
\begin{equation}\nonumber
	\frac{\partial^2}{\partial \theta^2_{e_r}} \langle \chi_\frak{p} \chi_\frak{q} \rangle=  \frac{1}{2\sin(\theta_{e_r})}\Bigg[ \langle \chi_\frak{p} \chi_\frak{q} \rangle \mathbb{E}[\varepsilon_{e_r}]-  \langle \chi_\frak{p} \chi_\frak{q}\varepsilon_{e_r} \rangle  \Bigg]\cdot \Bigg[ \frac{-\cos(\theta_r)}{\sin(\theta_r)}  +  \frac{\mathbb{E}[\varepsilon_{e_{r}}]}{\sin(\theta_r)}  \Bigg].
\end{equation}
When $\frak{p},\frak{q}\in z_r$ one can use again the trick of changing the disorder line avoiding the edge $e_r$ up to a sign flip which ensures that 
\begin{equation}\nonumber
	\frac{\partial^2}{\partial \theta^2_{e_r}} \langle \chi_\frak{p} \chi_\frak{q} \rangle= - \frac{1}{2\sin(\theta_{e_r})}\Bigg[ \langle \chi_\frak{p} \chi_\frak{q} \rangle \mathbb{E}[\varepsilon_{e_r}]-  \langle \chi_\frak{p} \chi_\frak{q}\varepsilon_{e_r} \rangle  \Bigg]\cdot  \Bigg[ \frac{-\cos(\theta_r)}{\sin(\theta_r)}  +  \frac{\mathbb{E}[\varepsilon_{e_{r}}]}{\sin(\theta_r)}  \Bigg].
\end{equation}

We now use those first and second order derivatives to prove Lemma \ref{lem:Lipchitzness-random-environement}, that states that as long as the deformation of the $s$-embedding doesn't degenerate too much from the critical square lattice, all $s$-holomorphic function remain Lipschitz.

\subsection{Second order derivatives of embedding fermions}\label{app:derivative-embedding}
The goal of the second part of the Appendix is to present some formal computations of second order derivatives (differentiating twice with respect to the same coupling constant) when running the embedding ODE \eqref{eq:ODE-embedding} with constant masses. In particular, we remain at formal level and only focus on how local relations factor out. Fix a set of masses $(m_k)_{k\in \Lambda_n}$ and smooth function $\cY=\cY(\theta_1,\ldots,\theta_{|\Lambda_n|}):E(\Lambda_n) \to (\mathbb{C})^{\Upsilon^\times \cap \Lambda_n}$ such that 
\begin{equation}\label{eq:first-derivative-Y}
		\frac{d \cY(\frak{p})}{d\theta_k}(\theta_1,\ldots,\theta_{|\Lambda_n|})=\frac{m_k}{2} \Big( \cY(c^+_k) \langle \chi_{\frak{p}}\chi_{a^+_k} \rangle_{S} -  \cY(a^+_k) \langle \chi_{\frak{p}}\chi_{c^+_k} \rangle_{S} \Big),
\end{equation}
where $S=S(\theta_1,\ldots,\theta_{|\Lambda_n|})$ is the full-plane Ising model with critical coupling constant $\frac{\pi}{4}$ outside of $\Lambda_n$ and coupling constants given by $\theta_1,\ldots,\theta_{|\Lambda_n|}$ inside $\Lambda_n$. To lighten notations, we skip the $S$ subscript for the Kadanoff-Ceva correlators written below, but they are all implicitly evaluated for the Ising models with weights induced by $S$.
We are going to prove that for any $\frak{p}\in \Lambda_n $ one has
\begin{equation}\label{eq:second-derivative-embedding-formula}
	\frac{d^2 \cY(\frak{p})}{d^2\theta_k}= \frac{m_k^2}{2} \Bigg( \frac{\cos(\theta_{e_k})}{\sin(\theta_{e_k})} -\frac{\mathbb{E}[\varepsilon_{e_k}]}{\sin(\theta_{e_k})}  
 \Bigg)  \Bigg[ \cY(c_k^{+})\langle \chi_{\frak{p}}\chi_{a_k^{+}} \rangle - \cY(a_k^{+})\langle \chi_{\frak{p}}\chi_{c_k^{+}} \rangle  \Bigg].
\end{equation}
Let us compute $\frac{d^2 \cY(\frak{p})}{d^2\theta_k}$, by differentiating \eqref{eq:first-derivative-Y} with respect to $\theta_k$. One has
\begin{align}\label{eq:second-derivative-Y}
	\frac{d^2 \cY(\frak{p})}{d^2\theta_k}&= \frac{m_k}{2}\Bigg[ \frac{d\cY(c_k^{+})}{d\theta_k} \langle \chi_{\frak{p}}\chi_{a_k^{+}} \rangle + \cY(c_k^{+}) \frac{d}{d\theta_k} \langle \chi_{\frak{p}}\chi_{a_k^{+}} \rangle \\
	& \quad - \frac{d\cY(a_k^{+})}{d\theta_k} \langle \chi_{\frak{p}}\chi_{c_k^{+}} \rangle - \cY(a_k^{+}) \frac{d}{d\theta_k} \langle \chi_{\frak{p}}\chi_{c_k^{+}} \rangle\Bigg]\nonumber
\end{align}
Recall that derivative formulae for Kadanoff-Ceva fermions depend a priori on whether the two associated corners belong or not the same quad. One should therefore treat separately two cases, depending on whether $\frak{p} \not \in  z_k $ or $\frak{p} \in  z_k$.

\textbf{Case A: $\frak{p} \not \in  z_k $}. In that case one can plug \eqref{eq:derivative-fermion-away} into \eqref{eq:second-derivative-Y} and write
\begin{align*}
	\frac{d^2 \cY(\frak{p})}{d^2\theta_k}&= \frac{m_k}{2}\Bigg[ \frac{m_k}{2}\Big( \cY(c_k^{+}) \langle \chi_{c_k^{+}}\chi_{a_k^{+}} \rangle - \cY(a_k^{+}) \Big) \langle \chi_{\frak{p}}\chi_{a_k^{+}} \rangle \\
	& \quad + \frac{m_k}{2\sin(\theta_{e_k})}\cY(c_k^{+}) \Big( \langle \chi_{\frak{p}}\chi_{a_k^{+}} \rangle \mathbb{E}[\varepsilon_{e_k}] - \langle \chi_{\frak{p}}\chi_{a_k^{+}} \varepsilon_{e_k} \rangle  \Big) \Bigg]\\
	& \quad - \frac{m_k}{2}\Bigg[ +\frac{m_k}{2}\Big( \cY(c_k^{+})- \cY(a_k^{+}) \langle \chi_{a_k^{+}}\chi_{c_k^{+}} \rangle \Big) \langle \chi_{\frak{p}}\chi_{c_k^{+}} \rangle \\
	& \quad +\frac{m_k}{2\sin(\theta_{e_k})}\cY(a_k^{+}) \Big( \langle \chi_{\frak{p}}\chi_{c_k^{+}} \rangle \mathbb{E}[\varepsilon_{e_k}] - \langle \chi_{\frak{p}}\chi_{c_k^{+}} \varepsilon_{e_k} \rangle  \Big) \Bigg].
\end{align*}
One can then regroup together the $\cY(c_k^{+})$ and $\cY(a_k^{+})$ terms and rewrite $\frac{d^2 \cY(\frak{p})}{d^2\theta_k}$ as
\begin{align*}
	\frac{m_k^2}{4}\cY(c_k^{+})\Bigg[ \langle \chi_{c_k^{+}}\chi_{a_k^{+}} \rangle  \langle \chi_{\frak{p}}\chi_{a_k^{+}} \rangle + \frac{1}{\sin(\theta_{e_k})}\Big( \langle \chi_{\frak{p}}\chi_{a_k^{+}} \rangle \mathbb{E}[\varepsilon_{e_k}] - \langle \chi_{\frak{p}}\chi_{a_k^{+}} \varepsilon_{e_k} \rangle  \Big) - \langle \chi_{\frak{p}}\chi_{c_k^{+}} \rangle \Bigg] \\
+\frac{m_k^2}{4}\cY(a_k^{+})\Bigg[ -\langle \chi_{\frak{p}}\chi_{a_k^{+}}\rangle+   \langle \chi_{a_k^{+}}\chi_{c_k^{+}} \rangle  \langle \chi_{\frak{p}}\chi_{c_k^{+}} \rangle   -  \frac{1}{\sin(\theta_{e_k})} \Big( \langle \chi_{\frak{p}}\chi_{c_k^{+}} \rangle \mathbb{E}[\varepsilon_{e_k}] - \langle \chi_{\frak{p}}\chi_{c_k^{+}} \varepsilon_{e_k} \rangle  \Big)  \Bigg].
\end{align*}
This is where some additional simplifications appear. First note that $\varepsilon_{e_k}=\chi_{d_k^{+}}\chi_{a_k^{+}}$. Moreover, when can look at the 3 terms identity \eqref{eq:3-terms} for the fermion $\frak{q}\mapsto \langle \chi_{\frak{p}}\chi_{\frak{q}}\rangle$ around the quad $z_k$ when $ \frak{p}\not \in z_k$ is fixed. More specifically an analog of $(\widetilde{\mathrm{D}}_k^{(\frak{p})})$ when moving the second argument around $z_k$ ensures that
\begin{equation}
	-\frac{1}{\sin(\theta_{e_k})}\langle \chi_{\frak{p}}\chi_{a_k^{+}} \varepsilon_{e_k} \rangle  - \langle \chi_{\frak{p}}\chi_{c_k^{+}}\rangle = \frac{\cos(\theta_{e_k})}{\sin(\theta_{e_k})}\langle \chi_{\frak{p}}\chi_{a_k^{+}}\rangle.
\end{equation}
Similarly writing $\varepsilon_{e_k}=\chi_{b_k^{+}}\chi_{c_k^{+}}$ together with an analog of relation $(\widetilde{\mathrm{B}}_k^{(\frak{p})})$ when moving the second fermion around $z_{k}$ ensures that
\begin{equation}
	\frac{1}{\sin(\theta_{e_k})}\langle \chi_{\frak{p}}\chi_{c_k^{+}} \varepsilon_{e_k} \rangle  - \langle \chi_{\frak{p}}\chi_{a_k^{+}}\rangle = -\frac{\cos(\theta_{e_k})}{\sin(\theta_{e_k})}\langle \chi_{\frak{p}}\chi_{c_k^{+}}\rangle.
\end{equation}
which allows to rewrite $\frac{d^2 \cY(\frak{p})}{d^2\theta_k}$ as
\begin{align*}
	\frac{m_k^2}{4}\cY(c_k^{+})\langle \chi_{\frak{p}}\chi_{a_k^{+}} \rangle \Big[ \langle \chi_{c_k^{+}}\chi_{a_k^{+}} \rangle  - \frac{\mathbb{E}[\varepsilon_{e_k}]}{\sin(\theta_{e_k})} + \frac{\cos(\theta_{e_k})}{\sin(\theta_{e_k})} \Big] \\
	+\frac{m_k^2}{4}\cY(a_k^{+})\langle \chi_{\frak{p}}\chi_{c_k^{+}} \rangle\Big[ \langle \chi_{a_k^{+}}\chi_{c_k^{+}} \rangle + \frac{\mathbb{E}[\varepsilon_{e_k}]}{\sin(\theta_{e_k})}      - \frac{\cos(\theta_{e_k})}{\sin(\theta_{e_k})}  \Big]
\end{align*}
One can conclude using the three term identity around the quad $z_k$ (which holds here on $\Upsilon^{\times}_{(a_k)}$) and deduce that
\begin{equation}\label{eq:energy-density-local-fermionic-relation}
	\langle \chi_{c_k^{+}}\chi_{a_k^{+}} \rangle=-\langle \chi_{a_k^{+}}\chi_{c_k^{+}} \rangle = -\frac{\mathbb{E}[\varepsilon_{e_k}]}{\sin(\theta_{e_k})} + \frac{\cos(\theta_{e_k})}{\sin(\theta_{e_k})}
\end{equation} 
This allows to conclude that, when $\frak{p} \not\in z_{k}$ \eqref{eq:second-derivative-embedding-formula} indeed holds. 	
		
\textbf{Case B: $\frak{p}  \in  z_k $}. In that case one can plug \eqref{eq:derivative-fermion-close} into \eqref{eq:second-derivative-Y}. Changing the sign of the correlator derivative allows to rewrite 
\begin{align*}
	\frac{d^2 \cY(\frak{p})}{d^2\theta_k}&=\frac{m_k}{2}\Bigg[ \frac{m_k}{2}\Big( \cY(c_k^{+}) \langle \chi_{c_k^{+}}\chi_{a_k^{+}} \rangle - \cY(a_k^{+}) \Big) \langle \chi_{\frak{p}}\chi_{a_k^{+}} \rangle \\
	& \quad -\frac{m_k}{2\sin(\theta_{e_k})}\cY(c_k^{+}) \Big( \langle \chi_{\frak{p}}\chi_{a_k^{+}} \rangle \mathbb{E}[\varepsilon_{e_k}] - \langle \chi_{\frak{p}}\chi_{a_k^{+}} \varepsilon_{e_k} \rangle  \Big) \Bigg]\\
	& \quad -\frac{m_k}{2}\Bigg[ +\frac{m_k}{2}\Big( \cY(c_k^{+})- \cY(a_k^{+}) \langle \chi_{a_k^{+}}\chi_{c_k^{+}} \rangle \Big) \langle \chi_{\frak{p}}\chi_{c_k^{+}} \rangle \\
	& \quad -\frac{m_k}{2\sin(\theta_{e_k})}\cY(a_k^{+}) \Big( \langle \chi_{\frak{p}}\chi_{c_k^{+}} \rangle \mathbb{E}[\varepsilon_{e_k}] - \langle \chi_{\frak{p}}\chi_{c_k^{+}} \varepsilon_{e_k} \rangle  \Big) \Bigg].
\end{align*}
One regroups once again the $\cY(c_k^{+})$ and $\cY(a_k^{+})$ terms together and rewrite $\frac{d^2 \cY(\frak{p})}{d^2\theta_k}$ as
\begin{align*}
	\frac{m_k^2}{4}\cY(c_k^{+})\Bigg[ \langle \chi_{c_k^{+}}\chi_{a_k^{+}} \rangle  \langle \chi_{\frak{p}}\chi_{a_k^{+}} \rangle - \frac{1}{\sin(\theta_{e_k})}\Big( \langle \chi_{\frak{p}}\chi_{a_k^{+}} \rangle \mathbb{E}[\varepsilon_{e_k}] - \langle \chi_{\frak{p}}\chi_{a_k^{+}} \varepsilon_{e_k} \rangle  \Big) - \langle \chi_{\frak{p}}\chi_{c_k^{+}} \rangle \Bigg] \\
+\frac{m_k^2}{4}\cY(a_k^{+})\Bigg[ -\langle \chi_{\frak{p}}\chi_{a_k^{+}}\rangle+   \langle \chi_{a_k^{+}}\chi_{c_k^{+}} \rangle  \langle \chi_{\frak{p}}\chi_{c_k^{+}} \rangle   +  \frac{1}{\sin(\theta_{e_k})} \Big( \langle \chi_{\frak{p}}\chi_{c_k^{+}} \rangle \mathbb{E}[\varepsilon_{e_k}] - \langle \chi_{\frak{p}}\chi_{c_k^{+}} \varepsilon_{e_k} \rangle  \Big)  \Bigg].
\end{align*}
Similarly to the case where $\frak{p}\not\in z_k$, one can use the local relations for Kadanoff-Ceva correlators near their singularities. Let us work out the case $\frak{p}=a_k^{+}$ in details, the others can be treated exactly similarly. In that case one can rewrite $\frac{d^2 \cY(a_k^{+})}{d^2\theta_k}$ as
\begin{multline*}
	\frac{m_k^2}{4}\cY(c_k^{+})\Bigg[ \langle \chi_{c_k^{+}}\chi_{a_k^{+}} \rangle   - \langle \chi_{a_k^{+}}\chi_{c_k^{+}} \rangle \Bigg]+ \\
\frac{m_k^2}{4}\cY(a_k^{+})\Bigg[ -1 +   \langle \chi_{a_k^{+}}\chi_{c_k^{+}} \rangle^2   +  \frac{1}{\sin(\theta_{e_k})} \Big( \langle \chi_{a_k^{+}}\chi_{c_k^{+}} \rangle \mathbb{E}[\varepsilon_{e_k}] - \langle \chi_{d_k^{+}}\chi_{c_k^{+}} \rangle  \Big)  \Bigg],
\end{multline*}
where we used once again the identities $\langle \chi_{a_k^{+}}\chi_{a_k^{+}} \rangle = 1$, $\varepsilon_{e_k}=\chi_{d_k^{+}}\chi_{a_k^{+}}$ and the anti-symmetry of the fermions when exchanging their arguments. Using this time the local relations around $z_k$ given by $(\textrm{D}_{k}'')$ and the local energy density identity \eqref{eq:energy-density-local-fermionic-relation} ensures that
\begin{equation}
	\frac{d^2 \cY(a_k^{+})}{d^2\theta_k}= \frac{m_k^2}{2} \Bigg( \frac{\cos(\theta_{e_k})}{\sin(\theta_{e_k})} -\frac{\mathbb{E}[\varepsilon_{e_k}]}{\sin(\theta_{e_k})}  
 \Bigg)  \Bigg[ \cY(c_k^{+})\langle \chi_{a_k^{+}}\chi_{a_k^{+}} \rangle - \cY(a_k^{+})\langle \chi_{a_k^{+}}\chi_{c_k^{+}} \rangle  \Bigg].
\end{equation}	
The main application of these formulae is some explicit formulae for the second order derivative of the fermion encoding the embedding when constructed as a solution to \eqref{eq:ODE-embedding}. More precisely, assume that $\cY$ is a solution to \eqref{eq:ODE-embedding} in $\Lambda_n$, with all the masses in $\Lambda_n $ given by the collection $(m_k)_{k\in \Lambda_n}$. Then for $\frak{p}\in \Lambda_n$, one can easily see that
\begin{align}\label{eq:second-derivative-fermion-for-Ito-formula}
	\frac{d^2}{dt^2}\cY^{(t)}(\frak{p})&= \sum\limits_{e_k\in \Lambda_n} \frac{m_k^2}{2} \Bigg( \frac{\cos(\theta^{(t)}_{e_k})}{\sin(\theta^{(t)}_{e_k})} -\frac{\mathbb{E}_{S(t)}[\varepsilon_{e_k}]}{\sin(\theta^{(t)}_{e_k})}  
 \Bigg)  \Bigg[ \cY^{(t)}(c_k^{+})\langle \chi_{\frak{p}}\chi_{a_k^{+}} \rangle_{S(t)} - \cY^{(t)}(a_k^{+})\langle \chi_{\frak{p}}\chi_{c_k^{+}} \rangle_{S(t)}  \Bigg]\\
 &\quad + \sum\limits_{e_k\neq e_r\in \Lambda_n} m_k\cdot m_r \cdot g^{(\frak{p})}_{k,r}(t),
\end{align}
where $g^{(\frak{p})}_{k,r}(t)$ is given by (using the Pfaffian rules and voluntarily not written in a more-simplified manner as the involved contributions will disappear when one plugs it in the proof of Lemma \ref{lem:SDE-embedding} due to independence of Brownian motions at different edges $e_k\neq e_r$) 
\begin{align}
	g^{(\frak{p})}_{k,r}(t)&=\frac{1}{4} \langle \chi_{\frak{p}}\chi_{a_k^+} \rangle_{S(t)} \Big( \cY^{(t)}(c_r^+)\langle \chi_{c_k^+}\chi_{a_r^+} \rangle_{S(t)} - \cY^{(t)}(a_r^+)\langle \chi_{c_k^+}\chi_{c_r^+} \rangle_{S(t)} \Big)\nonumber \\
	& \quad +\cY^{(t)}(c_k^+) \cdot  \frac{1}{2\sin(\theta^{(t)}_{e_r})}\Big( \langle \chi_\frak{p} \chi_{a_k^+} \rangle_{S(t)} \mathbb{E}_{S(t)}[\varepsilon_{e_r}]-  \langle \chi_\frak{p} \chi_{a_k^+}\varepsilon_{e_r} \rangle_{S(t)}  \Big) \nonumber\\
	&\quad -\frac{1}{4} \langle \chi_{\frak{p}}\chi_{c_k^+} \rangle_{S(t)} \Big( \cY^{(t)}(c_r^+)\langle \chi_{a_k^+}\chi_{a_r^+} \rangle_{S(t)} - \cY^{(t)}(a_r^+)\langle \chi_{a_k^+}\chi_{c_r^+} \rangle_{S(t)} \Big) \nonumber\\
	& \quad -\cY^{(t)}(a_k^+) \cdot  \frac{1}{2\sin(\theta^{(t)}_{e_r})}\Big( \langle \chi_\frak{p} \chi_{c_k^+} \rangle_{S(t)} \mathbb{E}_{S(t)}[\varepsilon_{e_r}]-  \langle \chi_\frak{p} \chi_{c_k^+}\varepsilon_{e_r} \rangle_{S(t)}  \Big)\nonumber\\
	& = \quad  \frac{1}{4}\cY^{(t)}(c_r^+)\Big( \langle \chi_{\frak{p}}\chi_{a_r^+} \varepsilon_{e_k} \rangle_{S(t)} - \langle \chi_{\frak{p}}\chi_{a_r^+} \rangle_{S(t)} \mathbb{E}_{S(t)}[\varepsilon_{e_k}] \Big)\nonumber\\
	& \quad -\frac{1}{4}\cY^{(t)}(a_r^+)\Big( \langle \chi_{\frak{p}}\chi_{c_r^+} \varepsilon_{e_k} \rangle_{S(t)} - \langle \chi_{\frak{p}}\chi_{c_r^+} \rangle_{S(t)} \mathbb{E}_{S(t)}[\varepsilon_{e_k}] \Big)\nonumber\\
	&  \quad + \frac{\cY^{(t)}(c_k^+)}{2\sin(\theta^{(t)}_{e_r})}\Big( \langle \chi_\frak{p} \chi_{a_k^+} \rangle_{S(t)} \mathbb{E}_{S(t)}[\varepsilon_{e_r}]-  \langle \chi_\frak{p} \chi_{a_k^+}\varepsilon_{e_r} \rangle_{S(t)}  \Big)\nonumber\\
	& \quad - \frac{\cY^{(t)}(a_k^+)}{2\sin(\theta^{(t)}_{e_r})}\Bigg( \langle \chi_\frak{p} \chi_{c_k^+} \rangle_{S(t)} \mathbb{E}_{S(t)}[\varepsilon_{e_r}]-  \langle \chi_\frak{p} \chi_{c_k^+}\varepsilon_{e_r} \rangle_{S(t)}  \Big).\label{eq:g_{k,r}}
\end{align}

\printbibliography

\end{document}